\documentclass[10pt,a4paper,oneside,reqno]{amsart} 
\usepackage{graphicx} 
\usepackage{color,bm} %
\usepackage{amsmath,amssymb,amsthm,mathrsfs,amsfonts,dsfont,braket} 
\usepackage{listings,enumerate} 
\usepackage{mdwlist}
\usepackage{booktabs,changepage}
\usepackage{xspace,afterpage}

\usepackage[printonlyused,withpage]{acronym} 
\usepackage{rotating} 
\usepackage{mathrsfs} 
\usepackage[pdfencoding=auto,psdextra]{hyperref}
\usepackage[dvipsnames]{xcolor}
\usepackage[ocgcolorlinks]{ocgx2}
\hypersetup{breaklinks,colorlinks=true,linkcolor=MidnightBlue,citecolor=MidnightBlue,urlcolor=MidnightBlue}%
\usepackage{etoolbox}
\patchcmd{\abstract}{\scshape\abstractname}{\textbf{\abstractname}}{}{}

\usepackage{crimson}
\DeclareMathSizes{10}{9}{6}{5}

\makeatletter
\renewcommand{\tocsection}[3]{%
  \indentlabel{\@ifnotempty{#2}{\bfseries\ignorespaces#1 #2\quad}}\bfseries#3}
\renewcommand{\tocsubsection}[3]{%
  \indentlabel{\@ifnotempty{#2}{\ignorespaces#1 #2\quad}}#3}
\newcommand\@dotsep{4.5}
\def\@tocline#1#2#3#4#5#6#7{\relax
  \ifnum #1>\c@tocdepth %
  \else
    \par \addpenalty\@secpenalty\addvspace{#2}%
    \begingroup \hyphenpenalty\@M
    \@ifempty{#4}{%
      \@tempdima\csname r@tocindent\number#1\endcsname\relax
    }{%
      \@tempdima#4\relax
    }%
    \parindent\z@ \leftskip#3\relax \advance\leftskip\@tempdima\relax
    \rightskip\@pnumwidth plus1em \parfillskip-\@pnumwidth
    #5\leavevmode\hskip-\@tempdima{#6}\nobreak
    \leaders\hbox{$\m@th\mkern \@dotsep mu\hbox{}\mkern \@dotsep mu$}\hfill
    \nobreak
    \hbox to\@pnumwidth{\@tocpagenum{\ifnum#1=1\bfseries\fi#7}}\par%
    \nobreak
    \endgroup
  \fi}
\AtBeginDocument{%
\expandafter\renewcommand\csname r@tocindent0\endcsname{0pt}
}
\def\l@subsection{\@tocline{2}{0pt}{2.5pc}{5pc}{}}
\def\l@subsubsection{\@tocline{3}{0pt}{4.5pc}{2.5pc}{}}
\makeatother

\usepackage[left=1.9cm,right=1.9cm,top=2.2cm,bottom=1.7cm]{geometry}
\usepackage{dsfont,xstring}
\RequirePackage{setspace}
\usepackage{pgfplots} 
\pgfplotsset{compat=newest} 
\usepackage{ifdraft}
\usepackage{tikz}
\usepackage{etoolbox,listings,xstring,wasysym}

\tikzset{
  Gc/.style = {draw=none,circle split, inner sep=0pt,minimum size=8pt,rotate=90,path picture={\draw[pattern=#1] (0,0.07) circle (1.5pt); }},
  Gd/.style = {draw=none,circle split, inner sep=0pt,minimum size=8pt,rotate=270,path picture={\draw[pattern=#1] (0,0.07) circle (1.5pt); }},
  Gf/.style={draw=none,circle split, inner sep=1pt,minimum size=8pt,rotate=90},
  G/.style={circle,draw,minimum size=8pt,inner sep=1pt,font=\tiny},
  xx/.style={circle,fill,draw,inner sep=0pt,minimum size=3pt},
  Gend/.style={inner sep=0pt,minimum size=3pt}
}

\tikzset{circle split part fill/.style  args={#1}{%
 alias=tmp@name, 
  postaction={%
    insert path={
     \pgfextra{%
     \pgfpointdiff{\pgfpointanchor{\pgf@node@name}{center}}%
                  {\pgfpointanchor{\pgf@node@name}{east}}%
     \pgfmathsetmacro\insiderad{\pgf@x}
      \fill[white,fill opacity=0] (\pgf@node@name.base) ([xshift=-\pgflinewidth]\pgf@node@name.east) arc
                          (0:180:\insiderad-\pgflinewidth)--cycle;
      \fill[fill=white,preaction={fill, white},pattern=#1] (\pgf@node@name.base) ([xshift=\pgflinewidth]\pgf@node@name.west)  arc
                           (180:360:\insiderad-\pgflinewidth)--cycle;   
      \draw[line width=0.4pt] (\pgf@node@name.base) ([xshift=\pgflinewidth]\pgf@node@name.west)  arc
                           (180:360:\insiderad-\pgflinewidth)--cycle;                                
         }}}}}  
\makeatother  
\makeatletter
\tikzset{my loop/.style =  {to path={
  \pgfextra{}
  [looseness=6,min distance=4mm]
  \tikz@to@curve@path},font=\sffamily\small
  }}  
\makeatletter

\csedef{pat0}{}
\csedef{pat1}{north east lines}
\csedef{pat2}{crosshatch dots}
\csedef{pat3}{crosshatch}
\csedef{pat4}{grid}

\definecolor{col0}{HTML}{FFFFFF}
\definecolor{col1}{HTML}{A2B969}
\definecolor{col2}{HTML}{EBCB38}
\definecolor{col3}{HTML}{0D95BC}
\definecolor{col4}{HTML}{063951}
\definecolor{col5}{HTML}{F36F13}
\definecolor{col6}{HTML}{C13018}
\definecolor{lightgray}{HTML}{CCCCCC}

\newcommand\csum[1]{%
\sum_{\forcsvlist{\createColorCircle@item}{#1}}
}

\newcommand\createColorCircle@item[1]{
\StrDel{#1}{h}[\colNum]
\newif\ifhalf
\IfSubStr{#1}{h}{\halftrue}{\halffalse}
{\color{col\colNum}\ifhalf\circ\else\bullet\fi}
}

\newcommand\plotLambda[1]{
  \def\mgraphspecs{}
  \foreach \ll [count = \countl] in {#1} {
    \csedef{firstCol}{col0}
    \def\graphspecs{}
    \StrCount{\ll}{,}[\graphlen]
    \ifnum\graphlen>1 
    \foreach \node [count = \g] in \ll {
      \StrDel{\node}{m}[\nodenumber]
      \StrDel{\nodenumber}{.}[\nodenumber]
      \global\csedef{tempCol}{col\nodenumber}
      \global\csdef{tempPat}{\col{1}}
      \ifnum\g=1
        \global\csedef{firstCol}{\tempCol}
        \IfSubStr{\node}{.}{\global\csedef{open}{y}}{\global\csedef{open}{n}}
        \IfStrEqCase{\open}{
          {y}{\xappto\graphspecs{\countl\g[as=,draw=none] }}
          {n}{\xappto\graphspecs{\countl\g[as=,preaction={fill, white},fill=\tempCol,pattern=\csname pat\nodenumber\endcsname] }}
        }
      \else
        \IfBeginWith{\node}{m}{\global\csedef{secondArrow}{{<[sep=-3pt,length=8pt]}}}{\global\csedef{secondArrow}{}}%
        \ifnum\g=2
          \IfStrEqCase{\open}{
            {y}{\xappto\graphspecs{ --[dotted,thick] \countl\g[as=,preaction={fill, white},fill=\tempCol,pattern=\csname pat\nodenumber\endcsname] }}
            {n}{\xappto\graphspecs{ --[\firstArrow-\secondArrow] \countl\g[as=,preaction={fill, white},fill=\tempCol,pattern=\csname pat\nodenumber\endcsname] }}
          }
        \else
          \xappto\graphspecs{ --[\firstArrow-\secondArrow] \countl\g[as=,preaction={fill, white},fill=\tempCol,pattern=\csname pat\nodenumber\endcsname] }%
        \fi
      \fi
      \IfEndWith{\node}{m}{\global\csedef{firstArrow}{{>[sep=-3pt,length=8pt]}}}{\global\csedef{firstArrow}{}}
      \global\csedef{prevTempCol}{\tempCol} 
    }
    \IfStrEqCase{\open}{
      {n}{\IfBeginWith{\ll}{m}{\global\csedef{secondArrow}{{<[sep=-3pt,length=8pt]}}}{\global\csedef{secondArrow}{}}
          \xappto\graphspecs{ --[\firstArrow-\secondArrow,decorate,decoration={snake,amplitude=.3mm,segment length=.6mm}] \countl1; } }
      {y}{\xappto\graphspecs{ --[dotted,thick] \countl0[draw=none,as=] ; } }
    }
    \else
      \ifnum\graphlen=0
        \StrDel{\ll}{m}[\nodenumber]
        \csedef{tempCol}{col\nodenumber}
        \IfSubStr{\ll}{m}{\csedef{firstArrow}{{>[sep=-3pt,length=8pt]}}}{\csedef{firstArrow}{}}
        \def\graphspecs{ \countl1[as=,preaction={fill, white},fill = \tempCol,pattern=\csname pat\nodenumber\endcsname] --[my loop, decorate,decoration={snake,amplitude=.3mm,segment length=.6mm},\firstArrow-] \countl1; }
      \else
        \IfSubStr{\ll}{.}{
          \StrDel{\ll}{.}[\nodenumber]
          \StrDel{\nodenumber}{,}[\nodenumber]
          \csedef{tempCol}{col\nodenumber}
          \def\graphspecs{ \countl0[draw=none,as=,orient = left] --[dotted,thick] \countl1[as=,fill = \tempCol,pattern=\csname pat\nodenumber\endcsname] --[dotted,thick] \countl2[draw=none,as=,nudge down=10pt]; }
        }{
          \StrBefore{\ll}{,}[\nodeOne]
          \StrBehind{\ll}{,}[\nodeTwo]
          \StrDel{\nodeOne}{m}[\nodeOneNumber]
          \StrDel{\nodeTwo}{m}[\nodeTwoNumber]
          \def\tempColOne{col\nodeOneNumber}
          \def\tempColTwo{col\nodeTwoNumber}
          \IfEndWith{\nodeOne}{m}{\global\csedef{firstArrowOne}{{>[sep=-3pt,length=8pt]}}}{\global\csedef{firstArrowOne}{}}
          \IfEndWith{\nodeTwo}{m}{\global\csedef{firstArrowTwo}{{<[sep=-3pt,length=8pt]}}}{\global\csedef{firstArrowTwo}{}}
          \IfBeginWith{\nodeOne}{m}{\global\csedef{secondArrowOne}{{>[sep=-3pt,length=8pt]}}}{\global\csedef{secondArrowOne}{}}
          \IfBeginWith{\nodeTwo}{m}{\global\csedef{secondArrowTwo}{{<[sep=-3pt,length=8pt]}}}{\global\csedef{secondArrowTwo}{}}
          \def\graphspecs{ \countl1[as=,preaction={fill, white},fill = \tempColOne,pattern=\csname pat\nodeOneNumber\endcsname] --[bend right,decorate,decoration={snake,amplitude=.3mm,segment length=.6mm}, \firstArrowOne-\secondArrowTwo ] \countl2[as=,preaction={fill, white},fill = \tempColTwo,pattern=\csname pat\nodeTwoNumber\endcsname]; \countl1 --[bend left,\secondArrowOne-\firstArrowTwo] \countl2;}
        }
      \fi
    \fi
    \xappto\mgraphspecs{ \graphspecs }
  }
  \xdef\mgraphspecs{\noexpand\graph[simple necklace layout,componentwise,component packing=skyline,components go right center aligned,orient=0,nodes=G]{ \mgraphspecs }}
  \begin{tikzpicture}[baseline={([yshift=-2pt]current bounding box.center)},font=\tiny,>=Stealth, node distance = 15pt,node sep=12pt,component sep=5pt]
    \mgraphspecs;
  \end{tikzpicture}%
}
\newcommand\plotLambdas[1]{
  \begin{tikzpicture}[baseline={([yshift=-2pt]current bounding box.center)},font=\tiny,>=Stealth, node distance = 5pt,node sep=12pt,component sep=5pt]
   \graph[simple necklace layout,componentwise,component packing=skyline,components go right center aligned,orient=0,nodes=G] { #1; };
  \end{tikzpicture}
}
\newcommand\pB[1]{
  \StrDel{#1}{h}[\patNum]
  \IfSubStr{#1}{h}{
  \begin{tikzpicture}[baseline={([yshift=-2pt]current bounding box.center)},font=\tiny,>=Stealth, node distance = 0pt,node sep=1pt,component sep=1pt]
   \graph[simple necklace layout,componentwise,component packing=skyline,components go right center aligned,orient=0,nodes=G] { 1[Gf,rotate=270,circle split part fill={\csname pat\patNum\endcsname},as=,minimum size=6pt]; };
  \end{tikzpicture}
  }{
  \begin{tikzpicture}[baseline={([yshift=-2pt]current bounding box.center)},font=\tiny,>=Stealth, node distance = 0pt,node sep=1pt,component sep=1pt]
   \graph[simple necklace layout,componentwise,component packing=skyline,components go right center aligned,orient=0,nodes=G] { 1[pattern=\csname pat\patNum\endcsname,as=,minimum size=6pt]; };
  \end{tikzpicture}
  }
}
\newcommand\plotlLambda[1]{
  \def\mgraphspecs{}
  \foreach \ll [count = \countl] in {#1} {
    \csedef{firstCol}{col0}
    \csedef{openEnd}{no}
    \def\graphspecs{}
    \StrCount{\ll}{,}[\graphlen]
    \foreach \node [count = \g] in \ll {
      \StrDel{\node}{m}[\nodenumber]
      \csedef{type}{n}
      \csedef{marking}{no}
      \IfSubStr{\node}{x}{\csedef{type}{x}}{}
      \IfSubStr{\node}{c}{\csedef{type}{c}}{}
      \IfSubStr{\node}{d}{\csedef{type}{d}}{}
      \IfSubStr{\node}{.}{\csedef{type}{open}}{}
      \IfBeginWith{\node}{m}{\csedef{firstMarking}{yes}\csedef{marking}{yes}}{\csedef{firstMarking}{no}}
      \IfEndWith{\node}{m}{\csedef{secondMarking}{yes}\csedef{marking}{yes}}{\csedef{secondMarking}{no}}
      \StrDel{\nodenumber}{.}[\nodenumber]
      \StrDel{\nodenumber}{c}[\nodenumber]
      \StrDel{\nodenumber}{d}[\nodenumber]
      \StrDel{\nodenumber}{x}[\nodenumber]
      \global\csedef{tempCol}{col\nodenumber}
      \ifnum\g=1
        \global\csedef{firstCol}{\tempCol}
        \IfStrEq{\type}{open}{
          \xappto\graphspecs{ \countl0[xx,as=,grow right,draw=none,fill=white] --[dotted,thick] \countl1[fill=white,preaction={fill, white},pattern=\csname pat\nodenumber\endcsname,as=]}
        }{
        \IfSubStr{\node}{x}{
          \xappto\graphspecs{ \countl0[xx,as=,grow right,fill=lightgray,draw=none] --[color=lightgray]\countl1[Gf,circle split part fill={\csname pat\nodenumber\endcsname},as=]}}
        {
          \IfStrEq{\marking}{yes}
          {
            \xappto\graphspecs{ \countl0[xx,as=,grow right] --[-{<[fill=lightgray,color=black,sep=-3pt,length=8pt]}] \countl1[fill=white,preaction={fill, white},pattern=\csname pat\nodenumber\endcsname,as=]}
          }{
            \xappto\graphspecs{ \countl0[xx,as=,grow right] -- \countl1[fill=white,preaction={fill, white},pattern=\csname pat\nodenumber\endcsname,as=]}
          }
        }}
      \else
        \IfStrEq{\firstMarking}{yes}{\csedef{secondArrow}{{<[sep=-3pt,length=8pt]}}}{\csedef{secondArrow}{}}
        \IfStrEqCase{\type}{
          {n}{\xappto\graphspecs{ --[\firstArrow-\secondArrow] \countl\g[fill=white,preaction={fill, white},pattern=\csname pat\nodenumber\endcsname,as=] }}%
          {c}{\xappto\graphspecs{ --[\firstArrow-\secondArrow] \countl\g[Gc=\csname pat\nodenumber\endcsname,rotate=180,circle split part fill={\csname pat\nodenumber\endcsname},as=] }}%
          {d}{\xappto\graphspecs{ --[\firstArrow-\secondArrow] \countl\g[Gd=\csname pat\nodenumber\endcsname,rotate=180,circle split part fill={\csname pat\nodenumber\endcsname},as=] }}%
          {open}{\global\csedef{openEnd}{yes}}%
        }
      \fi
      \IfStrEq{\secondMarking}{yes}{
        \global\csedef{firstArrow}{{>[sep=-3pt,length=8pt]}}%
      }{
        \global\csedef{firstArrow}{}%
      }
      \global\csedef{prevTempCol}{\tempCol} 
    }
    \IfStrEqCase{\openEnd}{
      {yes}{\xappto\mgraphspecs{ \graphspecs --[dotted,thick] \countl17[Gend,as=,draw=none]; }}%
      {no}{\xappto\mgraphspecs{ \graphspecs -- \countl17[Gend,as=]; }}%
    }%
  }
  \xdef\mgraphspecs{\noexpand\graph[tree layout,componentwise,component packing=skyline,components go down left aligned,nodes=G]{ \mgraphspecs }}
  \begin{tikzpicture}[grow=right,baseline={([yshift=-2pt]current bounding box.center)},font=\tiny,>=Stealth, node distance = 5pt,node sep=12pt,component sep=5pt]
  \mgraphspecs;
  \end{tikzpicture}%
}
\newcommand\plotlLambdas[1]{
  \begin{tikzpicture}[grow=right,baseline={([yshift=-2pt]current bounding box.center)},font=\tiny,>=Stealth, node distance = 5pt,node sep=12pt,component sep=5pt]
  \graph[tree layout,componentwise,component packing=skyline,components go down left aligned,nodes=G, node distance = 15pt,node sep=12pt,component sep=5pt]{ #1 };
  \end{tikzpicture}%
}
\usepgfplotslibrary{external}
\tikzsetfigurename{tikz-figure}                   %
\newtheorem{theorem}{Theorem}
\numberwithin{equation}{section}
\numberwithin{theorem}{section}

\newtheorem{assumption}{Assumption}

\makeatletter
\def\subsubsection{\@startsection{subsubsection}{3}%
  \z@{.5\linespacing\@plus.7\linespacing}{-.5em}%
  {\normalfont\bfseries}}
\def\section{\@startsection{section}{1}%
\z@{.7\linespacing\@plus\linespacing}{.5\linespacing}%
{\normalfont\scshape\centering}}
\def\subsection{\@startsection{subsection}{2}%
\z@{.7\linespacing\@plus\linespacing}{.5\linespacing}%
{\normalfont\bfseries\centering}}
\makeatother

\newcommand{\ii}{\mathrm{i}}
\newtheorem{lemma}[theorem]{Lemma}

\newtheorem{corollary}[theorem]{Corollary}
\newtheorem{definition}[theorem]{Definition}

\newtheorem{proposition}[theorem]{Proposition}
\newtheorem{remark}[theorem]{Remark}
\newtheorem{example}[theorem]{Example}
\lstnewenvironment{pseudoc}{\lstset{frame=lines,language=C,mathescape=true}}{}
\lstnewenvironment{logs}{\lstset{frame=lines,basicstyle=\footnotesize\ttfamily,numbers=none}}{}
\lstnewenvironment{cc}{\lstset{frame=lines,language=C}}{}
\lstnewenvironment{c64}{\lstset{backgroundcolor=\color{c64},basewidth=0.65em,basicstyle=\commodoreface\color{c64light},numbers=none,framerule=10pt,rulecolor=\color{c64light},frame=tb,framexbottommargin=30pt}}{}
\lstnewenvironment{html}{\lstset{frame=lines,language=html,numbers=none}}{}
\lstnewenvironment{pseudo}{\lstset{frame=lines,mathescape=true,morekeywords={learn_string_domain, save_model}}}{}
\lstnewenvironment{pseudoctiny}{\lstset{language=C,mathescape=true,basicstyle=\tiny\sffamily}}{}
\lstnewenvironment{cctiny}{\lstset{language=C,basicstyle=\tiny\sffamily}}{}
\lstnewenvironment{pseudotiny}{\lstset{mathescape=true,basicstyle=\tiny\sffamily}}{}
\pgfplotsset{plot coordinates/math parser=false} 
\newtheorem*{genericthm*}{Assumption \ref{assumption correlations}'}
\newtheorem*{assCD}{Assumption (CD)}

\newtheoremstyle{named}{}{}{\itshape}{}{\bfseries}{.}{.5em}{\thmnote{#3}}
\theoremstyle{named}

\newlength\figureheight 
\newlength\figurewidth 
\usepackage{grffile}

\newcommand{\inD}[1]{\quad\text{in}\quad\DD_{#1}^\delta}

\newcommand{\Dout}{\DD_{\text{out}}}
\usepackage{authblk}
\usepackage{subcaption,amsaddr}

\newcommand{\1}{\mathds{1}}
\newcommand{\R}{\mathbb{R}}  %
\ifdefined\C\renewcommand{\C}{\mathbb{C}}\else\newcommand{\C}{\mathbb{C}}\fi %
\newcommand{\N}{\mathbb{N}}  %
\newcommand{\cM}{\mathcal{M}}  %
\newcommand{\cE}{\mathcal{E}}  %
\newcommand{\cG}{\mathcal{G}}  %

\newcommand{\HC}{\mathbb{H}}  %
\newcommand{\deq}{%
  \mathrel{\vbox{\offinterlineskip\ialign{%
    \hfil##\hfil\cr
    $\scriptscriptstyle d$\cr
    \noalign{\kern.1ex}
    $=$\cr
}}}}

\newcommand{\Z}{\mathbb{Z}}  %
\newcommand{\DD}{\mathbb{D}}  %

\newcommand{\wick}[1]{\mathpunct{:} #1 \mathpunct{:}}

\newcommand{\vx}{\mathbf{x}}
\newcommand{\vy}{\mathbf{y}}

\newcommand{\tuple}[1]{
(\,
\StrSubstitute{#1}{|}{\;\vert\;}[\temp]%
\temp
\,)
}
\newcommand{\vw}{\mathbf{w}}
\newcommand{\NN}{\mathcal{N}}

\newcommand{\vu}{\mathbf{u}}
\newcommand{\cI}{\mathcal{I}}

\newcommand{\vv}{\mathbf{v}}
\renewcommand{\SS}{\mathcal{S}}

\newcommand{\Q}{\mathbb{Q}}  %
\newcommand{\K}{\mathbb{K}} %
\newcommand{\cP}{\mathcal{P}} %
\newcommand{\cR}{\mathcal{R}} %
\newcommand{\cJ}{\mathcal{J}} %
\newcommand{\cC}{\mathcal{C}} %

\newcommand{\cQ}{\mathcal{Q}} %

\renewcommand{\L}{\mathcal{L}} %

\newcommand{\landauO}[2][]{\mathcal{O}_{#1}\left(#2\right)}

\newcommand*{\defeq}{\mathrel{\vcenter{\baselineskip0.5ex \lineskiplimit0pt\hbox{\scriptsize.}\hbox{\scriptsize.}}}=}
\newcommand*{\defqe}{=\mathrel{\vcenter{\baselineskip0.5ex \lineskiplimit0pt\hbox{\scriptsize.}\hbox{\scriptsize.}}}}                     
\newcommand{\norms}[2][0]{%
  \ifcase#1\relax
    \lVert #2\rVert\or         %
    \big\lVert #2\big\rVert\or   %
    \Big\lVert #2\Big\rVert\or   %
    \bigg\lVert #2\bigg\rVert\or  %
    \Bigg\lVert #2\Bigg\rVert\or     %
    \left\lVert #2\right\rVert
  \fi
} 
\makeatletter
\def\norm{\@ifnextchar[{\@normwith}{\@normwithout}}
\def\@normwith[#1]#2{\norms[#1]{#2}}
\def\@normwithout#1{\norms[5]{#1}}
\makeatother

\newcommand{\tnorms}[2][0]{%
  \ifcase#1\relax
    \lvert\kern-0.25ex\lvert\kern-0.25ex\lvert #2 \rvert\kern-0.25ex\rvert\kern-0.25ex\rvert\or %
    \big\lvert\kern-0.25ex\big\lvert\kern-0.25ex\big\lvert #2 \big\rvert\kern-0.25ex\big\rvert\kern-0.25ex\big\rvert\or   %
    \Big\lvert\kern-0.25ex\Big\lvert\kern-0.25ex\Big\lvert #2 \Big\rvert\kern-0.25ex\Big\rvert\kern-0.25ex\Big\rvert\or   %
    \bigg\lvert\kern-0.25ex\bigg\lvert\kern-0.25ex\bigg\lvert #2 \bigg\rvert\kern-0.25ex\bigg\rvert\kern-0.25ex\bigg\rvert\or   %
    \Bigg\lvert\kern-0.25ex\Bigg\lvert\kern-0.25ex\Bigg\lvert #2 \Bigg\rvert\kern-0.25ex\Bigg\rvert\kern-0.25ex\Bigg\rvert\or   %
    \left\lvert\kern-0.25ex\left\lvert\kern-0.25ex\left\lvert #2 \right\rvert\kern-0.25ex\right\rvert\kern-0.25ex\right\rvert
  \fi
} 
\makeatletter
\def\tnorm{\@ifnextchar[{\@tnormwith}{\@tnormwithout}}
\def\@tnormwith[#1]#2{\tnorms[#1]{#2}}
\def\@tnormwithout#1{\tnorms[5]{#1}}
\makeatother

\newcommand{\abss}[2][0]{%
  \ifcase#1\relax
    \lvert #2\rvert\or         %
    \big\lvert #2\big\rvert\or   %
    \Big\lvert #2\Big\rvert\or   %
    \bigg\lvert #2\bigg\rvert\or  %
    \Bigg\lvert #2\Bigg\rvert\or     %
    \left\lvert #2\right\rvert
  \fi
} 
\makeatletter
\def\abs{\@ifnextchar[{\@abswith}{\@abswithout}}
\def\@abswith[#1]#2{\abss[#1]{#2}}
\def\@abswithout#1{\abss[5]{#1}}
\makeatother

\providecommand{\braket}[1]{\left\langle#1\right\rangle} %
\providecommand{\Cov}[1]{\mathbf{Cov}\left(#1\right)} %

\DeclareMathOperator{\supp}{supp}

\DeclareMathOperator{\Val}{Val}

\DeclareMathOperator{\Spec}{Spec}

\DeclareMathOperator{\E}{\mathbf{E}}

\renewcommand{\P}{\mathbf{P}}
\DeclareMathOperator{\Tr}{Tr}

\DeclareMathOperator{\graph}{Gr} %
\DeclareMathOperator{\dist}{dist} %

\newcommand\restr[2]{{%
  \left.\kern-\nulldelimiterspace %
  #1 %
  \vphantom{\big|} %
  \right|_{#2} %
  }}
\makeatletter
\providecommand*{\diff}%
        {\@ifnextchar^{\DIfF}{\DIfF^{}}}
\def\DIfF^#1{%
        \mathop{\mathrm{\mathstrut d}}%
                \nolimits^{#1}\gobblespace
}
\def\gobblespace{%
        \futurelet\diffarg\opspace}
\def\opspace{\let\DiffSpace\! \ifx\diffarg(\let\DiffSpace\relax\else\ifx\diffarg\let\DiffSpace\relax\else\ifx\diffarg\{\let\DiffSpace\relax\fi\fi\fi\DiffSpace}

\newcommand{\nc}{\normalcolor}

\usepackage[style=phys,giveninits=true,url=false,doi=false,isbn=false,eprint=true,datamodel=mrnumber,sorting=nty,maxbibnames=99,backref=false,block=space,backend=biber]{biblatex}
\renewbibmacro{in:}{}
\DefineBibliographyStrings{english}{%
  backrefpage = {$\hookrightarrow$~},%
  backrefpages = {$\hookrightarrow$~},%
}
\bibliography{references}
\renewcommand*{\bibfont}{\scriptsize}
\DeclareFieldFormat[article]{title}{\emph{#1}} 
\DeclareFieldFormat{mrnumber}{%
  \ifhyperref
    {\href{http://www.ams.org/mathscinet-getitem?mr=#1}{\nolinkurl{MR#1}}}
    {\nolinkurl{#1}}}
\renewbibmacro*{doi+eprint+url}{%
  \iftoggle{bbx:doi}
    {\printfield{doi}}
    {}%
  \newunit\newblock
  \printfield{mrnumber}%
  \newunit\newblock
  \iftoggle{bbx:eprint}
    {\usebibmacro{eprint}}
    {}%
  \newunit\newblock
  \iftoggle{bbx:url}
    {\usebibmacro{url+urldate}}
    {}}

\title{Random Matrices with Slow Correlation Decay} 
\author{L\'aszl\'o Erd\H{o}s$^\dagger$ \and Torben Kr\"uger$^\dagger$ \and Dominik Schr\"oder$^{\dagger\ddagger}$}
\address{IST Austria, Am Campus 1, A-3400 Klosterneuburg, Austria}
\email{dschroed@ist.ac.at} 
\email{lerdos@ist.ac.at}
\email{tkrueger@ist.ac.at}
\thanks{$^\dagger$ Partially supported by ERC Advanced Grant No. 338804}
\thanks{$^\ddagger$ Partially supported by the IST Austria Excellence Scholarship}

\subjclass[2010]{60B20, 15B52} 
\keywords{Local Law, Bulk Universality, Correlated Random Matrix, Multivariate Cumulant Expansion}
\date{\today}
\begin{document}
\thispagestyle{empty}
\maketitle 
\begin{abstract}
We consider large random matrices with a general  slowly decaying  correlation among its entries.
We prove universality of the local eigenvalue statistics and  optimal local laws for the resolvent away from the spectral edges,
generalizing the recent result of \cite{1604.08188} to allow slow correlation decay and arbitrary expectation.
The main novel tool is a systematic diagrammatic control of a multivariate cumulant expansion.
\end{abstract}  
\tableofcontents

\section{Introduction}
In recent years it has been proven for increasingly general random matrix ensembles that
 their spectral measure converges to a deterministic measure up to the scale of individual eigenvalues
 as the size of the matrix tends to infinity, and that the fluctuation of the individual eigenvalues follows a universal distribution, independent of the specifics of the random matrix itself. The former is commonly called a \emph{local law}, whereas the latter is known as the \emph{Wigner-Dyson-Mehta (WDM) universality conjecture}, first 
envisioned by Wigner in the 1950's and formalized later by Dyson and Mehta in the 1960's \cite{MR0220494}. 
In fact, the conjecture extends beyond the customary random matrix ensembles in probability theory and is believed to hold for
any random operator in the delocalization regime of the Anderson metal-insulator phase transition. Given this profound
universality conjecture for general disordered quantum systems,
the ultimate goal of local spectral analysis of large random matrices is to prove the WDM conjecture for the largest
possible class of matrix ensembles. In the current paper we complete this program for  random matrices with a general, slow
correlation decay among its matrix elements.  Previous works covered only correlations with such a  fast decay  that, in a certain sense, they could be treated as a perturbation of the independent model.
 Here we present a new method that goes well beyond the perturbative regime.
 It relies on a novel multi-scale  version of the cumulant expansion and its rigorous
  Feynman diagrammatic representation that can be useful for other problems as well. To put our work in
  context, we  now explain the previous results.

In the last ten years a powerful new approach, the \emph{three-step strategy} has been developed 
to resolve WDM universality problems, see \cite{univBook} for a summary. In particular, the WDM conjecture in its classical form, stated
for Wigner matrices with a general distribution of the entries, has been proven with this strategy in \cite{MR2662426,MR2810797,MR2981427}; an independent
proof for the Hermitian symmetry class was given in \cite{MR2784665}.
Recent advances  have crystallized that 
the  only model dependent step in this strategy is the first one,  the local
law. The other two steps, the fast relaxation to equilibrium of the Dyson Brownian motion and 
the approximation by Gaussian divisible ensembles, have been formulated 
as very general ``black-box'' tools whose only input is the local law \cite{MR3687212,MR3729630,landon2016fixed}. 
Thus the proof of the WDM universality, at least for mean field ensembles, is automatically reduced to obtaining a local law.

Both local law and universality have first been established for \emph{Wigner matrices}, which are real symmetric or complex Hermitian 
$N\times N$ matrices with mean-zero entries which are independent and identically distributed (i.i.d.) up to symmetry \cite{MR2810797, MR2919197}. For Wigner matrices it has long been known that the \emph{limiting}, or \emph{self-consistent} density is the \emph{Wigner semicircle law}. In subsequent work the condition on the i.i.d.~entries has been relaxed in several steps. First, it was proven in \cite{MR2981427}, that
for \emph{generalized Wigner ensembles}, i.e., for matrices with stochastic variance profile and uniform upper and lower bound on the
variance of the matrix entries, the local law and universality also hold, with the self-consistent density still given by the semicircle law. Next, the condition of stochasticity was removed by introducing the \emph{Wigner-type} ensemble \cite{MR3719056}, in which case the self-consistent density is, generally, not semicircular any more. Finally, 
the independence condition was dropped and in \cite{1604.08188}  both a local law on the optimal local scale and  bulk universality were obtained
for matrices with correlated entries  with fast decaying general correlations. 
Special correlation structures were also considered before in \cite{MR3478311,
MR3629874} on a local scale. We also mention that there exists an extensive literature on the global law for random matrices with correlated entries \cite{MR2444540,MR3332852,MR2417889,MR1887675,MR2191967,MR2155229,MR1431189}. These results, however, either concern Gaussian random matrices or more specific correlation structures than considered in the present work. In a parallel development the zero-mean condition on  the matrix elements  has also been relaxed. First this was achieved  for the \emph{deformed Wigner ensembles} 
 that have diagonal deterministic shifts in \cite{MR3502606,MR3208886} and more recently for i.i.d.~Wigner matrices shifted by an arbitrary deterministic matrix in \cite{MR3800833}.

In this paper we prove a local law and bulk universality for random matrices with a slowly decaying correlation structure and arbitrary expectation,
generalizing both \cite{1604.08188, MR3800833}. The main point is to considerably relax the condition on the decay of correlations compared to \cite{1604.08188}: We allow for a polynomial decay of order two in a neighbourhood of size $\ll \sqrt N$ around every entry and
we only have to assume a polynomial decay of a certain finite order outside these neighbourhoods. Another  novelty is that our new concept of neighbourhoods
is completely general, it is not induced 
by the product structure of the index set labelling the matrix elements.  In particular,
the improved correlation condition also includes many other matrix models of interest, for example, general block matrix type models,
that have not been covered by \cite{1604.08188}.

Regarding strategy of proving the local law, the starting point is to find the deterministic approximation of the resolvent $G(z)=(H-z)^{-1}$ of
the random matrix $H$ with a complex spectral parameter $z$ in the upper half plane $\HC=\set{z\in\C|\Im z\ge0}$. This approximation is given as the solution $M=M(z)$ to
 the \emph{Matrix Dyson Equation (MDE)} 
 \[ 1+(z-A+\SS[M])M=0,
 \]
  where the expectation matrix $A\defeq\E H$ and the linear map $\SS[V]\defeq \E (H-A)V(H-A)$ on the space of matrices $R$
encode the first two moments of the random matrix.  The resolvent approximately  satisfies  the MDE
with an additive perturbation term \[D\defeq (H-A)G + \SS[G]G.\] The smallness of $D$ and stability of the MDE
against small perturbations imply that $G$ is indeed close to $M$.
The necessary stability properties of the MDE have already been established in  \cite{1604.08188},
so the main focus in this paper is to bound $D$ in appropriate norms that 
can then be fed into the stability analysis. Most proofs of the previous local laws loosely follow a strategy of first reducing the problem to a smaller number of relevant variables, such as the diagonal entries of $G$. Instead, correlated ensembles
require to carry out the analysis genuinely on the matrix level since $G$ is not even approximately diagonal.
This key feature distinguishes the current paper as well as  \cite{1604.08188}
from all previous works, where the Dyson equation was only a scalar equation for the trace of the resolvent or a vector equation for its diagonal elements.
 Although adding a general expectation matrix $A$ to a Wigner matrix 
 already  induces a non-diagonal resolvent, diagonalization 
of $A$ reduced the analysis to the scalar level in \cite{MR3800833}. A similar algebraic reduction is not possible for
general correlations even if they decay as fast as in \cite{1604.08188}. 
However, in \cite{1604.08188} 
every matrix quantity, such as $G$ or $M$, still had a very fast off-diagonal decay and thus 
it was sufficient to focus only on matrix elements very close to the diagonal; the rest was treated as an irrelevant error. 
For the slow correlation decay considered in this paper such direct perturbative treatment for the off-diagonal elements 
is not possible. In fact, with our new method 
we can even handle the essentially optimal integrable correlation decay  on  a scale $\sqrt{N}$ near the diagonal.

To obtain a probabilistic bound on $D$, 
essentially two approaches are available.
When $G$ is approximately diagonal and when the columns of $H$ are independent, 
one may use resolvent expansion formulas  involving minors that lead to standard linear and quadratic
large deviation bounds -- a natural idea that first arose in the works of Girko and Pastur \cite{MR1023102,MR0406251}, 
as well as in the works of Bai et.~al., e.g.~\cite{MR929083}. For correlated models the natural extension of this method 
requires a somewhat involved successive expansion of minors; this was the main technical tool in \cite{1604.08188}.
This approach is thus  restricted to very fast correlation decay  since it is 
 essentially a perturbation around nearly diagonal matrices.
The alternative method relies on the  cumulant expansion of the form
 $\E h f(h)=\sum_{k} (\kappa_{k+1} /k!) \E f^{(k)}$, where $\kappa_k$ is  the $k$-th order cumulant of the random variable $h$.
 The power of this expansion in studying resolvents of random matrices
  was first recognized in \cite{MR1411619} and it has been revived in several
  recent papers, e.g.~\cite{MR3678478,MR3405746,MR3805203}. It gives more flexibility than the minor expansion 
  on two accounts. First,  it can handle the stochastic effect of  individual matrix elements instead of 
  treating an entire column at the same time. This observation was essential in \cite{MR3800833} to handle
  deformations of Wigner matrices with an arbitrary expectation matrix. Single entry expansions,
  as opposed to expansion by entire columns, also appeared in 
  the proof of a version of the \emph{fluctuation averaging theorem} 
  \cite{MR2847916}, but in this context it did not have any major advantage 
  over the row expansions. Secondly,  a multivariate version of the  cumulant expansion is inherently 
  well suited to correlated models; it  automatically keeps track of the correlation structure
  without artificial cut-offs and strong restrictions on the off-diagonal decay.
   This is the method we use to bound $D$ in the current work to handle the slow correlation decay effectively.

After presenting our main results in Section \ref{sec main results}, in Section \ref{sec cum exp} we first  give
a multivariate cumulant expansion formula with an explicit error term that is especially well suited
for mean field random matrix models. The main ingredient is a novel \emph{pre-cumulant decoupling identity}, Lemma \ref{pre cum lemma}. 
We were not able to find these formulas in the literature; related formulas, however, have probably been known. 
They are reminiscent  to the Wick polynomials, their relationship is explained in Appendix \ref{wick poly appendix}.
 Some  consequences are collected in Section~\ref{sec:toy}
via a toy model. When applying it to our problem,  in order 
to bookkeep the numerous  terms, we develop a graphical language which allows us
 to actually compute $\E \abs{\Lambda(D)}^p$ up to a tiny error
 for arbitrary linear functionals $\Lambda$. The structure of $D$ contains an essential
cancellation: the  term $(H-A)G$ is compensated by $\SS[G]G$ that acts
as a counter term or \emph{self-energy renormalization} in the physics terminology. Our cumulant expansion automatically  
exploits this cancellation to all orders and the diagrammatic representation in Sections \ref{cancellation ids sec}--\ref{iso bound sec} conveniently visualizes this mechanism.
Section \ref{section step} contains the main novel part of this paper, in Section \ref{section stability} we combine the bounds on $D$ with the stability argument
for the MDE to prove the local law. Section \ref{section deloc rig univ} is devoted to the short proofs of bulk universality and other natural corollaries
of the local law. 

\smallskip
\noindent\emph{Acknowledgements.} T.K.~gratefully acknowledges private communications with Antti Knowles
on the preliminary version of \cite{MR3800833}. D.S.~would like to thank Nikolaos Zygouras for raising the question how our novel pre-cumulants are related to Wick polynomials.

\section{Main results}\label{sec main results}
For a Hermitian $N\times N$ random matrix $H=H^{(N)}$ we denote its resolvent by \[G(z)=G^{(N)}(z)=(H-z)^{-1},\] where the spectral parameter $z$ is assumed to be in the upper half plane $\HC$. The first two moments of $H$ determine the limiting behaviour of $G(z)$ in the large $N$ limit. More specifically, let 
\[A\defeq \E H,\qquad H\defqe A + \frac{1}{\sqrt N}W, \qquad \SS[V]\defeq \frac{1}{N}\E W V W,\]
where $\SS$ is a linear map on the space of $N\times N$ matrices and $W$ is a  random matrix with zero expectation.  Then the unique, deterministic solution $M=M(z)$ to the matrix Dyson equation (MDE)
 \begin{align}\label{matrix dyson eq}1+(z-A+\SS[M])M=0\qquad\text{under the constraint}\qquad \Im M\defeq \frac{1}{2\ii}[M-M^\ast]>0,\end{align}
 approximates the random matrix $G(z)$ increasingly well as $N$ tends to $\infty$. Here $\Im M>0$ indicates that the matrix $\Im M$ is positive definite. The properties of \eqref{matrix dyson eq} and its solution have been comprehensively studied in \cite{1604.08188}. In particular, it has been shown that 
 \begin{align}\label{M stieltjes trans} \frac{1}{N}\Tr M(z)=\int_\R \frac{1}{x-z}\diff \mu(x) \end{align}
 is the Stieltjes transform of a measure $\mu$ on $\R$, which we call the \emph{self-consistent density of states}, and whose support $\supp\mu$ we call the \emph{self-consistent spectrum}. Under an additional flatness Assumption (see Assumption \ref{assumption flatness} later) it has also been shown that $\mu$ is absolutely continuous with compactly supported H\"older continuous probability density \begin{align}\label{rho density}\diff \mu(x)=\rho(x)\diff x\qquad\text{and that}\qquad \rho(z)\defeq \frac{1}{\pi N}\Im\Tr M(z)\end{align} is the harmonic extension of $\rho\colon\R\to[0,\infty)$.
Moreover, \eqref{matrix dyson eq} is stable with respect to small additive perturbations and therefore it is sufficient to show that the error matrix $D=D(z)$ defined by 
 \begin{align} D\defeq 1+(z-A+\SS[G])G=(H-A+\SS[G])G = \frac{W}{\sqrt N}G + \SS[G]G \label{eq D def}\end{align}
is small. 

Choosing the correct norm to measure smallness of the error terms is a key technical ingredient. 
Similarly to the resolvent $G$, the error matrix $D$ is very large in the usual induced $\ell^p\to\ell^q$ matrix norms, but its quadratic form
${\braket{\vx ,D\vy}}$ is under control with very high probability for any fixed deterministic vectors $\vx, \vy$. 
Furthermore, to improve precision, we will  distinguish two different concepts of measuring the size of $D$. We will show that $D$ can be bounded in \emph{isotropic sense} as $\abs{\braket{\vx, D \vy}}\lesssim \norm{\vx}\norm{\vy}/\sqrt{N\Im z}$ for fixed deterministic vectors $\vx,\vy$ as well as in an \emph{averaged sense} as $N^{-1}\abs{\Tr BD }\lesssim \norm{B}/N\Im z$ for fixed deterministic matrices $B$. Here $\norm{\vx},\norm{\vy},\norm{B}$ denote the standard (Euclidean) vector norm $\norm{\vx}^2=\sum_a \abs{x_a}^2$ and (matrix) operator norm $\norm{B}\defeq \sup_{\norm{\vx},\norm{\vy}\le 1}\abs{\braket{\vx,B\vy}}$. The second step of the proof will be to show that because $D$ is small, and \eqref{matrix dyson eq} is stable under small additive perturbations, also $G-M$ is small in an appropriate sense. 

\subsection{Notations and conventions} An inequality with a subscript indicates that we allow for a constant in the bound depending only on the quantities in the subscript. For example, $A(N,\epsilon)\le_{\epsilon} B(N,\epsilon)$ means that there exists a constant $C=C(\epsilon)$, independent of $N$, such that $A(N,\epsilon)\le C(\epsilon ) B(N,\epsilon)$ holds for all $N$ and $\epsilon>0$. In many statements we will implicitly assume that $N$ is sufficiently large, depending on any other parameters of the model. Moreover, we will write $f\sim g$ if $f=\landauO{g}$ and $g=\landauO{f}$, if it is clear from the context in which regime we claim this comparability and how the implicit constant may depend on parameters.

An abstract index set $J$ of size $N$ labels the rows and columns of our matrix (generally one can think of $J=[N]\defeq \{1,\dots,N\}$ but there is no need for having a (partial) order or a notion of distance on $J$). The elements of $J$ will be denoted by letters $a,b,\ldots$ and $i,j,\ldots$ from the beginning of the alphabet.
We will use boldfaced letters $\vx, \vy, \vu, \vv, \ldots$ from the end of the alphabet to denote $J$-vectors with entries $\vx=(x_a)_{a\in J}$. 
We will denote the set of ordered pairs of indices by $I\defeq J\times J$ and will often call the elements of $I$ \emph{labels} to avoid confusion with other types of indices, and will denote them by Greek letters $\alpha=(a,b)\in I$. The matrix element $w_{ab}$ will thus often be denoted by $w_\alpha$. Summations of the form $\sum_a$ and $\sum_\alpha$ are always understood to sum over all $a\in J$ and $\alpha\in I$.

For indices $a,b\in J$ and vectors $\vx,\vy\in \C^J$ we shall use the notations 
\[A_{\vx \vy}\defeq \braket{\vx, A\vy}, \qquad A_{\vx a}\defeq \braket{\vx,A e_a}, \qquad A_{a \vx}\defeq \braket{e_a,A \vx},\] 
where $e_a$ is the $a$-th standard basis vector. We will frequently write $\Delta^{ab}=e_a e_b^t$ for the matrix of all zeros except a one in the $(a,b)$ entry.
The normalized trace of an $N\times N$ matrix is denoted by $\braket{A}\defeq N^{-1}\Tr A$. Sometimes we will also use the notation $\braket{z}\defeq 1+\abs{z}$ for the complex number $z$, but this should not create confusions as it will only be used for $z$. We will furthermore use the maximum norm and the normalized  Hilbert-Schmidt  norm \[\norm{A}_{\max}\defeq \max_{a,b}\abs{A_{ab}}, \quad \norm{A}_{\text{hs}}\defeq \Big[\frac{1}{N}\sum_{a,b} \abs{A_{ab}}^2\Big]^{1/2}\] for an $N\times N$ matrix $A$. 
  
\subsection{Assumptions}We now formulate our main assumptions on $W$ and $A$. 
\begin{assumption}[Bounded expectation]\label{assumption A}
There exists some constant $C$ such that $\norm{A}\le C$  for all $N$.
\end{assumption}
\begin{assumption}[Finite moments]\label{assumption high moments}
For all $q\in\N$ there exists a constant $\mu_q$ such that $\E \abs{w_{\alpha}}^q\le \mu_q$ for all $\alpha.$
\end{assumption}

Next, we formulate our conditions on the correlation decay conveniently phrased in terms of the multivariate cumulants $\kappa$ of random variables of $\set{ w_\alpha | \alpha\in I}$. In Appendix~\ref{appendix cumulants}
we recall the definition and some basic properties of multivariate cumulants. First we present a simple condition in terms of a tree type $\rho$-mixing decay of the cumulants  with respect to the standard Euclidean metric on $[N]^2$. Later, in Section \ref{sec kappa def}, we formulate weaker and more general conditions which we actually use for the proof of our results but their formulation
is quite involved, so for 
 the sake of clarity we first rather state  simpler but more restrictive assumptions.
 
  Consider $J=[N]$, $I=[N]^2$ equipped with the standard Euclidean distance modulo the Hermitian symmetry, i.e.,
for $\alpha, \beta\in I$ we set $d(\alpha, \beta) \defeq \min \{ \abs{\alpha-\beta}, \abs{\alpha^t-\beta}\}$ where $\alpha^t \defeq(b,a)$ for $\alpha = (a,b)$. This distance naturally extends to subsets of $I$, i.e., $d(A,B)=\min \set{d(\alpha,\beta)|\alpha\in A,\beta\in B}$ for any $A,B\subset I$.
\stepcounter{assumption}\stepcounter{assumption}
\begin{assCD}[\hypertarget{assumpCD}Polynomially decaying metric correlation structure]
For the $k=2$ point correlation we assume a decay of the type 
\begin{subequations}\label{metric tree decay CD}
\begin{align}\abs{\kappa(f_1(W),f_2(W))} \le \frac{C}{1+d(\supp f_1,\supp f_2)^s}\norm{f_1}_2\norm{f_2}_2, 
\label{cor decay}
\intertext{for some $s>12$ and all square integrable functions $f_1,f_2$ on $N\times N$ matrices. For $k\ge 3$ we assume a decay condition of the form}
\label{tree decay}\abs{\kappa(f_1(W),\dots,f_k(W))} \le_k \prod_{e\in E(T_{\text{min}})} \abs{\kappa(e)}, \end{align}
\end{subequations}
where $T_{\text{min}}$ is the minimal spanning tree in the complete graph on the vertices $1,\dots,k$ with respect to the edge length $d(\{i,j\})=d(\supp f_i,\supp f_j)$, i.e., the tree for which the sum of the lengths $d(e)$ is minimal, 
and $\kappa(\{i,j\})=\kappa(f_i,f_j)$. 
\end{assCD}

A correlation decay of type \eqref{tree decay} is typical for various statistical physics models, see, e.g.~\cite{MR0337229}. Besides the assumptions on the decay of correlations we also impose a \emph{flatness condition} to guarantee the stability of the Dyson equation:
\begin{assumption}[Flatness]\label{assumption flatness} 
There exist constants $0<c<C$ such that 
\[c\braket{T} \le \SS[T] \le C \braket{T} \] 
for any positive semi-definite matrix $T$.    
\end{assumption}
Flatness is a certain \emph{mean field} condition on the random matrix $W$. In particular, choosing $T$ to be
the diagonal matrix with a single nonzero entry in the $(i,i)$ element, flatness implies that the variances of the matrix elements $\E \abs{w_{ij}}^2$ are comparable for all $i,j=1,\dots,N$.

\subsection{Local law}
We now formulate our main theorem on the isotropic and averaged local laws. They 
 compare the resolvent $G$ with the (unique) solution to  
the MDE in \eqref{matrix dyson eq}  away from the spectral edges.  To specify the range of spectral parameters $z$ we define two spectral domains specified via any given parameters $\delta,\gamma>0$.
Outside of the self-consistent spectrum we will work on
\[\DD_{\text{out}}^\delta\defeq\Set{z\in\HC|\abs{z}\le N^{C_0}, \dist(z,\supp\mu)\ge N^{-\delta}}\]
for some arbitrary fixed $C_0\ge100$. Under Assumption \ref{assumption flatness}, which guarantees the existence of a density $\rho$, we consider the spectral domains 
\[ \DD_\gamma^\delta  \defeq \Set{z\in\HC | \abs{z}\le N^{C_0},\,\Im z\ge N^{-1+\gamma},\,\rho(\Re z)+\dist(\Re z,\supp \mu)\ge N^{-\delta}} \]
that will be used  away from the edges of the self-consistent spectrum.
\begin{theorem}[Local law outside of the spectrum and global law]\label{isotropic local law away}
Under Assumptions \ref{assumption A}, \ref{assumption high moments} and \hyperlink{assumpCD}{(CD)}, the following statements hold: For any $\epsilon>0$ there exists $\delta>0$ such that for all $D>0$ we have the isotropic law away from the spectrum,
\begin{subequations}
\begin{align}\label{iso outside}
\P \left(\abs{\braket{\vx,(G-M)\vy}}\le \norm{\vx}\norm{\vy}\frac{N^\epsilon}{\braket{z}^2\sqrt{N}}\quad\text{in}\quad\Dout^\delta \right) \ge 1 - C N^{-D}
\end{align}
for all deterministic vectors $\vx,\vy\in \C^N$ and we have the averaged law away from the spectrum,
\begin{align}\label{av outside}
\P \left(\abs{\braket{B(G-M)}}\le \norm{B} \frac{N^\epsilon}{\braket{z}^2N}\quad\text{in}\quad\Dout^\delta \right) \ge 1 - C N^{-D}
\end{align}
\end{subequations}
for all deterministic matrices $B\in\C^{N\times N}$. In fact, for small $\epsilon$, $\delta$ can be chosen such that $\delta=c\epsilon$ for some absolute constant $c>0$. Here $G=G(z)$, $M=M(z)$ and $C=C(D,\epsilon)$ is some constant, depending only on its arguments and the constants in Assumptions \ref{assumption A}--\hyperlink{assumpCD}{(CD)}. Moreover, instead of Assumption \hyperlink{assumpCD}{(CD)} it is sufficient to assume the more general Assumptions \ref{assumption correlations} (or \ref{assumption correlations}' for complex Hermitian matrices) and \ref{assumption neighbourhood decay}, as stated in Section \ref{sec kappa def}.
\end{theorem}
If we additionally assume flatness in the form of Assumption \ref{assumption flatness}, then we also obtain an optimal local law away from the spectral edges, especially in the bulk,
\begin{theorem}[Local law in the bulk of the spectrum]\label{isotropic local law}
Under Assumptions \ref{assumption A}, \ref{assumption high moments}, \hyperlink{assumpCD}{(CD)} and \ref{assumption flatness}, the following statements hold: For any $\gamma,\epsilon>0$ there exists $\delta>0$ such that for all $D>0$ we have the isotropic law in the bulk,
\begin{subequations}
\begin{align}\label{iso bulk}
\P \left(\abs{\braket{\vx,(G-M)\vy}}\le \norm{\vx}\norm{\vy}\frac{N^\epsilon}{\sqrt{N\Im z}}\inD{\gamma}\right) \ge 1 - C N^{-D}
\end{align}
for all deterministic vectors $\vx,\vy\in \C^N$ and we have the averaged law in the bulk,
\begin{align}\label{av bulk}
\P \left(\abs{\braket{B(G-M)}}\le \norm{B}\frac{N^\epsilon}{N\Im z}\inD{\gamma}\right) \ge 1 - C N^{-D}
\end{align}
\end{subequations}
for all deterministic matrices $B\in\C^{N\times N}$. In fact, $\delta$ can be chosen such that $\delta=c\min\{\epsilon,\gamma\}$ for some absolute constant $c>0$. Here $C=C(D,\epsilon,\gamma)$ is some constant, depending only on its arguments and the constants in Assumptions \ref{assumption A}--\ref{assumption flatness}. Moreover, as in the previous theorem, instead of Assumption \hyperlink{assumpCD}{(CD)} it is sufficient to assume the more general Assumptions \ref{assumption correlations} (or \ref{assumption correlations}' for complex Hermitian matrices) and \ref{assumption neighbourhood decay}, as stated in Section \ref{sec kappa def}.
\end{theorem}
Note that both theorems cover the regime where $z$ is far away from the spectrum; in this case the estimates in 
Theorem~\ref{isotropic local law away} are stronger and require less conditions. Theorem~\ref{isotropic local law} is really relevant
when $\Re z$ is inside the bulk of the spectrum and $\Im z$ is very small; this is why we called it local law in the bulk.
In the literature this regime is typically characterized by $\rho(\Re z)\ge \delta$ for some $\delta>0$, but in Theorem \ref{isotropic local law} it is extended to $\rho(\Re z)\ge N^{-\delta}$
for some sufficiently small $\delta>0$.

\subsection{Delocalization, rigidity and universality}
The local law is the main input for eigenvector delocalization, eigenvalue rigidity and universality, as stated below. We formulate them as corollaries
since they follow from a general theory that has been developed recently.
We explain how to adapt the general arguments to prove these corollaries in Sections \ref{sec local law proof} and \ref{section deloc rig univ}.
\begin{corollary}[No eigenvalues outside the support of the self-consistent density]\label{cor no eigenvalues outside}
Under the assumptions of Theorem \ref{isotropic local law away} there exists a $\delta>0$ such that for any $D>0$,
\[\P\left(\Spec H \not\subset (-N^{-\delta},N^{-\delta})+\supp \mu \right)\le_{D} N^{-D},\]
where $\supp \mu\subset \R$ is the support of the self-consistent density of states $\mu$.
\end{corollary}
\begin{corollary}[Bulk delocalization]\label{cor delocalization}
Under the assumptions of Theorem \ref{isotropic local law} it holds for an $\ell^2$-normalized eigenvector $\bm u$ corresponding to a bulk eigenvalue $\lambda$ of $H$ that
\[\P\left( \max_{a\in J} \abs{u_a}\ge \frac{N^\epsilon}{\sqrt N},\, H\vu=\lambda\vu,\, \rho(\lambda)\ge \delta\right)\le_{\epsilon,\delta,D} N^{-D}\]
for any $\epsilon,\delta,D>0$.
\end{corollary}
\begin{corollary}[Bulk rigidity]\label{cor rigidity}
Under the assumptions of Theorem \ref{isotropic local law} the following holds. Let $\lambda_1\le \dots\le\lambda_N$ be the ordered eigenvalues of $H$ and denote the classical position of the eigenvalue close to energy $E\in \R$ by \[k(E)\defeq \Big\lceil N\int_{-\infty}^E \rho(x)\diff x\Big\rceil,\]
where $\lceil\cdot \rceil$ denotes the ceiling function. It then holds that
\[\P\left( \sup \Set{\abs{\lambda_{k(E)} - E } | E\in\R,\, \rho(E)\ge \delta}\ge \frac{N^\epsilon}{N} \right)\le_{\epsilon,\delta,D} N^{-D}\]
for any $\epsilon,\delta,D>0$.
\end{corollary}
For proving the bulk universality we replace the lower bound from Assumption \ref{assumption flatness} by the following, stronger, assumption:
\begin{assumption}[Fullness]\label{assumption fullness}
There exists a constant $\lambda>0$ such that 
\[\E \abs{\Tr B W}^2 \ge \lambda \Tr B^2 \]
for any deterministic matrix $B$ of the same symmetry class as $H$.
\end{assumption}
Fullness is a technical condition which ensures that the covariance matrix of W is bounded from below by that of a full GUE or GOE matrix with variance $\lambda$. Note this
is the only condition that induces the difference between the complex Hermitian and real symmetric
symmetry classes in the following universality statement.

\begin{corollary}[Bulk universality]\label{cor universality}
Under the assumptions of Theorem \ref{isotropic local law} and additionally Assumption \ref{assumption fullness} the following holds: Let $k\in \N$, $\delta>0,E\in \R$ with $\rho(E)\ge \delta$ and let $\Phi\colon \R^k\to\R$ be a compactly supported smooth test function. Denote the $k$-point correlation function of the eigenvalues of $H$ by $\rho_k$ and denote the corresponding $k$-point correlation function of the GOE/GUE-point process by $\Upsilon_k$. Then there exists a positive constant $c=c(\delta,k)>0$ such that 
\begin{align*}
&\abs{\int_{\R^k} \Phi(\bm t) \left[ \frac{1}{\rho(E)}\rho_k\Big(E\mathbf 1+\frac{\bm t}{N\rho(E)}\Big)-\Upsilon_k(\bm t) \right]\diff\bm t}\le_{\Phi,\delta,k} N^{-c}, \\
&\abs{\E \Phi\Big( \big(N \rho(\lambda_{k(E)})[\lambda_{k(E)+j}-\lambda_{k(E)}]\big)_{j=1}^k \Big)-\E_{\text{GOE/GUE}}\Phi\Big( \big(N \rho_{sc}(0)[\lambda_{\lceil N/2\rceil+j}-\lambda_{\lceil N/2\rceil}]\big)_{j=1}^k \Big) }\le_{\Phi,\delta,k} N^{-c},
\end{align*}
where $\textbf 1$ is the vector of $k$ ones, $\textbf 1=(1,\dots,1)$, the expectation $\E_{\text{GOE/GUE}}$ is taken with respect to the Gaussian matrix ensemble in the same symmetry class as $H$, and $\rho_{sc}$ denotes the semicircular density.
\end{corollary}

\begin{remark}
We chose the standard Euclidean distance on $J$ in the formulation of Assumption \hyperlink{assumpCD}{(CD)} merely for convenience. In the context of \cite{1604.08188} a similar key assumption was formulated in terms of a pseudometric $\delta$ on $J$ which has sub-$P$ dimensional volume, i.e., 
\[\max_{a\in J}\abs{\Set{b\in J | \delta(a,b)\le\tau} }\le \tau^P\]
for all $\tau>1$ and some $P>0$. This pseudometric naturally extends to $I$ as a product metric modulo the symmetry, 
\[\delta_2((a,b),(c,d))\defeq\min\{ \max\{\delta(a,c),\delta(b,d)\},\max\{\delta(a,d),\delta(b,c)\} \}\]
and to any two subsets $A,B$ of $I$ as $\delta_2(A,B)\defeq\min\set{\delta_2(\alpha,\beta)|\alpha\in A,\beta\in B}$. 
All our results hold in this more general setup as well if $d$  is replaced by $\delta_2$
in Assumption \hyperlink{assumpCD}{(CD)} and we  require that $s>12P$. We do not pursue the pseudometric formulation further in the present work since the relaxed decay conditions formulated in Section \ref{sec kappa def} are  more general as they allow for further symmetries in the matrix, for which \hyperlink{assumpCD}{(CD)} is not satisfied irrespective of the pseudometric. 
A typical example for such an additional symmetry is the fourfold model (see \cite{MR3406427}).
\end{remark}

\subsection{Relaxed assumption on correlation decay}\label{sec kappa def}
\setcounter{assumption}{2}
We now state the more general conditions on the correlation structure which are actually used in the proof of Theorem \ref{isotropic local law} and its corollaries, and are implied by Assumption \hyperlink{assumpCD}{(CD)}. For the more general conditions we split the correlation into two regimes. In the short range regime we express the correlation decay as a condition on cumulants, 
while in the long range regime, beyond neighbourhoods of size $\sqrt N$, we impose a mixing condition. 

In the short range regime we assume the boundedness of certain norms on cumulants $\kappa(\alpha_1,\dots,\alpha_k)\defeq\kappa(w_{\alpha_1},\dots,w_{\alpha_k})$ of matrix entries $w_\alpha$, which are modifications of the usual $\ell^1$-summability  condition 
\[ \frac{1}{N^2} \sum_{\alpha_1,\dots,\alpha_k} \abs{\kappa(\alpha_1,\dots,\alpha_k)}<\infty.\]
\subsubsection*{Cumulant norms}\label{subsec cumulant norms}
In order to formulate the conditions on the cumulants concisely, we from now on assume that $W$ is real symmetric. We refer the reader to Appendix \ref{complex W} for the necessary modifications for the complex Hermitian case. In Appendix \ref{appendix cumulants} we will recall the equivalent analytical and combinatorial definitions of $\kappa$ for the reader's convenience (see also \cite{MR725217}). We note that $\kappa$ is invariant under any permutation of its arguments. Here we recall one central property of cumulants (which is also proved in the appendix): If $w_{\alpha_1},\dots,w_{\alpha_j}$ are independent from $w_{\alpha_{j+1}},\dots,w_{\alpha_k}$ for some $1\le j\le k-1$, then $\kappa(\alpha_1,\dots,\alpha_k)$ vanishes. Intuitively, the $k$-th order cumulant $\kappa(\alpha_1,\dots,\alpha_k)$ measures the part of the correlation of $w_{\alpha_1},\dots,w_{\alpha_k}$, which is truly of $k$-body type. For our results, cumulants of order four and higher require simple $\ell^1$-type bounds, while the second and third order cumulants are controlled in specific, somewhat stronger norms. Finiteness of these norms imply a decay of correlation in a certain combinatorial sense even without a distance on the index set $I$. The isotropic and the averaged bound on $D$ require slightly different norms, so we define two sets of norms distinguished by appropriate superscripts and we also define their sums without superscript. 

We first introduce some custom notations which keep the definition of the cumulant norms relatively compact. 
If, in place of an index $a\in J$, we write a dot ($\cdot$) in a scalar quantity then we consider the quantity as 
a vector indexed by the coordinate at the place of the dot.
 For example $\kappa(a_1\cdot,a_2b_2)$ is a $J$-vector, the $i$-th entry of which is $\kappa(a_1i,a_2b_2)$, and $\norm{\kappa(a_1\cdot,a_2b_2)}$ is its (Euclidean) vector norm.  Similarly,  $\norm{A(\ast,\ast)}$ refers to the operator norm  
of the matrix with matrix elements $ A(i,j)$. We also define a combination of these conventions, in particular 
$\norm[1]{\norm{\kappa(\vx \ast,\ast \cdot)}}$ will denote the operator norm $\norm{A}$ of the matrix $A$ with matrix 
elements $A(i,j)=\norm{\kappa(\vx i, j \cdot)} = \norm{\sum_a x_a \kappa(ai, j\cdot)}$.
 Since $\norm{A}=\norm{A^t}$ this does not introduce ambiguities with respect of the order of $i,j$.
Notice that we use dot  $(\cdot)$ for the dummy variable related to the inner norm and star ($\ast$) for the outer norm. 
\nc

For $k$-th order cumulants we set
\begin{subequations}\label{kappadef}
\begin{equation}
\tnorm{\kappa}_k\defeq\tnorm{\kappa}^{\text{av}}_k+ \tnorm{\kappa}^{\text{iso}}_k,\qquad \tnorm{\kappa}^{\text{av}/\text{iso}}=\tnorm{\kappa}^{\text{av}/\text{iso}}_{\le R} \defeq \max_{2\le k\le R} \tnorm{\kappa}_k^{\text{av}/\text{iso}},
\end{equation}
where the averaged norms are given by
\begin{equation}\label{kappa norms}
\begin{split}
\tnorm{\kappa}_2^{\text{av}} &\defeq \norm[1]{\abs{\kappa(\ast,\ast)}},\qquad \tnorm{\kappa}_k^{\text{av}} \defeq  N^{-2} \sum_{\alpha_1,\dots,\alpha_k} \abs{\kappa(\alpha_1,\dots,\alpha_k)},\qquad k\ge 4,\\
\tnorm{\kappa}_3^{\text{av}}&\defeq \norm[2]{\sum_{\alpha_1}\abs{\kappa(\alpha_1,\ast,\ast)}}+  \inf_{\kappa=\kappa_{dd}+\kappa_{dc}+\kappa_{cd}+\kappa_{cc}} \Big(\tnorm[0]{\kappa_{dd}}_{dd}+\tnorm[0]{\kappa_{dc}}_{dc}+\tnorm[0]{\kappa_{cd}}_{cd}+\tnorm[0]{\kappa_{cc}}_{cc}\Big)
\end{split}
\end{equation}
and the infimum is taken over all decompositions of $\kappa$ in four symmetric functions $\kappa_{dd},\kappa_{cd}$, etc. The letters $d$ and $c$ refer to ``direct'' and ``cross'', see Remark~\ref{rem:dc} below. The corresponding norms are given by
\begin{equation}\label{kappa ef norms}
\begin{split}
\tnorm{\kappa}_{cc} =& \tnorm{\kappa}_{dd} \defeq N^{-1}\sqrt{\sum_{b_2,a_3}\bigg(\sum_{a_2,b_3}\sum_{\alpha_1}\abs{\kappa(\alpha_1,a_2b_2,a_3b_3)} \bigg)^2},\\
\tnorm{\kappa}_{cd} \defeq& N^{-1}\sqrt{\sum_{b_3,a_1}\bigg(\sum_{a_3,b_1}\sum_{\alpha_2}\abs{\kappa(a_1b_1,\alpha_2,a_3b_3)} \bigg)^2},\quad \tnorm{\kappa}_{dc} \defeq N^{-1}\sqrt{\sum_{b_1,a_2}\bigg(\sum_{a_1,b_2}\sum_{\alpha_3}\abs{\kappa(a_1b_1,a_2b_2,\alpha_3)} \bigg)^2}\nonumber. 
\end{split}
\end{equation}
For the isotropic bound we define
\begin{equation}\label{kappa iso norms}
\begin{split}
\tnorm{\kappa}_2^{\text{iso}} &\defeq \inf_{\kappa=\kappa_d+\kappa_c} \big(\tnorm[0]{\kappa_d}_d + \tnorm[0]{\kappa_c}_c\big)\qquad \tnorm{\kappa}_d \defeq \sup_{\norm{\vx}\le 1} \norm[1]{\norm{\kappa(\vx \ast,\cdot \ast)}}\qquad \tnorm{\kappa}_c \defeq \sup_{\norm{\vx}\le 1} \norm[1]{\norm{\kappa(\vx \ast,\ast \cdot)}},\\
\tnorm{\kappa}_k^{\text{iso}} &\defeq \norm[2]{\sum_{\alpha_1,\dots,\alpha_{k-2}}\abs{\kappa(\alpha_1,\dots,\alpha_{k-2},\ast,\ast)}},\qquad k\ge 3, 
\end{split}
\end{equation}
\end{subequations}
where the inner norms in \eqref{kappa iso norms} indicate vector norms and the outer norms operator norms, and the infimum is taken over all decomposition of $\kappa$ into the sum of symmetric $\kappa_c$ and $\kappa_d$. 

\begin{remark}\label{rem:dc}
We remark that the particular form of the norms $\tnorm{\kappa}_2^{\text{iso}}$ and $\tnorm{\kappa}_3^{\text{av}}$ on $\kappa$ is chosen to conform with the Hermitian symmetry. For example, in the case of Wigner matrices we have 
\begin{align}\kappa(a_1b_1,a_2b_2)=\delta_{a_1,a_2}\delta_{b_1,b_2}+\delta_{a_1,b_2}\delta_{b_1,a_2} \defqe \kappa_d(a_1b_1,a_2b_2)+\kappa_c(a_1b_1,a_2b_2),
\label{kappa 2 wigner}\end{align} 
i.e., the cumulant naturally splits into a direct and a cross part $\kappa_d$ and $\kappa_c$. In general, the splitting $\kappa=\kappa_c+\kappa_d$ may not be unique
but for the sharpest bound we can consider the most optimal splitting; this is reflected in the infimum in the definition of $\tnorm{\kappa}_2^{\text{iso}}$. Note that in the example \eqref{kappa 2 wigner} $\tnorm[0]{\kappa_d}_d$ and $\tnorm[0]{\kappa_c}_c$ are bounded, but $\tnorm[0]{\kappa_c}_d$ would not be.
 A similar rationale stands behind the definition of $\tnorm{\kappa}_3^{\text{av}}$. 

 We also remark that only the conditions on $\tnorm{\kappa}_2^{\text{iso}}$ and $\tnorm{\kappa}_3^{\text{av}}$ use the product structure $I=J\times J$. All other decay conditions are inherently conditions on index pairs $\alpha \in I$.
\end{remark}

\begin{assumption}[$\kappa$--correlation decay]\label{assumption correlations}There exists a constant $C$ such that for all $R\in\N$ and $\epsilon>0$ 
\[ \tnorm{\kappa}_2^{\text{iso}} \le C,\qquad \tnorm{\kappa}=\tnorm{\kappa}_{\le R}\defeq \max_{2\le k\le R} \tnorm{\kappa}_k \le_{\epsilon,R} N^{\epsilon}\]
where the norms $\tnorm{\cdot}_k$ and $\tnorm{\cdot}_2^{\text{iso}}$ on $k$-th order cumulants were defined in \eqref{kappadef}.
 If the matrix $W$ is complex Hermitian we use Assumption \ref{assumption correlations}', as stated in Appendix~\ref{complex W} 
 instead of  Assumption \ref{assumption correlations}.
\end{assumption}

Furthermore, in the long range regime beyond certain neighbourhoods of size $\ll \sqrt{N}$ we assume a finite polynomial decay of correlations that is reminiscent of the standard $\rho$-mixing condition in statistical physics (see, e.g.~\cite{MR2178042} for an overview of various mixing conditions). We will need this decay in a certain iterated sense that we now formulate precisely. 

\begin{assumption}[Higher order correlation decay]\label{assumption neighbourhood decay}
There exists $\mu>0$ such that the following holds: For every $\alpha\in I$ and $q,R\in\N$ there exists a sequence of nested sets $\NN_k=\NN_k(\alpha)$ such that $\alpha\in \NN_1\subset \NN_2\subset\dots\subset \NN_R=\NN\subset I$, $\abs{\NN}\le N^{1/2-\mu}$ and
\begin{align*}
\kappa\Big(f(W_{I\setminus \bigcup_j \NN_{n_j+1}(\alpha_j)}),g_1(W_{\NN_{n_1}(\alpha_1)\setminus \bigcup_{j\not=1} \NN(\alpha_j)}),\dots,g_q(W_{\NN_{n_q}(\alpha_q)\setminus \bigcup_{j\not=q} \NN(\alpha_j)})\Big) \le_{R,q,\mu} N^{-3 q} \norm{f}_{q+1}\prod_{j=1}^q \norm{g_j}_{q+1}
\end{align*}
for any $n_1,\dots,n_q<R$, $\alpha_1,\dots,\alpha_q\in I$ and functions $f,g_1,\dots,g_q$. We will refer to these sets as ``neighbourhoods'' of $\alpha$, although we do not assume any topological structure on $I$. For any $\NN\subset I$, here $W_\NN$ denotes the set of $w_\alpha$ indexed by $\alpha\in\NN$.
\end{assumption}

\begin{remark}
For the proof of Theorem \ref{isotropic local law} we  need Assumptions \ref{assumption high moments}, \ref{assumption correlations} and \ref{assumption neighbourhood decay} only for finitely many values of $q,R$ up to some threshold, depending only on the parameters $D,\gamma,\epsilon$ in the statement and $\mu$ from Assumption \ref{assumption neighbourhood decay}. This follows from the fact that the high moment bound from Theorem \ref{theorem step} is only needed for a finite value of $p$ which relates to certain threshold on $q,R$. 
\end{remark}

\subsection{Some examples}\label{sec examples}
We end this section by providing examples of correlated matrix models satisfying Assumptions \ref{assumption correlations}--\ref{assumption neighbourhood decay}.
Our main example is the one already advertised in Assumption \hyperlink{assumpCD}{(CD)}. In Example \ref{ex polynomial decay} we check that Assumption \hyperlink{assumpCD}{(CD)} indeed implies \ref{assumption correlations}--\ref{assumption neighbourhood decay}.
\begin{example}[Polynomially decaying model]\label{ex polynomial decay} Recall the metric setting of Assumption \hyperlink{assumpCD}{(CD)}. Simple calculations show that Assumption \ref{assumption correlations} is satisfied even if we only request $s\ge 2$ in \eqref{metric tree decay CD}, independent of the chosen neighbourhood systems. As for Assumption \ref{assumption neighbourhood decay}, we define the neighbourhoods $\NN_k(\alpha)\defeq\set{\beta\in I| d(\alpha,\beta)\le k\,N^{1/4-\mu}}$ so that $d(\NN_k(\alpha),\NN_{k+1}(\alpha)^c)=N^{1/4-\mu}$. To ensure that 
\[\abs{\kappa\big(f_1(W_{\NN_n(\alpha)}),f_2(W_{\NN_{n+1}(\alpha)^c})\big)}\le \frac{\norm{f_1}_2\norm{f_2}_2}{N^3}, \]
we thus have to choose $s\ge 12/(1-4\mu)$. The tree decay structure \eqref{tree decay} then ensures that Assumption \ref{assumption neighbourhood decay} is satisfied for all $q$.
\end{example}
\begin{example}[Block matrix]
For $n,M,N\in \N$ with $nM=N$ we set $J=[N]$ and consider an $n\times n$-block matrix with identical copies of an $M\times M$ Wigner matrix in each block. We introduce an equivalence relation on $I=J \times J$ in such a way that we first identify $a \sim b\in J$ if $a = b\,(\text{mod }M)$, and then $(a,b) \sim (c,d)\in I$ if $(a,b)=(c,d)$ or $(a,b)=(d,c)$ according to the Hermitian symmetry. Then the correlation structure is such that $\kappa(\alpha_1,\dots,\alpha_k)=\landauO{1}$ if $\alpha_1, \ldots, \alpha_k$ 
all belong to the same equivalence class 
 and $\kappa(\alpha_1,\dots,\alpha_k)=0$ otherwise. Since every entry is correlated with at most $\landauO{n^2}$ other entries, 
 Assumptions \ref{assumption correlations}, \ref{assumption neighbourhood decay} are clearly satisfied  as long as $n$ is bounded.
 
 The same correlation structure is obtained if the blocks contain possibly different random matrices with independent entries
 (respecting only the overall Hermitian symmetry, but possibly without symmetry within each block), see e.g. the ensemble discussed in \cite{1706.08343}. Furthermore,  one may combine the block matrix model with a polynomially decaying model from 
 Example \ref{ex polynomial decay} to construct yet another example  for which Theorem \ref{isotropic local law} is applicable.
 In this general model the matrices in each block  should merely exhibit a polynomially decaying correlation instead of strictly 
 independent elements.
\end{example}
  \begin{example}[Correlated Gaussian matrix models]\label{example gauss}
    Since all higher order cumulants for Gaussian random variables vanish, our method allows to prove the local law (and its corollaries) for correlated Gaussian random matrix models under even weaker conditions. In fact, besides Assumptions \ref{assumption A} and \ref{assumption flatness} (or \ref{assumption fullness} for universality) we only have to assume that 
    \[ \tnorm{\kappa}_2^{\text{av}}+\tnorm{\kappa}_2^{\text{iso}} \le_{\epsilon} N^{\epsilon} \]
    for all $\epsilon>0$. In particular, this includes the polynomially decaying model from Example \ref{ex polynomial decay} for $s\ge 2$. These statements can be directly proved by following our general proof, setting all higher order cumulants to zero and using neighbourhoods $\NN(\alpha)=I$ for all $\alpha$. The details are left to the reader. 
\end{example}

\begin{example}[Fourfold symmetry]
A Wigner matrix $W$ with fourfold symmetry is a matrix of independent entries $w_\alpha$ of unit variance up to the symmetries $w_{a,b}=w_{b,a}=w_{-a,-b}=w_{-b,-a}$ for all $a,b\in \Z/N\Z$. From the explicit formula 
\[\kappa(ab,cd)=\kappa_d(ab,cd)+\kappa_c(ab,cd)\defeq(\delta_{a,c}\delta_{b,d}+\delta_{a,-c}\delta_{b,-d})+(\delta_{a,d}\delta_{b,c}+\delta_{a,-d}\delta_{b,-c}),\]
and a similar one for the third order cumulants, Assumption \ref{assumption correlations} is straightforward to verify. By choosing the neighbourhoods $\NN(\alpha)$ to contain the three other companions of $\alpha$ from the symmetry, it is obvious that also Assumption \ref{assumption neighbourhood decay} is fulfilled. Strictly speaking, the flatness condition \ref{assumption flatness} is violated by the fourfold symmetry, but as the resulting $M$ is diagonal, there is an easy replacement for the flatness. For more details on the random matrix model with a fourfold symmetry we refer the reader to \cite{MR3406427}.

A similar argument shows that Assumptions \ref{assumption correlations}--\ref{assumption neighbourhood decay} are also satisfied for other symmetries which naturally split in such a way that $w_{a,b}$ is identified with $w_{f_1(a),f_2(b)}$ and $w_{g_1(b),g_2(a)}$ for a finite collection of functions $f_i,g_i$. The appropriate replacement for the flatness condition \ref{assumption flatness}, however, has to be checked on a case-by-case basis.
\end{example}

\section{General multivariate cumulant expansion}\label{sec cum exp}
The goal of this section is the derivation of a finite-order multivariate cumulant expansion with a precise control on the approximation error. 

\subsection{Precumulants: Definition and relation to cumulants}
We begin by introducing the concept of \emph{pre-cumulants} and establishing some of their important properties. For any collection of random variables $X, Y_1, \dots,Y_m$ we define the quantities 
\begin{align*} 
K(X) \defeq& X\\
K_{t_1,\dots,t_m}(X;\bm Y)=&K_{t_1,\dots,t_m}(X;Y_1,\dots,Y_m) \defeq Y_m (\1_{t_m\le t_{m-1}}-\E ) Y_{m-1}(\1_{t_{m-1}\le t_{m-2}}-\E ) Y_{m-2}\dots Y_1(\1_{t_1\le 1}-\E) X
\end{align*}
for $m\ge 1$, that depend on real parameters $t_1,\dots,t_m\in[0,1]$. We will call them \emph{time ordered pre-cumulants}. We moreover introduce the integrated \emph{symmetrized pre-cumulants}  
\[ K(X;\bm Y) \defeq \sum_{\sigma\in S_{\abs{\bm Y}}} \iiint_{0}^1 K_{t_1,\dots,t_{\abs{\bm Y}}}(X;\sigma(\bm Y)) \diff\bm t,\]
where $S_{\abs{\bm Y}}$ is the group of permutations on a $\abs{\bm Y}$-element set and $\diff \bm t = \diff t_1 \dots \diff t_m$  indicates integration over $[0,1]^{\abs{\bm Y}}$. Note that the first variable $X$ of $K(X;\bm Y)$ plays a special role. Moreover, $K(X;\bm Y)$ is invariant under permutations of the components of the vector $\bm Y$. These pre-cumulants are -- other than the actual cumulants -- random variables, but their expectations turn out to produce the traditional cumulants, justifying their name. While they appear to be very natural objects in the study of cumulants, we are not aware whether the pre-cumulants $K$ have been previously studied, and whether the result of the following lemma is already known. 
\begin{lemma}[Pre-cumulant Lemma]\label{pre cum lemma}
Let $X$ be a random variable and let $\bm Y$, $\bm Z$ be random vectors. Then we have
\begin{subequations}
\begin{align}
\E K(X;\bm Y) &= \kappa(X,\bm Y), \label{E pre cum = cum}\\
K(X;\bm Y) &= \kappa(X,\bm Y) + X (\Pi \bm Y) - \sum_{\bm Y'\subset \bm Y} (\Pi \bm Y')\kappa(X, \bm Y\setminus \bm Y'),\label{pre cum alt def}\\
\intertext{and the \emph{pre-cumulant decoupling identity}}
K(X; \bm Y\sqcup \bm Z) &- \kappa(X,\bm Y\sqcup\bm Z) = (\Pi \bm Z) \big[ K(X; \bm Y) - \kappa(X,\bm Y)\big] - \sum_{\substack{\bm Y' \subset \bm Y\\ \bm Z' \subsetneq \bm Z}} (\Pi \bm Y')(\Pi \bm Z') \kappa(X,(\bm Y\setminus \bm Y')\sqcup (\bm Z\setminus \bm Z')),\label{two groups pre cum}
\end{align}
where $\bm Y'\subset \bm Y$ indicates that $\bm Y'$ is a sub-vector of $\bm Y$ (with $\bm Y'=\emptyset$ and $\bm Y'=\bm Y$ allowed) and $\bm Y\setminus \bm Y'$ is the vector of the remaining entries. By $\bm Z'\subsetneq \bm Z$ we denote all proper sub-vectors of $\bm Z$, i.e., not including $\bm Z$. By $\Pi\bm Z$ we mean the product of all entries of the vector $\bm Z$, while by $\bm Z\cup\bm Y$ we mean the concatenation of the two vectors $\bm Z,\bm Y$. The order of the vector is of no importance as $K(X;\bm Y)$ is symmetric with respect to the vector $\bm Y$ and $\kappa$ is overall symmetric.
\end{subequations}
\end{lemma}
We note that  \eqref{two groups pre cum} is intentionally  not symmetric in $\bm Y, \bm Z$, although 
 an analogous formula holds with $\bm Y$ and $\bm Z$ interchanged. The relation \eqref{two groups pre cum} should be interpreted as a refined version of the fact that centred precumulants factor independent random variables. Indeed, if $\bm Z$ was independent of $X,\bm Y$, then the sum in \eqref{two groups pre cum} would vanish by independence properties of the cumulant and \eqref{two groups pre cum} would simplify to 
\[K(X; \bm Y\sqcup \bm Z) - \kappa(X,\bm Y\sqcup\bm Z) = (\Pi \bm Z) \big[ K(X; \bm Y) - \kappa(X,\bm Y)\big].\]
In our applications $\bm Z$ will depend only very weakly on $X$ and $\bm Y$, hence the sum in \eqref{two groups pre cum} will be a small error term. 
\begin{proof}
By the definition of the pre-cumulants, we have for $\bm Y=(Y_1,\dots, Y_m)$
\begin{align}
K(X;\bm Y)=\sum_{\sigma\in S_{m}} \iiint_0^1 Y_{\sigma(m)} (\1_{t_m\le t_{m-1}}-\E)Y_{\sigma(m-1)}  (\1_{t_{m-1}\le t_{m-2}}-\E)\dots  (\1_{t_{2}\le t_{1}}-\E) Y_{\sigma(1)} (\1_{t_{1}\le 1}-\E) X\diff\bm t.\label{pre cum expansion}
\end{align}
Multiplying out the brackets and pulling  the  characteristic functions involving the $t$-variables out of the expectations, each term is a product of moments of $(X,\bm Y)$-monomials. We rearrange the sum according to the number of moments in the form that $K(X;\bm Y)=\sum_{b=0}^m \phi_b$, where $\phi_b$ contains exactly $b$ moments. These terms are given by 
\begin{align}\nonumber
&\phi_b=(-1)^b\sum_{1\le j_1<\dots < j_b\le m} \sum_{\sigma\in S_m}  \iiint_0^1 \1_{t_m\le \dots \le t_{j_b}}\1_{t_{j_b-1}\le \dots \le t_{j_{b-1}}}\dots \1_{t_{j_2-1}\le \dots\le t_{j_{1}}}\1_{t_{j_1-1}\le \dots\le t_{1}}\diff \bm t\\
&\qquad\times Y_{\sigma(m)}\dots Y_{\sigma(j_b)}(\E Y_{\sigma(j_b-1)}\dots Y_{\sigma(j_{b-1})})\dots(\E Y_{\sigma(j_2-1)}\dots Y_{\sigma(j_{1})}) (\E Y_{\sigma(j_1-1)}\dots Y_{\sigma(1)}X),\qquad b\ge 1 \label{b blocks}
\end{align}
and the integral in \eqref{b blocks} can be computed to give
\[ \iiint_0^1 [\,\cdots]\diff\bm t= \frac{1}{(m-j_b+1)!}\frac{1}{(j_b - j_{b-1})!}\dots \frac{1}{(j_2 - j_{1})!}\frac{1}{(j_1 - 1)!}\defqe V.\]
Here we introduced an additional variable $t_0=1$ for notational convenience and follow the convention that the last factor in \eqref{b blocks} for $j_1=1$ reads $\E X$. For $b=0$ the analogue of \eqref{b blocks} is given by
\[ \phi_0=\bigg(\sum_{\sigma\in S_m} \iiint_0^1 \1_{t_m\le \dots \le t_1}\diff\bm t\bigg) Y_{1}\dots Y_m X  = Y_{1}\dots Y_m X.\] 

Let the summation indices $1\le j_1<\dots<j_b\le m$ be fixed and fix a labelled partition of $[m]=\pi_1\sqcup\dots\sqcup \pi_{b+1}$ into subsets of sizes $\abs{\pi_1}=j_1-1$, $\abs{\pi_2}=j_2-j_1$, \dots, $\abs{\pi_b}=j_{b}-j_{b-1}$ and $\abs{\pi_{b+1}}= m-j_b+1$. Those permutations $\sigma$ in \eqref{b blocks} for which $\sigma([1,j_1-1])=\pi_1, \sigma([j_1,j_2-1])=\pi_2, \dots, \sigma([j_{b-1},j_b-1])=\pi_b$ and $\sigma([j_b,m])=\pi_{m+1}$ all produce the same term $(-1)^b V \Pi\bm Y_{\pi_{b+1}}\dots(\E \Pi\bm Y_{\pi_2})(\E X \Pi\bm Y_{\pi_1})$, where $\bm Y_\pi=(\, Y_k\,\mid\,k\in \pi\,)$. We note that $\pi_1$ plays a special role since it is explicitly allowed to be the empty set, in which the last factor is just $X$. The combinatorial factor $V$ is precisely cancelled by the number of such permutations, i.e., $1/V$. Thus \eqref{b blocks} can be rewritten as 
\begin{subequations}
\begin{align}\label{pre cum exp 2}
\phi_b &=(-1)^b \sum_{\substack{\pi_1\sqcup \dots \sqcup \pi_{b+1}=[m]\\ \abs{\pi_j}\ge 1\text{ for }j\ge 2}} \Pi \bm Y_{\pi_{b+1}} (\E \Pi \bm Y_{\pi_{b}})\dots(\E \Pi\bm Y_{\pi_2})(\E X\Pi\bm Y_{\pi_1}), \intertext{and therefore}
K(X;\bm Y) &= \sum_{b=0}^m \phi_b =\sum_{b=0}^m (-1)^b \sum_{\substack{\pi_1\sqcup \dots \sqcup \pi_{b+1}=[m]\\ \abs{\pi_j}\ge 1\text{ for }j\ge 2}} \Pi \bm Y_{\pi_{b+1}} (\E \Pi \bm Y_{\pi_{b}})\dots(\E \Pi\bm Y_{\pi_2})(\E X\Pi\bm Y_{\pi_1}),\qquad b\ge 1.\label{K formula}
\end{align}
\end{subequations}

We recognize the expectation of \eqref{pre cum exp 2} as the sum over all unlabelled partitions $\mathcal P\vdash (X,\bm Y)$ with $\abs{\mathcal P}=b+1$ blocks, under-counting by a factor of $b!$ as the first $b$ factors on the rhs.~of\eqref{pre cum exp 2} after taking the expectation are interchangeable (the last factor is special due to $X$). We can thus conclude that $\E K(X;\bm Y)$ reads 
\begin{align}\label{moebius}
\E K(X;\bm Y)=\sum_{b=0}^m (-1)^b b! \sum_{\substack{\mathcal P\vdash (X,\bm Y) \\ \abs{\mathcal P}=b+1}} \prod_{A\in\mathcal P}\E\Pi (X,\bm Y)_A =  \kappa(X,\bm Y),
\end{align}
where we used \eqref{cumulant to moment} in the ultimate step, an identity that is equivalent to the analytical definition of the cumulant, see Appendix \ref{appendix cumulants} for more details. This completes the proof of \eqref{E pre cum = cum}. Now \eqref{pre cum alt def} follows from first separating $b=0$ to produce the $X(\Pi \bm Y)$ term and then separating the $\pi_{b+1}$ summation in \eqref{K formula} so that $\bm Y_{\pi_{b+1}}$ plays the role of $\bm Y'$ for $\bm Y'\not=\emptyset$. The sum over the remaining moments is exactly the cumulant $\kappa(X,\bm Y\setminus \bm Y')$, see \eqref{moebius}. Finally, the term $\bm Y'=\emptyset$ in \eqref{pre cum alt def} cancels the first $\kappa(X,\bm Y)$ term, completing the proof of \eqref{pre cum alt def}. The identity \eqref{two groups pre cum} follows from \eqref{pre cum alt def} where $\bm Y$ plays the role of $\bm Y\sqcup \bm Z$. The $\bm Z'=\bm Z$ term is considered separately, and then the identity \eqref{pre cum alt def} is used again, this time for $X$ and $\bm Y$.
\end{proof} 

\subsection{Precumulant expansion formula}
We consider a random vector $\bm w\in \R^\cI$, indexed by an abstract set $\cI$, and a sufficiently often differentiable function $f\colon \R^\cI\to\C$. The goal is to derive an expansion for $\E w_{i_0} f(w)$ in the variables indexed by a fixed subset $\NN\subset \cI$ that contains a distinguished element $i_0\in \NN$. The expansion will be in terms of cumulants $\kappa(w_{i_1},\dots,w_{i_m})$ and expectations $\E \partial_{\bm i} f$ of derivatives $\partial_{\bm i} f\defeq \partial_{i_1}\dots\partial_{i_m} f$, where we identify $\partial_{i}=\partial_{w_i}$ and $\bm i=\{i_1,\dots,i_m\}$. To state the expansion formula compactly we first introduce some notations and definitions. We recall that a multiset is an unordered set with possible multiple appearances of the same element. For a given tuple $\bm i=(i_1,\dots,i_m)\in \NN^m$ we define the multisets 
\begin{align}\underline w_{\bm i}\defeq \{w_{i_1},\dots,w_{i_m}\} \quad\text{and the augmented multiset}\quad\underline w_{i_0\bm i}\defeq \{w_{i_0}\}\sqcup \underline w_{\bm i},\label{augmented multiset}\end{align}
where we consider $\sqcup$ as a disjoint union in the sense that $\underline w_{i_0\bm i}$ has $m+1$ elements (counting repetitions), regardless of whether $i_0=i_k$ for some $k\in [m]$. Similarly, we write $\underline w_\ast\subset \underline w$ to indicate that $\underline w_\ast$ is a sub-multiset of a multiset $\underline w$. As cumulants are invariant under permutations of their entries we will write $\kappa(\underline w)$ for multisets $\underline w$ of random variables. We will also write $\Pi \underline w\defeq \prod_{j=1}^m w_{i_j}$ for the product of elements of a multiset $\underline w=\set{w_{i_j}|j\in[m]}$. 

Equipped with Lemma \ref{pre cum lemma} we can now state and prove the version of the multivariate cumulant expansion with a remainder that is best suitable for our application. Recall from \eqref{E pre cum = cum} that $\E K(w_{i_0}; \underline w_{\bm i})=\kappa(\underline w_{i_0\bm i})$.
\begin{proposition}[Multivariate cumulant expansion]\label{cum exp prop}
Let $f\colon \R^\cI\to \C$ be $R$ times differentiable with bounded derivatives and let $\bm w\in \R^\cI$ be a random vector with finite moments up to order $R$. Fix a subset $\NN\subset\cI$ and an element $i_0\in \NN$, then it holds that
\begin{subequations}
\begin{align}\label{cum expansion statement}
\E w_{i_0} f(w) =& \sum_{m=0}^{R-1} \sum_{\bm i\in\NN^m }\left[\E\frac{\kappa(\underline w_{i_0\bm i})}{m!}\partial_{\bm i} f+\E\frac{K(w_{i_0};\underline w_{\bm i})-\kappa(\underline w_{i_0\bm i})}{m!} \partial_{\bm i} f \big\rvert_{\bm w_\NN=0}\right] +\Omega(f,i_0,\NN),\quad\text{where}\\
 \Omega(f,i_0,\NN)\defeq&\sum_{\bm i\in\NN^R } \E \iiint_0^1  K_{t_1,\dots, t_R}(w_{i_0},\dots,w_{i_{R}})\diff t_1\dots \diff t_{R-1} \int_0^1 (\partial_{\bm i} f)(t_{R} \bm w',\bm w'')\diff t_{R},\label{def Omega}
 \end{align} 
\end{subequations}
where for $m=0$ the derivative should be considered as the $0$-th derivative, i.e.~as the function itself. Here we introduced a decomposition $\bm w=(\bm w',\bm w'')$ of all random variables $\bm w=\bm w_\cI$ such that $\bm w'=\bm w_\NN=(w_i\,|\,i\in\NN )$ and $\bm w''=\bm w_{\NN^c}=(w_{i}\,|\,i\in \cI\setminus\NN )$ and we write $f(\bm w)=f(\bm w',\bm w'')$. Moreover, if $\E \abs{w_i}^{2R}\le \mu_{2R}$ for all $i\in \cI$, then 
\begin{align} \abs{\Omega(f,i_0,\NN)}\le_R \mu_{2R}^{1/2}\sum_{\bm i\in \NN^{R}}\int_0^1 \Big(\E\abs{(\partial_{\bm i} f)(t_{R} \bm w',\bm w'')}^2 \Big)^{1/2}\diff t_{R} .\label{hoelder error bound}\end{align}
\end{proposition}
\begin{proof}
For functions $f=f(\bm w)$, $g=g(\bm w)$ a Taylor expansion yields, for any $s\ge 0$,
\[ \E g(\bm w)  f(s\bm w',\bm w'') =  (\E g) (\E f(0,\bm w'')) + \Cov{ g,f(0,\bm w'')} + \sum_{i\in\NN  }\int_0^s \E g(\bm w) w_{i} (\partial_{i} f)(t \bm w',\bm w'')\diff t\]
and after another Taylor expansion to restore $f(\bm w',\bm w'')$ in the first term we find 
\begin{align}
\E g(\bm w) f( s \bm w',\bm w'') =&  (\E g) (\E f) + \Cov{ g,f(0,\bm w'')} + \sum_{i\in\NN  }\int_0^1 \E w_{i} [\1_{t\le s} g-(\E g) ] (\partial_{i} f)(t \bm w',\bm w'') \diff t.\label{cum exp general step} 
\end{align} 
Here we follow the convention that if no argument is written, then $\E g=\E g(\bm w)$. Starting with $g(\bm w)=w_{i_0}$, the last term in \eqref{cum exp general step} requires to compute $\E K_{t}(w_{i_0}; w_{i}) (\partial_{i} f)(t\bm w',\bm w'')$ with $t=t_1$, $i=i_1$. So this has the structure $\E \widetilde g \widetilde f(t\bm w',\bm w'')$ with $\widetilde g=K_{t_1}$ and $\widetilde f=\partial_{i_1} f$ and we can use \eqref{cum exp general step} again. Iterating this procedure with 
\[(g(\bm w),s,i,t)=(K_{t_1,\dots,t_{m-1}}(w_{i_0};w_{i_1}\dots,w_{i_{m-1}}),t_{m-1},i_m,t_m)\] for $m=1,\dots,R,$ we arrive at
\begin{align}\nonumber
\E w_{i_0} f &= \sum_{m=0}^{R-1} \sum_{i_1,\dots,i_m\in \NN} \bigg(\E\iiint_0^1 K_{t_1,\dots,t_m} \diff \bm t\bigg)\left(\E \partial_{\bm i} f\right) + \sum_{m=0}^{R-1} \sum_{i_1,\dots,i_l\in \NN} \E \bigg(\iiint_0^1 K_{t_1,\dots,t_m} \diff \bm t-\E\iiint_0^1 K_{t_1,\dots,t_m} \diff \bm t\bigg) (\partial_{\bm i}f)(0,\bm w'')\\
&\quad + \sum_{i_1,\dots,i_{R}\in\NN} \E \iiint_0^1  K_{t_1,\dots,t_{R}}\diff t_1\dots \diff t_{R-1} \int_0^1 (\partial_{\bm i} f)(t_{R} \bm w',\bm w'')\diff t_{R}\label{cum exp Kk},
\end{align}
where $K_{t_1,\dots,t_m}=K_{t_1,\dots,t_m}(w_{i_0},\dots,w_{i_m})$ and $\diff \bm t=\diff t_1\dots\diff t_m$. We note that \eqref{cum exp Kk} does not include the sum over permutations, but since the summation over all $i_1,\dots,i_m$ is taken we can artificially insert the permutation as in
\[\sum_{i_1,\dots,i_m}\phi (i_1,\dots,i_m)=\frac{1}{m!}\sum_{i_1,\dots,i_m}\sum_{\sigma\in S_m} \phi(i_{\sigma(1)},\dots,i_{\sigma(m)}).\] Now \eqref{cum expansion statement} follows from combining \eqref{cum exp Kk} with \eqref{E pre cum = cum}. Finally, \eqref{hoelder error bound} follows directly from a simple application of the H\"older inequality.
\end{proof}

\subsection{Toy model}\label{sec:toy}
Proposition \ref{cum exp prop} will be the main ingredient for the probabilistic part of the proofs of Theorems \ref{isotropic local law away} and \ref{isotropic local law}. For pedagocial reasons we first demonstrate the multiplicative cancellation effect of \emph{self-energy renormalization} through iterated cumulant expansion in a toy model.

Let $f$ and $\bm w$ be as in Proposition \ref{cum exp prop} and let us suppose that $\cI$ is equipped with a metric $d$. We furthermore assume that $\E \bm w=0$ and that the multivariate cumulants of $\bm w$ follow a tree-like mixing decay structure as in Example \ref{ex polynomial decay}, i.e.,
\begin{align} \kappa(f_1(\bm w),\dots,f_k(\bm w)) \lesssim \prod_{\{i,j\}\in E(T_{\text{min}})} \frac{1}{1+d(\supp f_i,\supp f_j)^s}  \label{kappa tree decay metric}\end{align}
for some $s>0$, where $T_{\text{min}}$ is the tree such that the sum of $d(\supp f_i,\supp f_j)$ along its edges $\{i,j\}\in E(T_{\text{min}})$ is minimal. Fix now a finite positive integer parameter $R$ and a large length scale $l>0$. Around every $i\in\cI$ we use the metric $d$ to define neighbourhoods $\NN(i)\defeq \Set{ j\in \cI | d(i,j)\le l R }$ and $\NN_k(i)\defeq \set{j\in I|d(i,j)\le lk}$, as in Assumption \ref{assumption neighbourhood decay}. For definiteness we furthermore assume that $\cI$ has dimension two in the sense that $\abs{\NN}\sim l^2 R^2$ as for the standard labelling of a matrix where $\cI=[N]^2$. We now assume that $f$ does not depend strongly on any single $w_i$, more specifically, for an multi-index $\bm i$ we assume \begin{align}\abs{\partial_{\bm i}f}\lesssim \abs{\NN}^{-(1+\epsilon)\abs{\bm i}},\qquad \bm i=(i_1,\dots,i_p),\qquad \abs{\bm i}=p.\label{f der decay}\end{align} This bound ensures that the size of the derivative in the Taylor expansion in the neighbourhood $\NN$ compensates for the combinatorics.

\subsubsection{Expansion of a weakly dependent function} \label{exp weakly dep fct}
For this setup we want to study the size of the expression 
\[ \E w_{i_1}\dots w_{i_p} f(\bm w) \]
where $i_1,\dots,i_p$ are in general position in the sense that their $\NN(i_k)$ neighbourhoods do not intersect. If $f$ were constant we could use the following lemma:
\begin{lemma}\label{tree decay lemma mutally separated}
Assume that $\bm w$ has a tree-like correlation decay as in \eqref{kappa tree decay metric} and assume that the random variables $g_0(\bm w),\dots,g_p(\bm w)$ have mutually $l$-separated supports, i.e., that $d(\supp g_i,\supp g_j)\gtrsim l$ for all $i\not=j$. If furthermore $\E g_k=0$ for $k=1,\dots,p$, then it holds that
\[ \abs{\E g_0\dots g_p}\lesssim l^{-s\lceil p/2\rceil}.\]
\end{lemma}
\begin{proof}
Due to a basic identity on cumulants, see \eqref{moment to cumulant}, we have that
\[\E g_0\dots g_p = \sum_{A_1\sqcup \dots \sqcup A_k = [0,p]} \kappa(g_{A_1})\dots \kappa(g_{A_k}),\]
where the sum goes over all partitions $[0,p]$ and $g_{A}=\set{g_k|k\in A}$. From \eqref{kappa tree decay metric} it follows that \[ \abs{\kappa(g_{A_k})}\lesssim l^{-s(\abs[0]{A_k}-1)} \]
and due to the assumption of zero mean $\E g_k=0$ for $k\in [p]$ we have that $\kappa(g_A)=0$ whenever $A=\{k\}$ for some $k\in [p]$. It follows that the worst case is given by pair partitions with $\abs{A_k}=2$ for all $A_k$ not containing $0$ which completes the proof.
\end{proof}
From this lemma with $g_0=1$ and $g_k=w_{i_k}$ for $k=1,\dots,p$ we conclude that for constant $f$ we have the asymptotic bound $\abs{f \E w_{i_1}\dots w_{i_p}}\lesssim l^{-s\lceil p/2\rceil}$ by the zero mean assumption $\kappa(w_i)=\E w_i=0$. We now want to argue that for weakly dependent $f$ as in \eqref{f der decay} a similar bound still holds true although $f$ depends on all variables. Note that the weak dependence renders the minimal spanning tree distance trivial and a direct application of \eqref{kappa tree decay metric} would not 
 give  any decay. For brevity, we introduce the notations 
 \[ \kappa(i,\bm j) \defeq \kappa(w_i,w_{\bm j}),\qquad K(i;\bm j) \defeq K(w_i;w_{\bm j}),\]
i.e.~we identify cumulants and precumulants as functions of indices rather than random variables. 
 We begin by expanding the first $w_{i_1}$ to obtain from 
 \eqref{cum expansion statement}
\begin{equation} \E w_{i_1}\dots w_{i_p} f = \sum_{\bm j_1}^{\NN(i_1)}\E \left[\frac{\kappa(i_1,\bm j_1)}{\abs{\bm j_1}!}+\frac{K(i_1;\bm j_1)-\E K(i_1;\bm j_1)}{\abs{\bm j_1}!}\bigg\rvert_{\bm w_{\NN(i_1)}=0}^{\rightarrow}\right]w_{i_2}\dots w_{i_p}\partial_{\bm j_1}f + \landauO{l^{-2\epsilon R}}, 
\label{toy cum exp one f}\end{equation}
where we set $\sum_{\bm j}^{\NN}\defeq \sum_{0\le m<R}\sum_{\bm j\in\NN^m}$ and the parameter $R$, the maximal order 
of the expansion, is omitted for brevity. 
The notation $\rvert_{\bm w_\NN=0}^\rightarrow$ means that in all expressions to the right, the argument $\bm w$ is set to zero in the set $\NN$, i.e.~$\bm w_\NN=0$. This effect includes expectation values and cumulants. Note that $\rvert_{\bm w_{\NN_1}=0}^\rightarrow \rvert_{\bm w_{\NN_2}=0}^\rightarrow = \rvert_{\bm w_{\NN_1\cup\NN_2}=0}^\rightarrow$, i.e.~the effects of multiple $\rvert^\rightarrow$ operators accumulate.
For example,
  \begin{equation}\label{fg}
    f(w_1, w_2)\rvert_{w_1=0}^\rightarrow \; g(w_1, w_2)\rvert_{w_2=0}^\rightarrow \;  h(w_1, w_2) = f(w_1, w_2) g(0, w_2) h(0,0).
 \end{equation}
However, the order of $\rvert_{w_1=0}^\rightarrow$ and $\rvert_{w_2=0}^\rightarrow$  matters as long as there is a nontrivial 
function in between, clearly
\[
  g(w_1, w_2)\rvert_{w_2=0}^\rightarrow  \; f(w_1, w_2)\rvert_{w_1=0}^\rightarrow  \; h(w_1, w_2) = g(w_1, w_2) f(0, w_2) h(0,0),
 \] 
which is different from \eqref{fg}.
 Finally, the error term in \eqref{toy cum exp one f} was estimated using \eqref{hoelder error bound}, and by comparing the combinatorics $\abs{\NN}^R$ of the summation to the size of the $R$-th derivative, $\abs{\partial_{i_1}\dots\partial_{i_R} f}\le \abs{\NN}^{-(1+\epsilon)R}$. We will choose $R\approx ps/4\epsilon$ large, so that the error term is negligible.

 Iterating this procedure, we find 
\begin{align} \E w_{i_1}\dots w_{i_p} f = \bigg(\prod_{k\in[p]}\sum_{\bm j_k}^{\NN(i_k)}\bigg)  \E \prod_{k\in[p]}^\rightarrow\left[\frac{\kappa(i_k,\bm j_k)}{\abs{\bm j_k}!}+\frac{K(i_k;\bm j_k)-\E K(i_k;\bm j_k)}{\abs{\bm j_k}!}\bigg\rvert_{\bm w_{\NN(i_k)}=0}^{\rightarrow}\right]\partial_{\bm j_1}\dots \partial_{\bm j_p}f + \landauO{l^{-sp/2}} \label{E wwww f}\end{align}
where $\prod_{k\in[p]}^\rightarrow a_k$ indicates that the order of the factors $a_k$ is taken to be increasing in $k$, i.e., as $a_1\dots a_p$. This is important due to the noncommutativity of the effect of the $\rvert^\rightarrow$ operation on subsequent factors. We now open the bracket in \eqref{E wwww f} and first consider the extreme case, where we take the product all the first terms from each bracket, i.e., the product of $p$ factors with $\kappa$. In this case the summation is of order $1$ as the cumulant assumption \eqref{kappa tree decay metric} implies that $\sum_{\bm j\in\cI^k}\abs{\kappa(i, \bm j)}\lesssim 1$ for any fixed $i_1$ if $s\ge 2$. Therefore the worst case is when the least total number of derivatives is taken, i.e., when $\abs{\bm j_l}=1$ for all $l$, in which case $\abs{\partial_{\bm j_1}\dots \partial_{\bm j_p}f}\lesssim \abs{\NN}^{-(1+\epsilon)p}\lesssim l^{-2p}$. Now we consider the other extreme case where all the $(K-\E K)=(K-\kappa)$ factors are multiplied. There we a priori do not see the smallness as the summation size $\abs{\NN}^{\abs{\bm j_1}+\dots+\abs{\bm j_p}}$ roughly cancels the derivative size $\abs{\NN}^{-(1+\epsilon)(\abs{\bm j_1}+\dots+\abs{\bm j_p})}$. The desired smallness thus has to come from the correlation decay \eqref{kappa tree decay metric}. We can, however not directly apply the tree-like decay structure since there does not have to be a ``security distance'' between the supports of $\bm w_{\bm j_k}$ and $f$. For those $k$ with such a security we can apply the tree-like decay immediately, and for those $k$ where there is no such security distance we instead use \eqref{two groups pre cum} to write $K-\kappa$ approximately as the product of two functions whose supports are separated by a security distance of scale $l$. Indeed, if $\bm j_k$ is not separated from $\supp f$ at least by $l$, then by the pigeon hole principle of placing less than $R$ labels into $R$ nested layers, it splits into two groups $\bm j_k^{(i)}$ and $\bm j_k^{(o)}$ of ``inside'' and ``outside'' indices such that $\dist(\bm j_k^{(i)},\bm j_k^{(o)})\gtrsim l$. Now by \eqref{two groups pre cum} we have that 
\begin{align} K(i_k; \bm j_k) - \kappa(i_k; \bm j_k) = (\Pi \bm j_k^{(o)})\big[K(i_k;\bm j_k^{(i)})-\kappa(i_k,\bm j_k^{(i)})\big] - \sum_{\bm n_k^{(o)}\subsetneq \bm j_k^{(o)}}\sum_{\bm n_k^{(i)}\subset \bm j_k^{(i)}} (\Pi \bm n_{k}^{(i)})(\Pi \bm n_{k}^{(o)})\kappa(i_k,\bm j_{k}^{(i)}\setminus \bm n_{k}^{(i)},\bm j_{k}^{(o)}\setminus \bm n_{k}^{(o)}),\label{K-kappa i j}\end{align}
where $\Pi \bm j\defeq \Pi \bm w_{\bm j}$. When multiplying \eqref{K-kappa i j} for all $k$, in the product of the second terms we (multiplicatively) collect $p$ decay factors $l^{-s}$, resulting in $l^{-sp}$. For the product of the first terms we have to estimate a term of the type $\E g_1\dots g_p \widetilde f$ with $g_k$ being zero mean random variables such that all factors have mutually $l$-separated support. Here we set $g_k\defeq K(i_k;\bm j_k^{(i)})-\kappa(i_k,\bm j_k^{(i)})$ and absorbed the $\Pi \bm j_k^{(o)}$ factors into $\widetilde f$. It follows that 
\begin{align}\label{E ggg f} \abs[0]{\E g_1\dots g_p \widetilde f}\lesssim l^{-s \lceil p/2\rceil}, \end{align}
from Lemma \ref{tree decay lemma mutally separated}. In this argument we only considered the two extreme cases when we opened the bracket in \eqref{E wwww f} and even in the product $\Pi(K-\kappa)$, after using \eqref{K-kappa i j} for each factor we only considered the two extreme cases. There are many mixed terms in both steps but they can be estimated similarly and altogether we have 
\[\abs{ \E w_{i_1}\dots w_{i_p} f}\lesssim l^{-2p}+l^{-sp/2},\]
i.e., a power law decay whose exponent is proportional to the number of factors. 

\subsubsection{Expansion of a product of weakly dependent functions and self-energy renormalization}\label{exp mult weakly dpt fct} 
Now we generalize the expansion from Section \ref{exp weakly dep fct} and consider another simple example: the iterated expansion of multipole weakly 
dependent functions. In particular, we will introduce the concept of \emph{self-energy renormalization}.

Let $f_1,\dots,f_p$ be some functions of $\bm w$ which also depend weakly on each single $w_i$ in such a way that $\abs{\partial_{\bm j} f}\lesssim \abs{\NN}^{-(1+\epsilon)\abs{\bm j}}$, and let $i_1,\dots,i_p$ be in general position as in the previous example. We want to study 
\[\E \prod_{k\in [p]} w_{i_k} f_k,\]
which, by \eqref{E wwww f} with $f$ replaced by $\prod f_k$, can be expanded to 
\[ \E \prod_{k\in [p]}w_{i_k}f_k = \prod_{k\in [p]} \bigg(\sum_{\bm j_k}^{\NN(i_k)} \sum_{(\bm j_k^l)_{l\in [p]}=\bm j_k}\bigg) \E \prod_{k\in[p]}^\rightarrow\left[\frac{\kappa(i_k,\bm j_k)}{\abs{\bm j_k}!}+\frac{K(i_k;\bm j_k)-\E K(i_k;\bm j_k)}{\abs{\bm j_k}!}\bigg\rvert_{\bm w_{\NN(i_k)}=0}^{\rightarrow}\right]\prod_{n\in [p]}(\partial_{\bm j^n} f_n) + \landauO{l^{-sp/2}}. \]
Here the second sum is the sum over all partitions $\bm j_k^1\sqcup \dots \sqcup \bm j_k^p=\bm j_k$ of the multi-index $\bm j_k$, the multi-index $\bm j^n$ is given by the disjoint union $\bm j^n=\bm j^n_1\sqcup \dots \sqcup \bm j^n_p$, and we choose $R\approx ps/4\epsilon$,
 as in the previous example (recall that $R$ is the maximal order of expansion, i.e.~$\abs{\bm j_k}\le R$). Thus $\bm j_k^n$ encodes those derivatives hitting $f_n$ which originate from the expansion according to $w_{i_k}$. By expanding the product we can rewrite this expression as 
\begin{align*} \E \prod_{k\in [p]}w_{i_k}f_k &= \sum_{L_1\sqcup L_2=[p]} \E \prod_{k\in L_1}\bigg[\sum_{\bm j_k}^{\NN(i_k)}\frac{\kappa(i_k,\bm j_k)}{\abs{\bm j_k}!}  \sum_{(\bm j_k^n)_{n\in [p]}=\bm j_k} \bigg]\\
&\qquad\times\prod_{k\in L_2}^\rightarrow\bigg[\sum_{\bm j_k}^{\NN(i_k)}\frac{K(i_k;\bm j_k)-\E K(i_k;\bm j_k)}{\abs{\bm j_k}!}\bigg\rvert_{\bm w_{\NN(i_k)}=0}^{\rightarrow} \sum_{(\bm j_k^n)_{n\in [p]}=\bm j_k} \bigg] \prod_{n\in [p]}(\partial_{\bm j^n} f_n) + \landauO{l^{-sp/2}}.  \end{align*}
It turns out that in many relevant cases, in particular after the summation over $i_1,\dots,i_k,$ the leading contribution comes from those $k\in L_1$ for which $\abs{\bm j_k}=1$ and $\abs{\bm j_k^k}=1$. To counteract these leading terms we subtract this contribution from each factor $w_{i_k}f_k$ and instead compute 
\begin{align}\nonumber&\E \prod_{k\in [p]} \big[w_{i_k} f_k-\sum_{j\in \NN(i_k)} \kappa(i_k,j) \partial_j f_k \big]= \sum_{L_1\sqcup L_2=[p]} \E \prod_{k\in L_1}\bigg[\sum_{\bm j_k}^{\NN(i_k)}\frac{\kappa(i_k,\bm j_k)}{\abs{\bm j_k}!}  \sum_{(\bm j_k^n)_{n\in [p]}=\bm j_k}\1(\abs[1]{\bm j_k^k}=0\text{ if }\abs{\bm j_k}=1) \bigg]\\
&\qquad\qquad\qquad\times\prod_{k\in L_2}^\rightarrow\bigg[\sum_{\bm j_k}^{\NN(i_k)}\frac{K(i_k;\bm j_k)-\E K(i_k;\bm j_k)}{\abs{\bm j_k}!}\bigg\rvert_{\bm w_{\NN(i_k)}=0}^{\rightarrow} \sum_{(\bm j_k^n)_{n\in [p]}=\bm j_k} \bigg] \prod_{n\in [p]}(\partial_{\bm j^n} f_n) + \landauO{l^{-sp/2}}.\label{eq self energy renorm toy}
\end{align}
We note that this substraction or \emph{self-energy renormalization} does not affect the power counting bound of $l^{-2p}+l^{-sp/2}$ because it does not change the order of the terms but only excludes certain allocations of derivatives. However, beyond power counting, this exclusion can still reduce the effective size of the term considerably, see Section \ref{section step} where $f$ is the resolvent of a random matrix.

\section{Bound on the error matrix \texorpdfstring{$D$}{D} through a multivariate cumulant expansion}\label{section step} 
In this section we prove an isotropic and averaged bound on the error matrix $D$ defined in \eqref{eq D def}, in the form of high-moment estimates using the multivariate cumulant expansion. To formalize the bounds, we define the high-moment norms for random variables $X$ and random matrices $A$ by
 \[\norm{X}_p\defeq (\E \abs{X}^p)^{1/p},\quad \norm{A}_{p}\defeq \sup_{\norm{\vx},\norm{\vy}\le 1} \norm{\braket{\vx,A\vy}}_p=\Big[\sup_{\norm{\vx},\norm{\vy}\leq 1}\E\abs{\braket{\vx,A\vy}}^p\Big]^{1/p},\]
 where the supremum is taken over deterministic vectors $\vx,\vy$. 
\begin{theorem}[Bound on the Error]\label{theorem step}
Under Assumptions \ref{assumption A}, \ref{assumption high moments} and \ref{assumption neighbourhood decay}, there exist a constant $C_\ast$ such that for any $p\ge 1,\epsilon>0$, $z$ with $\Im z\ge N^{-1}$, $B\in \C^{N\times N}$ and $\vx,\vy\in \C^N$ it holds that 
\begin{subequations}
\begin{align}\label{bootstrapping step}
\frac{\norm{\braket{\vx,D\vy}}_{p}}{\norm{\vx}\norm{\vy}} &\le_{\epsilon,p} (1+\tnorm{\SS}+\tnorm{\kappa}_{\le R}^{\text{iso}})N^{\epsilon}\sqrt{\frac{\norm{\Im G}_{q}}{N\Im z}} \Big(1+\norm{G}_{q}\Big)^{\frac{C_\ast}{\mu}} \bigg(1+ \frac{\norm{G}_{q}}{N^{\mu}}\bigg)^{\frac{C_\ast p}{\mu}} \\
\frac{\norm{\braket{BD}}_{p}}{\norm{B}} &\le_{\epsilon,p} (1+\tnorm{\SS}+\tnorm{\kappa}_{\le R}^{\text{av}}) N^{\epsilon} 
 \braket{z}  \frac{\norm{\Im G}_{q}}{N\Im z} \Big(1+\norm{G}_{q}\Big)^{\frac{C_\ast}{\mu}} \bigg(1+ \frac{\norm{G}_{q}}{N^{\mu}}\bigg)^{\frac{C_\ast p}{\mu}},
\label{av bound D eq}
\end{align}
\end{subequations}
where $q=C_\ast p^4/(\mu^2\epsilon)$, $R=4p/\mu$, and for convenience we separately defined 
\begin{align}\label{tnorm SS}
  \tnorm{\SS}\defeq \tnorm{\kappa}_2^{\text{iso}}.
 \end{align}  
\end{theorem}
\begin{remark}\label{S decomp remark}
We remark that the size of $\SS$ can be effectively controlled by $\tnorm{\kappa}_{2}^{\text{iso}}$, justifying the definition of $\tnorm{\SS}$. To see this we note that due to 
\begin{align}
\SS[V]=\frac{1}{N} \sum_{\alpha_1,\alpha_2} \kappa(\alpha_1,\alpha_2)\Delta^{\alpha_1}V\Delta^{\alpha_2}\label{SS kappa}
\end{align}
an arbitrary partition of $\kappa=\kappa_c+\kappa_d$ naturally induces a partition $\SS=\SS_c+\SS_d$. Furthermore, it is easy to see that $\norm{\SS_c[V]T}_p\le \tnorm{\kappa_c}_c \norm{V}_{2p}\norm{T}_{2p}$ and $\norm{\SS_d[V]T}_p\le \tnorm{\kappa_d}_d \norm{V}_{2p}\norm{T}_{2p}$, c.f. Lemma \ref{S[R]T bound}, thus 
\[\norm{\SS[V]T}_p \le \tnorm{\kappa}^{\text{iso}}_2 \norm{V}_{2p}\norm{T}_{2p}.\]
\end{remark}
Here we recall that the double-index $\alpha$ stands for a pair $\alpha=(a,b)$ of single indices, and that the matrix $\Delta^{\alpha}$ is a matrix of $0$'s except for a $1$ in the $(a,b)$-entry. 
\begin{remark} \label{Gpower remark} 
We point out an additional feature of the estimates \eqref{bootstrapping step}--\eqref{av bound D eq}:  they not only provide the 
optimal power of  $\norm{ \Im G }_q/(N\Im z)$, but  the power of $\norm{ G}_q$, without an extra smallness factor $N^{-\mu}$, is independent of $p$.
This will be essential  in the second part of the proof of the local law, see \eqref{Dbound} later.
\end{remark}

The main tool for proving Theorem \ref{theorem step} is the multivariate cumulant expansion from Proposition \ref{cum exp prop}. To connect to the toy model considered in Section \ref{sec:toy}, we note that the \emph{self-energy renormalization} of $N^{-1/2}WG$ is $-\SS[G]G$, up to an irrelevant contribution from indices $j\not\in\NN(i_k)$ in \eqref{eq self energy renorm toy}. In this sense the error term $D=N^{-1/2}WG+\SS[G]G$ is the difference of $N^{-1/2}WG$ and its self-energy renormalization. As already noted in the context of the toy model we recall that this substraction does not change the power counting of the resulting terms. It does, however, exclude certain allocations of derivatives which in the case of $N^{-1/2}WG$ means that the main contributions coming from the diagonal elements of the form $G_{aa}$ are absent.
Off-diagonal elements $G_{ab}$ are smaller on average, in fact the main gain comes from the key formula about resolvents of Hermitian matrices
\[\sum_b \abs{G_{ab}}^2 =\frac{\Im G_{bb}}{\eta},\]
where $\eta=\Im z$. This identity follows directly from the spectral theorem. In the physics literature it is often called \emph{Ward identity} and we will refer to it with this name. Notice that a sum of order $N$ is reduced to a $1/\eta$ factor, so the Ward identity effectively gains a factor of $1/(N\eta)$ over the naive power counting. However, this effect is available only if off-diagonal elements of the resolvent are summed up, the same reduction would not  take place in the sum  $\sum_a \abs{G_{aa}}^2$ which remains of order $N$. So the precise index structure is important. The next calculation shows this effect in the simplest case.
\subsubsection*{Exemplary gain through self-energy renormalization} We now give a short calculation to 
demonstrate the role  of self-energy renormalization term   $\SS[G]G$ while computing $\E \braket{D}^2$. 
 Notice that 
\begin{equation}\label{ED=0}\braket{ D } = \frac{1}{N} \sum_{a} \big[  \sum_b \frac{w_{ab}}{\sqrt N}G_{ba} + (\SS[G] G)_{aa}   \big] = \frac{1}{N}\sum_{a,b}\big[   \frac{w_{ab}}{\sqrt N}G_{ba} -\sum_{c,d} \frac{\kappa(ab, cd)}{N} \partial_{cd} G_{ba} \big] \end{equation}
is the sum of terms  of the form $w_{i} f$  plus their self-energy renormalization $ -N^{-1}\sum_{c,d} \kappa(ab, cd) \partial_{cd} G_{ba} $
where  $i =(a,b)$ and $f = G_{ab}$. We note that \eqref{ED=0} is the direct analogue of the self-energy renormalization in the toy-model discussed in Section \ref{sec:toy}, see \eqref{eq self energy renorm toy}. In \eqref{ED=0} we expanded $\SS[V]=\sum_{\alpha,\beta} N^{-1}\kappa(\alpha,\beta)\Delta^{\alpha} V\Delta^\beta$ and used the fact that the resolvent derivative reads $\Delta_{\alpha} G=-G\Delta^{\alpha}G$. Thus one should think of $\SS[G]G$ as being the matrix self-energy renormalization of $N^{-1/2}WG$. To present this example in the simplest form, we assume that $W$ is a Gaussian random matrix which automatically makes all higher order cumulants vanish. We find
\[\E \braket{D}^2 = N^{-1} \sum_{\alpha_1,\beta_1}\kappa(\alpha_1,\beta_1)\E \braket{\Delta^{\alpha_1} G} \braket{\Delta^{\beta_1} G} + N^{-2}\sum_{\alpha_1,\beta_1} \kappa(\alpha_1,\beta_1)\sum_{\alpha_2,\beta_2} \kappa(\alpha_2,\beta_2) \E  \braket{\Delta^{\alpha_1} G\Delta^{\beta_2} G} \braket{\Delta^{\alpha_2} G\Delta^{\beta_1}G},\]
the first term of which can be further bounded by
\[ N^{-1} \sum_{\alpha_1,\beta_1}\abs{\kappa(\alpha_1,\beta_1) \braket{\Delta^{\alpha_1} G} \braket{\Delta^{\beta_1} G}} \le \frac{\tnorm{\kappa}_2^{\text{av}}}{N} \sum_\alpha \abs{\braket{\Delta^\alpha G}}^2 = \frac{\tnorm{\kappa}_2^{\text{av}}}{N^3}\sum_{a,b} \abs{G_{ba}}^2 = \frac{\tnorm{\kappa}_2^{\text{av}}}{N^2}\frac{\braket{\Im G}}{\eta}.\]
For the second term we instead compute 
\begin{align*}
\sum_{\alpha_1,\beta_1} \sum_{\alpha_2,\beta_2}\abs{ \frac{\kappa(\alpha_1,\beta_1)\kappa(\alpha_2,\beta_2)}{N^2} \braket{\Delta^{\alpha_1} G\Delta^{\beta_2} G} \braket{\Delta^{\alpha_2} G\Delta^{\beta_1}G}}&\le \frac{(\tnorm{\kappa}_2^{\text{av}})^2}{N^2}  \sum_{\alpha_1,\alpha_2} \abs[0]{\braket{\Delta^{\alpha_2} G\Delta^{\alpha_1}G}}^2 \\
& = \frac{(\tnorm{\kappa}_2^{\text{av}})^2}{N^4} \sum_{a_1,b_1,a_2,b_2} \abs{G_{b_2a_1}}^2\abs{G_{b_1a_2}}^2=(\tnorm{\kappa}_2^{\text{av}})^2\frac{\braket{\Im G}^2}{(N\eta)^2}
\end{align*}
and we conclude that 
\[ \E\abs{\braket{D}}^2 \le \frac{1}{N^2} \E\left[ \frac{\tnorm{\kappa}_2^{\text{av}}\braket{\Im G}}{\eta}+\left(\frac{\tnorm{\kappa}_2^{\text{av}}\braket{\Im G}}{\eta}\right)^2\right],\]
which is small if $\eta\gg 1/N$. Without self-energy renormalization, however, i.e., for $\E \braket{N^{-1/2} W G}^2$ we, for example, also encounter a term of the type
\[  N^{-2}\sum_{\alpha_1,\beta_1} \kappa(\alpha_1,\beta_1)\sum_{\alpha_2,\beta_2} \kappa(\alpha_2,\beta_2) \E  \braket{\Delta^{\alpha_1} G\Delta^{\beta_1} G} \braket{\Delta^{\alpha_2} G\Delta^{\beta_2}G},\]
which is of order $1$ because it lacks the gain from the Ward identity.

\subsection{Computation of high moments of \texorpdfstring{$D$}{D} through cancellation identities} \label{cancellation ids sec}
Before going into the proof of Theorem \ref{theorem step}, we sketch the strategy. 
For simplicity, we first present the proof for the case of bounded spectral parameter  $\braket{z}\le C$
and we will comment on the trivial modification for the general case at the end of Section~\ref{section step}.
We will also temporarily assume that $\| H\| \le C$ with some large constant  $C$. 
 For arbitrary linear (or conjugate linear in the sense that $\Lambda(\lambda\cdot)=\overline{\lambda}\Lambda(\cdot)$ for $\lambda\in\C$) functionals $\Lambda^{(1)},\dots,\Lambda^{(k)}$ we derive an explicit expansion for
\begin{align} \label{D product}\E \Lambda^{(1)}(D)\dots\Lambda^{(k)}(D) \end{align}
in terms of joint cumulants $\kappa(\alpha_1,\dots,\alpha_k)$ of the entries of $W$ and expectations of products of factors of the form \begin{align*}\Lambda(\Delta^{\alpha_1} G \Delta^{\alpha_2} G\dots G\Delta^{\alpha_k}G).\end{align*} 
In other words, we express \eqref{D product} solely in terms of matrix elements of $G$, which allows for a very systematic estimate. For the main part of the expansion we will then specialize to $\Lambda^{(k)}(D)=\braket{BD}$, $\Lambda^{(k)}(D)=\braket{\vx, D\vy}$ or their complex conjugates, and develop a graphical representation of the expansion. In this framework both the averaged and the isotropic bound on $D$ reduce to a sophisticated power counting argument which -- with the help of Ward estimates -- directly gives the desired size of the averaged and isotropic error.  

Equipped with the cumulant expansion from Proposition \ref{cum exp prop}, we now aim at expressing $\E \Lambda^{(1)}(D)\dots \Lambda^{(p)}(D)$ for linear and conjugate linear functions $\Lambda^{(j)}$, purely in terms of the expectation of products of $G$'s in the form 
 \begin{equation}\label{Lambda_def}\Lambda_{\alpha_1,\dots,\alpha_k}\defeq-(-1)^kN^{-k/2} \begin{cases}
 \Lambda(\Delta^{\alpha_{1}} G\dots \Delta^{\alpha_{k}} G) &\text{if $\Lambda$ is linear}\\
 \Lambda(\Delta^{\alpha_{1}^t} G\dots \Delta^{\alpha_{k}^t} G)&\text{if $\Lambda$ is conjugate linear}
 \end{cases}  \end{equation}
 for double indices $\alpha_1,\dots,\alpha_k\in I=J\times J$, where we recall that for $\alpha=(a,b)$ the transpose $\alpha^t$ denotes $\alpha^t=(b,a)$. The sign choice will make the subsequent expansion sign-free. The reason for the $N^{-k/2}$ pre-factor is that the $\Lambda_{\alpha_1,\dots,\alpha_k}$ terms appear through $k$ derivatives of $G$'s each of which carries a $N^{-1/2}$ from the scaling $H=A+N^{-1/2}W$. Since the derivatives of $G$ naturally come with many permutations from the Leibniz rule, we will also use the notations
 \begin{align}\Lambda_{\{\alpha_1,\dots,\alpha_m\}}\defeq\sum_{\sigma\in S_m}\Lambda_{\alpha_{\sigma(1)},\dots,\alpha_{\sigma(m)}}, \quad \Lambda_{\alpha,\{\alpha_1,\dots,\alpha_m\}}\defeq\sum_{\sigma\in S_m}\Lambda_{\alpha,\alpha_{\sigma(1)},\dots,\alpha_{\sigma(m)}},\quad \Lambda_{\underline\alpha,\underline\beta}\defeq \sum_{\alpha\in\underline\alpha}\Lambda_{\alpha,\underline\alpha\cup\underline\beta\setminus\{\alpha\}} \label{Lambda set def}\end{align}
 for multisets $\{\alpha_1,\dots,\alpha_m\}$, \smash{$\underline\alpha$}, \smash{$\underline\beta$}. We will follow the convention that underlined Greek letters denote multisets of labels from $I$, while non-underlined Greek letters still denote single labels from $I$. By convention we set $\Lambda_{\emptyset}=\Lambda_{\emptyset,\underline\beta}=0$. The last two definitions in \eqref{Lambda set def} reflect the fact that the first index of $\Lambda$ will often play a special role since derivatives of $\Lambda_{\alpha_1,\dots,\alpha_k}$ will all keep $\alpha_1$ as their first index. With these notations, we note that
 \[ \Lambda_{\underline\alpha}=-\1(\abs{\alpha}>0)\Lambda(G^{-1}\partial_{\underline\alpha}G), \qquad \Lambda_{\underline\alpha,\underline\beta}=\partial_{\underline\beta}\Lambda_{\underline\alpha}  \] 
hold for arbitrary multisets $\underline\alpha$, where $\abs{\underline\alpha}$ denotes the number of elements (counting multiplicity) in the multiset. 

\subsubsection*{Expansion of a single factor of $D$}
We now use Proposition~\ref{cum exp prop} to compute $\E \Lambda(D) f$ for any random variable $f$ (later $f$ will be the product of the other $\Lambda$'s). In the remainder of Section \ref{section step} the neighbourhoods $\NN=\NN(\alpha)$ are those from Assumption \ref{assumption neighbourhood decay}. The analogue of the length scale $l$ from Section \ref{sec:toy} is thus $N^{1/4-\mu/2}$, while the parameter $R$ is still a large integer, depending only on $p$ and $\mu$. We expand \[\E\Lambda(D)f = \E \frac{1}{\sqrt N}\sum_{\alpha} w_\alpha \Lambda( \Delta^\alpha G)f+\E\Lambda(\SS[G]G) f = \E \sum_{\alpha} w_\alpha \Lambda_\alpha f+\E\Lambda(\SS[G]G) f\] and from \eqref{cum expansion statement} we obtain
\begin{align}\E\Lambda(D)f = \sum_\alpha \sum_{0\le m<R}\sum_{\bm\beta \in\NN^m} \E\left[\frac{\kappa(\alpha,\underline\beta)}{m!} + \frac{K(\alpha;\underline\beta)-\kappa(\alpha,\underline\beta)}{m!}\bigg\rvert_{W_{\NN}=0}^{\rightarrow} \right]\partial_{\underline\beta} \Lambda_\alpha f + \E\Lambda(\SS[G]G) f + \sum_{\alpha}\Omega(\Lambda_\alpha f,\alpha,\NN)\label{canc id der}.\end{align}
Here we follow the convention that $\bm \beta$ is the tuple with elements $(\beta_1,\dots,\beta_m)$ and \smash{$\underline\beta$} is the multiset obtained from the entries \smash{$\underline\beta=\{\beta_1,\dots,\beta_m\}$}, and we recall that for $\cI=I$ we denote $\kappa(w_{\alpha_1},\dots,w_{\alpha_k})$ and $K(w_{\alpha_1};w_{\alpha_2},\dots,w_{\alpha_k})$ by $\kappa(\alpha_1,\dots,\alpha_k)$ and $K(\alpha_1;\alpha_2,\dots,\alpha_k)$ (in contrast to the general setting of Section \ref{sec cum exp} where $\kappa$ was viewed as a function of the random variables). 
For $m=0$ the first term in the first bracket of \eqref{canc id der} vanishes due to $\kappa(\alpha)=\E w_\alpha=0$; for $m=1$ its contribution is given by
\[ \sum_{\alpha\in I,\beta\in \NN} \kappa(\alpha,\beta) \partial_\beta (\Lambda_\alpha f) = \sum_{\alpha\in I,\beta\in \NN}\kappa(\alpha,\beta)\Lambda_{\alpha,\beta} f +\sum_{\alpha\in I,\beta\in \NN}\kappa(\alpha,\beta)\Lambda_\alpha\partial_\beta f, \]
where we observe that the first term almost cancels the $\E \Lambda(\SS[G]G)f=-\sum_{\alpha,\beta\in I} \kappa(\alpha,\beta)\Lambda_{\alpha,\beta} f$ term except for the small contribution from $\beta\not\in\NN$. We thus rewrite \eqref{canc id der} in the form 
\begin{subequations}\label{eq cancellation identities}
\begin{align}\label{eq Lambda D cancellation correlated}
\E \Lambda(D) f &= \E \sum_{\alpha\in I,\beta\in\NN} \kappa(\alpha,\beta) \Lambda_{\alpha} \partial_{\beta}f + \E\sum_{\alpha\in I}\sum_{m<R}\sum_{\bm\beta\in\NN^m}\bigg[\frac{\kappa(\alpha,\underline\beta)}{l!}\1_{m\ge 2}+\frac{K(\alpha;\underline\beta)- \kappa(\alpha,\underline\beta)}{m!}\bigg\rvert_{W_{\NN}=0}^{\rightarrow}\bigg] \partial_{\underline\beta}\big( \Lambda_{\alpha} f \big)  \nonumber\\
&\qquad + \E\bigg(-\sum_{\alpha,\beta\in I}\kappa(\alpha,\beta)+\sum_{\alpha\in I,\beta\in \NN}\kappa(\alpha,\beta)\bigg)\Lambda_{\alpha,\beta} f + \sum_\alpha \Omega(\Lambda_\alpha f,\alpha,\NN). \\ 
\intertext{In the above derivation of \eqref{eq cancellation identities} we used directly that $\Lambda$ is linear. In the case of conjugate linear we replace $\Lambda(D)$ by $\Lambda(D^\ast)$ which is linear again. This replacement is remedied by the fact that in the definition of $\Lambda_{\alpha_1,\dots,\alpha_k}$ in \eqref{Lambda_def} we consider transposed double indices. More generally, following the same computation, we have}
\label{eq Lambda Dp cancellation correlated}
\E \Lambda(\partial_{\underline\gamma}D) f &= \E \Lambda_{\underline\gamma}f+ \E \sum_{\alpha\in I,\beta\in \NN} \kappa(\alpha,\beta) \Lambda_{\alpha,\underline\gamma} \partial_{\beta}f + \E\sum_{\alpha\in I}\sum_{m<R}\sum_{\bm\beta\in\NN^m}\bigg[\frac{\kappa(\alpha,\underline\beta)}{m!}\1_{m\ge 2}+\frac{K(\alpha;\underline\beta)-\kappa(\alpha,\underline\beta)}{m!}\bigg\rvert_{W_{\NN}=0}^{\rightarrow}\bigg]\partial_{\underline\beta}\big( \Lambda_{\alpha,\underline\gamma} f \big) \nonumber\\
&\qquad + \E\bigg(-\sum_{\alpha,\beta\in I} \kappa(\alpha,\beta)+\sum_{\alpha\in I,\beta\in \NN}\kappa(\alpha,\beta) \bigg) \Lambda_{\alpha,\{\beta\}\sqcup\underline\gamma}f + \sum_\alpha \Omega(\Lambda_{\alpha,\underline\gamma} f,\alpha,\NN). \end{align}
\end{subequations}
We think of the first two terms and the first term of the square bracket in the third term \eqref{eq Lambda Dp cancellation correlated} as the leading order terms. The second summand in the third term will be small due to the structure of the pre-cumulants and the fact that the subsequent function $\partial\Lambda f$ has the $\NN$-randomness removed. The fourth term is small because the two sums in the parenthesis almost cancel; and finally the fifth term will be small by choosing $R$ sufficiently large. We call \eqref{eq cancellation identities} (approximate) \emph{cancellation identities} as they exhibit the cancellation of second order statistics due to the definition of $\SS$ and $D$.

\subsubsection*{Iterated expansion of multiple factors of $D$}
We now use \eqref{eq Lambda Dp cancellation correlated} repeatedly to compute $\E \prod_{k\in [p]}\Lambda^{(k)}(D)$. As a first step we expand the $D$ in the $\Lambda^{(1)}$ factor, for which the special case \eqref{eq Lambda D cancellation correlated} is sufficient and we find
\begin{align}\nonumber
&\E \Lambda^{(1)}(D)\prod_{k\ge 2}\Lambda^{(k)}(D) = \sum_{\alpha_1\in I} \Omega\bigg(\Lambda_{\alpha_1}^{(1)} \prod_{k\ge 2}\Lambda^{(k)}(D),\alpha_1,\NN(\alpha_1)\bigg)\\\nonumber
& +\E \sum_{\substack{\alpha_1\in I\\\beta_1\in\NN(\alpha_1)}} \kappa(\alpha_1,\beta_1) \Lambda_{\alpha_1}^{(1)} \partial_{\beta_1}\bigg(\prod_{k\ge 2}\Lambda^{(k)}(D)\bigg)+ \E\bigg(-\sum_{\alpha_1,\beta_1\in I}\kappa(\alpha_1,\beta_1)+\sum_{\substack{\alpha_1\in I\\\beta_1\in\NN(\alpha_1)}}\kappa(\alpha_1,\beta_1)\bigg) \Lambda_{\alpha_1,\beta_1}^{(1)}\prod_{k\ge 2}\Lambda^{(k)}(D)\\
& +  \E\sum_{\alpha_1\in I}\sum_{m<R}\sum_{\bm\beta_1\in\NN(\alpha_1)^m}\bigg[\frac{\kappa(\alpha_1,\underline\beta_1)}{m!}\1_{m\ge 2}+\frac{K(\alpha_1;\underline\beta_1)- \kappa(\alpha_1,\underline\beta_1)}{m!}\bigg\rvert_{W_{\NN(\alpha_1)}=0}^{\rightarrow}\bigg] \partial_{\underline\beta_1}\bigg( \Lambda_{\alpha_1}^{(1)} \prod_{k\ge 2}\Lambda^{(k)}(D) \bigg).\label{E prod Lambda der}
\end{align} 
We now distribute the \smash{$\underline\beta_1$}-derivatives in the last term among the \smash{$\Lambda_{\alpha_1}^{(1)}$} and $\Lambda^{(k)}(D)$ factors according to the Leibniz rule. We handle the $\partial_{\beta_1}$ derivative in the second term similarly but observe that this is slightly different in the sense that the $\partial_{\beta_1}$ derivative does not hit the \smash{$\Lambda^{(1)}_{\alpha_1}$} factor. In other words, terms involving second order cumulants ($m=1$) come with the restriction that $\partial_{\beta_1}\Lambda_{\alpha_1}^{(1)}$ derivative is absent. This is precisely the effect we already encountered in Section \ref{sec:toy}; the self-energy normalization does not cancel all second order terms, it merely puts a restriction on the index-allocations in such a way that gains through Ward estimates are guaranteed in all remaining terms. In order to write \eqref{E prod Lambda der} more concisely we introduce the notations
\begin{equation}\begin{split}
\sum_{\alpha_l,\bm\beta_l}^{\sim(l)} &\defeq \sum_{\alpha_l\in I} \sum_{1\le m<R} \sum_{\bm\beta_l\in \NN(\alpha_l)^m} \frac{\kappa(\alpha_l,\underline\beta_l)}{m!} \sum_{\underline\beta_l^1\sqcup\dots\sqcup \underline\beta_l^p=\underline\beta_l} \1\left(\abs[0]{\underline\beta_l^l}=0\text{ if }\abs[0]{\underline\beta_l}=1\right),\\
\sum_{\alpha_l,\bm\beta_l}^{\ast} &\defeq \sum_{\alpha_l\in I} \sum_{0\le m<R} \sum_{\bm\beta_l\in \NN(\alpha_l)^m}\sum_{\underline\beta_l^1\sqcup\dots\sqcup \underline\beta_l^p=\underline\beta_l} \frac{K(\alpha_l;\underline\beta_l)- \kappa(\alpha_l,\underline\beta_l)}{m!}, \\
\sum_{\alpha_l,\beta_l^l}^{\#} &\defeq \bigg[-\sum_{\alpha_l,\beta_l^l\in I}\kappa_\SS(\alpha_l,\beta_l^l)+ \sum_{\alpha_l\in I}\sum_{\beta_l^l\in \NN(\alpha_l)} \kappa(\alpha_l,\beta_l^l) \bigg],\label{sum star hash def}\end{split}\end{equation} 
where $\kappa_\SS(\alpha_1,\dots,\alpha_k)\defeq \kappa(\widetilde w_{\alpha_1},\dots,\widetilde w_{\alpha_k})$ and where $\widetilde W=(\widetilde w_{\alpha})_{\alpha\in I}$ is an identical copy of $W$. The reason for introducing this identical copy  will become apparent in the next step.
 We furthermore follow the convention that $\underline\beta_l^k=\emptyset$ if $\underline\beta_{l}^k$ does not appear in the summation (which is the case for all $k\not=l$ in $\sum^{\#}_{\alpha_l,\beta_l^l}$ in \eqref{E prod der2}). Using these notations we can write \eqref{E prod Lambda der} as 
\begin{align}\label{E prod der2}
\E \prod_{k\in [p]} \Lambda^{(k)}(D) = \E\Bigg(\sum_{\alpha_1,\bm\beta_1}^{\sim(1)}+\sum_{\alpha_1,\bm\beta_1}^\ast\bigg\rvert_{W_{\NN(\alpha_1)}=0}^\rightarrow + \sum_{\alpha_1,\beta_1^1}^{\#} \Bigg) \Lambda_{\alpha_1,\underline\beta_1^1} \prod_{k=2}^p \Lambda^{(k)}\big(\partial_{\underline\beta_1^k} D\big)+\Omega,
\end{align}
where the error term $\Omega$ collects all other terms and is defined in \eqref{Omega first error} below.
We point out that the notations introduced in \eqref{sum star hash def} implicitly depend on the parameter $R$ 
determining the order of expansion.
 
\subsubsection*{Estimate of error term $\Omega$}
It remains to estimate the error term $\Omega$ which is bounded by
\begin{align}\label{Omega first error}
\Omega\defeq \sum_{\alpha_1\in I} \Omega\bigg(\Lambda_{\alpha_1}^{(1)} \prod_{k\ge 2}\Lambda^{(k)}(D),\alpha_1,\NN(\alpha_1)\bigg)\le_R \sum_{\alpha_1,\bm\beta_1\in\NN(\alpha_1)^R} \norm[3]{\partial_{\underline\beta_1}\bigg( \Lambda_{\alpha_1}^{(1)} \prod_{k\ge 2}\Lambda^{(k)}(D)\bigg)\bigg\rvert_{\widehat W_t} }_2
\end{align}
for some $t\in[0,1]$, where $\widehat W_t =\widehat W_t^{(\alpha_1)}=t W_{\NN(\alpha_1)}+W_{\NN(\alpha_1)^c}$, where we recall the definition of $\Omega(\Lambda,\alpha,f)$ in \eqref{cum expansion statement} and its bound in \eqref{hoelder error bound}. To further estimate this expression, we first distribute the $\partial_{\underline\beta_1}$ derivative to the $p$ factors involving $\Lambda^{(1)},\dots,\Lambda^{(p)}$ following the Leibniz rule, and then separate those factors by a simple application of H\"older inequality into $p$ factors of $\norm{\cdot}_{2p}$ norms. Each of these factors can be written as a sum of terms of the type $\norm[1]{\Lambda^{(k)}(\partial_{\underline\gamma} G\big\rvert_{\widehat W_t})}_{2p}$ or $\norm[1]{\Lambda^{(k)}(\partial_{\underline\gamma} D\big\rvert_{\widehat W_t})}_{2p}$ for some derivative operator $\partial_{\underline\gamma}$. We can then estimate these norms using 
\begin{align}\label{Lambda norm bound Ghat Dhat bound}\norm{\Lambda(R)}_q\le \norm{\Lambda}\norm{R}_q,\quad\text{and}\quad \norm[1]{\partial_{\underline\gamma} G\big\rvert_{\widehat W_t}}_{q} + \norm[1]{\partial_{\underline\gamma} D\big\rvert_{\widehat W_t}}_{q} \le_{\abs[0]{\underline\gamma}} N^{-\abs[0]{\underline\gamma}/2} (1+\tnorm{S}) (1+\norm{G}_{Cq\abs[0]{\underline\gamma}/\mu})^{(\abs[0]{\underline\gamma}+5)/\mu},\end{align}
where the second inequality follows from Lemma \ref{G D triv bound lemma}, and we note that $Cp\abs[0]{\underline\gamma}\le CRp^2$. We now count the total number of derivatives: There are $R+1$ derivatives from $\abs[0]{\underline\beta_1}$ and $\alpha_1$, each providing a factor of $N^{-1/2}$. It remains to account for the $\alpha_1,\bm\beta_1$-sums which is at most of size $\sum_{\alpha_1}\abs{\NN(\alpha_1)}^R\le N^{2+R/2-\mu R}$. We now choose $R$ large enough so that 
\[N^{2-(R+1)/2+R/2-\mu R}\le N^{-p}, \]
which is satisfied if we choose $R\ge 3p/\mu$. Combining these rough bounds we have shown that, up to irrelevant combinatorial factors,
\begin{align}\label{E2 bound}
\Omega \le_{p,\mu} N^{-p}  \bigg[\prod_{k=1}^p &\norm[0]{\Lambda^{(k)}} \bigg] \Big(1+\tnorm{\SS}\Big)^p \Big(1+\norm{G}_{Cp^3/\mu^2}\Big)^{Cp/\mu^2}.
\end{align}

\subsubsection*{Main expansion formula for multiple factors of $D$}
Formula \eqref{E prod der2} with the bound \eqref{E2 bound} on the error term is the first step where the cumulant expansion was used in the $\Lambda^{(1)}(D)$ factor. Now we iterate this procedure for the $\Lambda^{(2)}(D),\Lambda^{(3)}(D),\dots$ inductively. We arrive at the following proposition modulo the claimed bound on the overall error which we will prove after an extensive explanation. 
\begin{proposition}\label{prop explicit formula Lambda(D) power}
Let $\Lambda^{(1)},\dots,\Lambda^{(p)}$ be linear (or conjugate linear) functionals and let $p\in \N$ be given. Then we have
\begin{align}\label{E prod D full exp}
\E \prod_{k\in[p]} \Lambda^{(k)}(D) = \E \prod_{l\in[p]}^\rightarrow \left(1+\sum_{\alpha_l,\bm\beta_l}^{\sim(l)} + \sum^{\ast}_{\alpha_l,\bm\beta_l}\Bigg\rvert_{W_{\NN(\alpha_l)}=0}^{\rightarrow} +\sum^{\#}_{\alpha_l,\beta_l^l}\right) \prod_{k\in [p]} \begin{cases}
\Lambda^{(k)}_{\alpha_k,\bigsqcup_{l\in [p]}\underline\beta_l^k}&\text{if }\sum_{\alpha_k}\\
\Lambda^{(k)}_{\bigsqcup_{l<k} \underline{\beta}_l^k, \bigsqcup_{l>k}\underline\beta_l^k}&\text{else}
\end{cases} + \Omega,
\end{align}
where ``if $\sum_{\alpha_k}$'' means cases where after multiplying out the first product $\prod_l$ the summation over the index $\alpha_k$ is performed. Under Assumptions \ref{assumption A}, \ref{assumption high moments} and \ref{assumption neighbourhood decay}, 
the error term $\Omega$  is bounded by 
\begin{align} &\abs{\Omega}  \le_{p,\mu} N^{-p}\braket{z}^{-p} \bigg[\prod_{k=1}^p \norm[0]{\Lambda^{(k)}} \bigg] (1+\tnorm{\SS})^p
\Big(1+\norm{G}_{q}\Big)^{\frac{Cp}{\mu^2}}\bigg(1+ \frac{\norm{G}_{q}}{N^\mu}\bigg)^{\frac{Cp^2}{\mu^2}},\label{cE error}\end{align}
if we choose $R=4p/\mu$ to be order of expansion in the summations, see \eqref{sum star hash def}. Furthermore, we set $q\defeq Cp^3/\mu^2$ for some constant $C$, 
and $\norm[0]{\Lambda^{(k)}}$ denotes the operator norm of the linear functional $\Lambda^{(k)}$.  
\end{proposition}
For \eqref{E prod D full exp} we recall the convention that $\underline\beta_l^k=\emptyset$ whenever $\underline\beta_l^k$ is not summed, i.e., for the contribution from the $1$ in the $l$-th factor, or the contribution from $\sum^{\#}$ in the $l$-th factor for $k\not=l$. Moreover, we remind the reader that the custom notation $\rvert^\rightarrow_{W_\NN=0}$ was introduced right after \eqref{E wwww f}. We also note that the terms with a $1$ from the first factor vanish as they contain $\Lambda^{(1)}_{\emptyset,\sqcup_{l>1}\underline\beta_l^1}=0$. Moreover, we can now explain why we introduced the identical copy $\widetilde W$ of $W$ in the definition of $\kappa_\SS$ in 
 \eqref{sum star hash def}.  The cumulants in the representation of the term 
 $\SS[G]G=-\sum_{\alpha,\beta\in I}\kappa_\SS(\alpha,\beta)\Lambda_{\alpha,\beta}$ should not be affected by the restriction
 imposed by the operation $\rvert_{W_\NN=0}^\rightarrow$.
 Changing $W$ to $\widetilde W$ within the definition of $\kappa_S$ protects it 
 from the action of $\rvert_{W_\NN=0}^\rightarrow$ that turns all subsequent $W$ variables zero. 
  This non-restriction of the particular sum is formally achieved by writing $\SS$ in terms of $\kappa_\SS$ instead of $\kappa$. This is only a notational pedantry, in the next step where we multiply \eqref{E prod D full exp} out, it will disappear. We remark that because of the effect of $\rvert^{\rightarrow}_{W_{\NN}=0}$ the order in which the product in (36) is performed matters. It starts with $l=1$ and ends with $l=p$.

We point out that the estimate \eqref{cE error} not only provides the necessary $N^{-p}$ factor,
but it also involves at most $O(p)$ power of $\norm{ G}_q$ without an extra smallness factor $N^{-\mu}$, see
Remark \ref{Gpower remark}.  While from the perspective of an $N$-power
counting, any factor $\norm{ G}_q$ is neutral, of order one, we need to track that its power is not too big.
 Factors of $\norm{ G}_q$ 
that come with a factor $N^{-\mu}$ can be handled much easier and are not subject to the restriction of their power.

\subsubsection*{Reformulation of the main expansion formula} 
We now derive an alternative, less compact formula \eqref{E prod Lambda tuples full} for \eqref{E prod D full exp} which avoids the provisional $\big\rvert^\rightarrow$ notation. By expanding the first product in \eqref{E prod D full exp} we can rearrange \eqref{E prod D full exp} according to partitions $[p]=L_1\sqcup \dots \sqcup L_4$, where $L_i$ contains those indices $l$ for which the $l$-th factor in the product contributes with its $i$-th term. In particular $L\defeq L_2\sqcup L_3\sqcup L_4\subset [p]$ contains those indices $l$, for which $\alpha_l,\bm\beta_l$ are summed. We shall use the nomenclature that labels $\alpha_l$ and the elements of $\underline\beta_l$ are \emph{type-$l$} labels. These labels have been generated in the $l$-th application of the cancellation identities \eqref{eq cancellation identities}. The partition $\underline\beta_l^1\sqcup\dots\sqcup\underline\beta_l^p=\underline\beta_l$ encodes how these labels have been distributed among the $p$ factors via the Leibniz rule. Thus labels $\underline\beta_l^k$ have been generated on $\Lambda^{(k)}$ at the $l$-th application of \eqref{eq cancellation identities}. Thus $L$ encodes the types of labels present in the different parts of the expansion. To specify the number of type--$l$ labels we introduce the notations
\[M_l\defeq \abs[0]{\underline\beta_l}, \quad M_l^k\defeq \abs[0]{\underline\beta_l^k}.\]
Thus the number of labels of \emph{type} $l$ is $M_l+1$ and the number of type $l$-labels in $\Lambda^{(k)}$ is $M_l^k+\delta_{lk}$. We observe that in all non-zero terms of \eqref{E prod D full exp} the labels $\alpha_l$, $\underline\beta_l$ for $l\in L$ are distributed to the $\Lambda^{(1)},\dots,\Lambda^{(p)}$ in such a way that
\begin{enumerate}[(a)]
\item there are $p$ factors $\Lambda^{(1)},\dots,\Lambda^{(p)}$,\label{rule p terms}
\item every $\Lambda^{(k)}$ carries at least one label (that is for all $k$, $\sum_{l\in L} (M_l^k+\delta_{kl})\ge 1$)\label{rule no empty lambda},
\item for every $l\in L$, there exist at least two and at most $R-1$ \emph{type-$l$} labels (that is for all $l\in L$, $M_l\ge 1$)\label{rule at least two labels}, for $l\in L_4$ there exist exactly two \emph{type-$l$} labels in such a way that $M_l=M_l^l=1$,
\item if for some $l\in L_2$ there are exactly two \emph{type-$l$} labels, then these two labels must occur in distinct $\Lambda's$ (that is, if $l\in L_2$ and $M_l=1$, then $M_l^l=0$)\label{rule cancellation}. 
\item for every $l\in L$, the first index of $\Lambda^{(l)}$ is $\alpha_l$.\label{every index appears first once} 
\end{enumerate}
We now reformulate \eqref{E prod D full exp} in such a way that we first sum up over the partitions $L_1\sqcup L_2\sqcup L_3\sqcup L_4=[p]$, the collection of multiplicities $M=\tuple{M_l^k|l\in L,k\in[p]}$ and the permutations of indices, and only then perform the actual summation over the labels from $I$. As the first three sums carry no $N$, they are irrelevant for the $N$-power counting. From \eqref{E prod D full exp} we find
\begin{align}\label{E prod Lambda tuples full}
\E \prod_{k\in [p]} \Lambda^{(k)}(D) =&\E\sum_{\bigsqcup L_i=[p]} \sum_M^{\sim(L)} C_M \sum_{\sigma}^{\sim(M)} \left[\prod_{l\in L_3}\sum_{\alpha_l,\bm\beta_l\not\in\NN_{L_3}^{<l}}^{(M,l)} \frac{K(\alpha_l;\bm\beta_l)-\kappa(\alpha_l,\bm\beta_l)}{\abs{\bm\beta_l}!}\right] \cM' + \landauO[p,\mu]{N^{-p}},\\\nonumber
\cM'\defeq &  \left[\prod_{l\in L_4}\Bigg(-\sum_{\alpha_l,\beta_l^l\in I}+\sum_{\alpha_l\in I\setminus\NN_{L_3}^{<l}}\sum_{\beta_l^l \in \NN(\alpha_l)\setminus\NN_{L_3}^{<l}}\Bigg) \frac{\kappa(\alpha_l,\beta_l^l)}{1!}\right] \cM, \\
\cM\defeq & \left[\prod_{l\in L_2}\sum_{\alpha_l,\bm\beta_l\not\in\NN_{L_3}^{<l}}^{(M,l)} \frac{\kappa(\alpha_l,\bm\beta_l)}{\abs{\bm\beta_l}!}\right]\bigg[\bigg(\prod_{k\in L} \Lambda^{(k)}_{\alpha_k,\sigma_k(\bm\beta^k)}\bigg) \bigg(\prod_{k\not\in L} \Lambda^{(k)}_{\sigma_k(\bm\beta^k)}\bigg)\bigg]\bigg\rvert_{W_{\NN_{L_3}}=0},\nonumber
\end{align}
where $\sum_{M}^{\sim(L)}$ is the sum over all arrays $M$ fulfilling (a)--(e) above and $C_M$ are purely combinatorial constants bounded by a function of $p,R$; $C_M\le C(p,R)$, in which we also absorbed the $(-1)$'s from the $L_4$ terms. Moreover, $\sum_{\sigma}^{\sim(M)}$ is the sum over all permutations $\sigma_1$,\dots, $\sigma_p$ in the permutation groups $S_{M^1}, \dots, S_{M^p}$ (where $M^k\defeq \sum_{l\in L} M^k_l$) such that for $k\not\in L$ the first element of $\sigma_k(\bm\beta^k)$ is from $(\,\bm\beta_l^k\,\mid\,l\in L\cap[k])$. Furthermore, for any $\NN\subset I$ we set 
\[\sum_{\alpha_l,\bm\beta_l\not\in\NN}^{(M,l)}\defeq \sum_{\alpha_l\in I\setminus \NN}\prod_{k\in[p]}\sum_{\bm\beta_l^k\in (\NN(\alpha_l)\setminus \NN)^{M_l^k}}.\] 
Finally, we introduced the notations $\NN_{L_3}^{<l}\defeq \bigcup_{l>k\in L_3}\NN(\alpha_k)$, and $\NN_{L_3}\defeq\bigcup_{k\in L_3}\NN(\alpha_k)$. Here the $\bm\beta_l^k$ are actual (ordered) tuples and not multisets, which is why we denote them by boldfaced Greek letters to avoid possible confusion with the previously used $\underline\beta_l^k$. In \eqref{E prod Lambda tuples full} we furthermore used the short-hand notation $\bm\beta^k=(\bm\beta^k_l)_{l\in L}$ for the tuple (ordered according to the natural order on $L\subset[p]\subset\N$) of $\bm\beta^{k}_{l}$. We note that the artificial $\kappa_\SS$ from \eqref{E prod D full exp} has been removed in \eqref{E prod Lambda tuples full} since we ``pushed'' the $\rvert^{\rightarrow}$-operator all the way to the end. In the following we will establish bounds on \eqref{E prod Lambda tuples full} for fixed $L$ and $M$ and fixed permutations $\sigma_1,\dots,\sigma_p$. Since the number of possible choices for $M$, $L$ and permutations is finite, depending on $R$ and $p$ only, this will be sufficient for bounding $\E \prod \Lambda^{(k)}(D)$. We also stress that the (multi)labels $\bm\beta_l^k$ themselves are not important, but only their type $l$.  

\subsubsection*{Proof of the error bound in Proposition \ref{prop explicit formula Lambda(D) power}}
We now turn to the proof of the claimed error bound \eqref{cE error}. So far this was only done for the error from the first cumulant expansion in \eqref{E2 bound}.
\begin{proof}[Proof of the error bound in Proposition \ref{prop explicit formula Lambda(D) power}] 
The error $\Omega$ in \eqref{E prod D full exp} is a sum over $p$ terms, where the $j$-th term is the error from the expansion of $\Lambda^{(j)}(D)$. 
Recalling the definition of $\Omega(f,i,\NN)$ from \eqref{def Omega}, this $j$-th expansion error is given by
\begin{align*}
\Omega_j\defeq\sum_{\alpha_j}\Omega\left(\prod_{l<j} \bigg(1+\sum_{\alpha_l,\bm\beta_l}^{\sim(l)} + \sum^{\ast}_{\alpha_l,\bm\beta_l}\Bigg\rvert_{W_{\NN(\alpha_l)}=0}^{\rightarrow} +\sum^{\#}_{\alpha_l,\beta_l^l}\bigg) \prod_{k=1}^p \left.\begin{cases}
\Lambda^{(k)}_{\alpha_k,\bigsqcup_{l\in [p]}\underline\beta_l^k}&\text{if }k=j\text{ or }\big(k<j,\sum_{\alpha_k}\big)\\
\Lambda^{(k)}\big( \partial_{\bigsqcup_{l<k} \underline{\beta}_l^k} D \big)&\text{if }k>j\\
\Lambda^{(k)}_{\bigsqcup_{l<k} \underline{\beta}_l^k, \bigsqcup_{l>k}\underline\beta_l^k}&\text{else}
\end{cases}\right\},\alpha_j,\NN(\alpha_j) \right),
\end{align*}
where ``if $(k<j,\sum_{\alpha_k})$'' means ``if $k<j$ and $\alpha_k$ is summed''. This $j$-th error $\Omega_j$ can be estimated through \eqref{hoelder error bound} and Assumption \ref{assumption high moments} by the sum of
\begin{align} \Bigg[\prod_{l\in L_2\sqcup L_3}\sum_{\alpha_l,\bm\beta_l}^{(M,l)}\Bigg]\Bigg[\prod_{l\in L_4}\sum_{\alpha_l,\beta_l^l\in I}\Bigg] \sum_{\alpha_j}\sum_{\bm\beta_j\in\NN(\alpha_j)^R} \norm[3]{\bigg(\prod_{k\in L} \Lambda^{(k)}_{\alpha_k,\sigma_k(\bm\beta^k)}\bigg) \bigg(\prod_{k\in [j]\setminus L} \Lambda^{(k)}_{\sigma_k(\bm\beta^k)}\bigg)\bigg(\prod_{k>j} \Lambda^{(k)}(\partial_{\sigma_k(\bm\beta^k)} D)\bigg\rvert_{\widehat W}\bigg)}_2,\label{general error}\end{align}
over partitions $L=L_2\sqcup L_3 \sqcup L_4 \subset [j-1]$, arrays $M$ fulfilling \eqref{rule p terms}--\eqref{every index appears first once} above and partitions $\sigma_k$. In all terms $\widehat W$ is a modification of $W$ which differs from $W$ in at most $C\sqrt N$ entries. The previously studied error from \eqref{Omega first error} for example corresponds to $j=1$, $L_2=L_3= L_4 =\emptyset$. The combinatorics of all these summations are independent of $N$, hence can be neglected. So we can focus on a single term of the form \eqref{general error}. The norm in \eqref{general error} will first be estimated by H\"older and then by \eqref{Lambda norm bound Ghat Dhat bound} to reduce it to many factor of $\norm{G}_q$. We now have to count the size of the sums, the number of $N^{-1/2}$ factors from the derivatives, and the number of $\norm{G}_q$'s we collect in the bound. We start with the sums which are at most of size
\begin{equation} \label{power counting sums}N^{2\abs{L_2\sqcup L_3}} (N^{1/2-\mu})^{M_{L_2\sqcup L_3}} (N^2\cdot N^2)^{\abs{L_4}} N^2 (N^{1/2-\mu})^{R}= N^{2\abs{L_2\sqcup L_3}+(M_{L_2\sqcup L_3}+R)(1/2-\mu)+4\abs{L_4}+2}.\end{equation}
Here the first factor comes from the $\alpha_l$ summations for $l\in L_2\sqcup L_3$, while the second term comes from the corresponding $\bm\beta_l$ summations. The third factor comes from the $\alpha_l,\beta_l^l$-summations for $l\in L_4$, and finally the fifth and sixth factor correspond to the $\alpha_j$ and $\bm\beta_j$ summations. Next, we count the total number of derivatives. Every index $\alpha_l$ and $\beta_l^k$ accounts for a derivative, and each derivative contributes a factor of $N^{-1/2}$. So we have
\begin{equation} \label{power counting derivatives}(N^{-1/2})^{\abs{L_2\sqcup L_3} + M_{L_2\sqcup L_3}+2\abs{L_4}+(R+1)}=  N^{-\abs{L_2\sqcup L_3}/2-M_{L_2\sqcup L_3}/2-\abs{L_4}-(R+1)/2}, \end{equation}
so that altogether from \eqref{power counting sums} and \eqref{power counting derivatives} we have an $N$-power of
\[N^{3/2 (\abs{L_2\sqcup L_3}+1)+ 3\abs{L_4} - R\mu} N^{-\mu M_{L_2\sqcup L_3}} \le N^{-p} N^{-\mu M_{L_2\sqcup L_3}}.\]
It remains to count the number of $\norm{G}_{CRp^2}=\norm{G}_q$ coming from the application of \eqref{Lambda norm bound Ghat Dhat bound}, which in total provides
\begin{align} \sum_{k\in L} (1+\abs[0]{\bm\beta^k}+5) + \sum_{k\in [j]\setminus L} (\abs[0]{\bm\beta^k}+5) + \sum_{k>j}(\abs[0]{\bm\beta^k}+5) = 5p + \abs{L_2\sqcup L_3} + M_{L_2\sqcup L_3} + 2\abs{L_4} + R+1\le Cp/\mu + M_{L_2\sqcup L_3} 
\label{Gfactors}
\end{align}
factors of $\norm{G}_q$. The claim \eqref{cE error} now follows from the trivial estimate $M_{L_2\sqcup L_3}\le Rp\le C p^2/\mu$.  
\end{proof}

Subsequently we establish a bound on the rhs.~of \eqref{E prod Lambda tuples full}, by first estimating it in terms of $\norm{\cM'}_p$, then estimating $\norm{\cM'}_p$ in terms of $\norm{\cM}_p$ and finally bounding the leading contribution $\cM$. We consider the first two steps in this procedure as errors stemming from the neighbourhood structure of the expansion, while the third step is concerned with the leading order contribution from the expansions. 
In Section \ref{sec neighbourhood errors} we consider the errors stemming from the neighbourhood structure, while in Sections \ref{av bound sec} and \ref{iso bound sec} we derive bounds on $\norm{\cM}_p$ for the averaged and isotropic case, separately. For simplicity we first carry out the technically most involved argument from Sections \ref{av bound sec}--\ref{iso bound sec} in the extreme case $L_3=L_4=\emptyset$ where the neighbourhood errors are absent. Finally, we explain the necessary modifications for the general case in Section \ref{sec mods}. 

\subsection{Bound on neighbourhood errors}\label{sec neighbourhood errors}
We start with the bound on the $L_3$-factors in \eqref{E prod Lambda tuples full}. Neglecting the irrelevant combinatorial factors $\abs[0]{\bm\beta_l}!$ and the summations over $L_i$, $M$ and $\sigma$, we have to estimate 
\begin{align}\label{eq cE def}
\cE\defeq\Bigg[\prod_{l\in L_3} \sum_{\alpha_l,\bm\beta_l\not\in\NN_{L_3}^{<l}}^{(M,l)} \Bigg] \cE(\alpha_{L_3},\bm\beta_{L_3})\defeq\Bigg[\prod_{l\in L_3} \sum_{\alpha_l,\bm\beta_l\not\in\NN_{L_3}^{<l}}^{(M,l)} \Bigg] \E \cM'\prod_{l\in L_3} \big[K(\alpha_l;\bm\beta_l)-\kappa(\alpha_l,\bm\beta_l)\big]. 
\end{align} 
By the pigeon hole-principle we find that for every $l\in L_3$ and any assignment of $\alpha_l,\bm\beta_l$ there exist some $n_l<R$ such that we have a partition $\underline\beta_l=\underline\beta_l^{(i)}\sqcup\underline\beta_l^{(o)}$ into \emph{inside} and \emph{outside} elements with $\underline\beta_l^{(i)}\subset \NN_{n_l}(\alpha_l)$ and $\underline\beta_l^{(o)}\subset \NN_{n_l+1}(\alpha_l)^c$ since $\abs[0]{\underline\beta_l}=M_l<R$ (see rule \eqref{rule at least two labels}). We recall the nested structure of the neighbourhoods as stated in Assumption \ref{assumption neighbourhood decay}, and provide an illustration of the ``security layers'' in Figure \ref{neighbourhood sketch}.
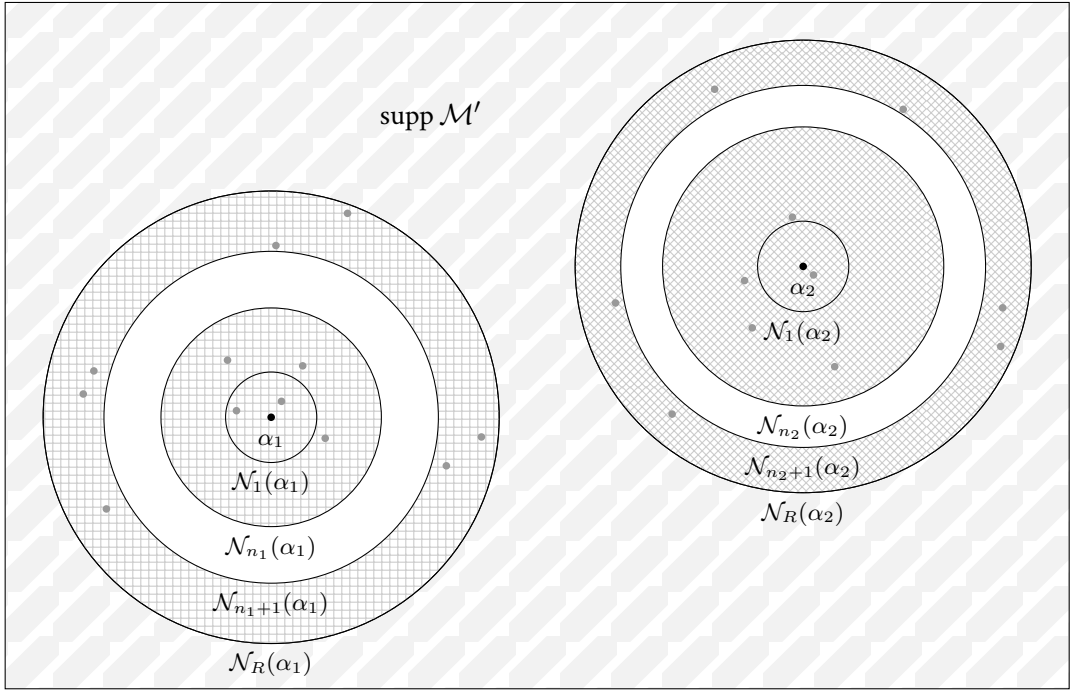
\begin{figure}
\centering
\makeatletter
\tikzset{
        hatch distance/.store in=\hatchdistance,
        hatch distance=5pt,
        hatch thickness/.store in=\hatchthickness,
        hatch thickness=5pt
        }
\pgfdeclarepatternformonly[\hatchdistance,\hatchthickness]{north east hatch}%
    {\pgfqpoint{-1pt}{-1pt}}%
    {\pgfqpoint{\hatchdistance}{\hatchdistance}}%
    {\pgfpoint{\hatchdistance-1pt}{\hatchdistance-1pt}}%
    {
        \pgfsetcolor{\tikz@pattern@color}
        \pgfsetlinewidth{\hatchthickness}
        \pgfpathmoveto{\pgfqpoint{0pt}{0pt}}
        \pgfpathlineto{\pgfqpoint{\hatchdistance}{\hatchdistance}}
        \pgfusepath{stroke}
    }
\makeatother
\usetikzlibrary{patterns}
\begin{tikzpicture}[
  node distance=7mm,
  title/.style={font=\fontsize{6}{6}\color{black!50}\ttfamily},
  typetag/.style={rectangle, draw=black!50, font=\scriptsize\ttfamily, anchor=west},
  NN/.style={circle,draw=black!50,minimum size=3},
  baseline={([xshift=-4cm]current bounding box.center)}
]
\small
  \pgfresetboundingbox
  \fill [use as bounding box,draw=black,pattern=north east hatch,
                    pattern color=black!5,
                    hatch distance=20pt,
                    hatch thickness=10pt] (-3.5,-3.6) rectangle (10.5,5.5);
  \pgfmathsetseed{7899778}
  \node [circle,minimum size=6cm,draw,fill=white] {};
  \node [label=below:$\mathcal{N}_R(\alpha_1)$,circle,minimum size=6cm,draw,fill=white,pattern=grid,pattern color=black!20](aa) {};
  \foreach \p in {1,...,15}
  {
    \pgfmathsetmacro{\rho}{rand};
    \pgfmathsetmacro{\rad}{rand};
    \fill[color=black!40] ({2.9*sqrt(abs(\rad))*sin(360*\rho)},{2.9*sqrt(abs(\rad))*cos(360*\rho)}) circle (0.05) node{};
  }
  \node[label=below:$\mathcal{N}_{n_1+1}(\alpha_1)$,circle,minimum size=4.4cm,draw,fill=white] (0,0) {};

  \node[label=below:$\mathcal{N}_{n_1}(\alpha_1)$,circle,minimum size=2.9cm,draw,pattern=grid,pattern color=black!20] (0,0) {};
    \foreach \p in {1,...,5}
  {
    \pgfmathsetmacro{\rho}{rand};
    \pgfmathsetmacro{\rad}{rand};
    \fill[color=black!40] ({1.4*sqrt(abs(\rad))*sin(360*\rho)},{1.4*sqrt(abs(\rad))*cos(360*\rho)}) circle (0.05) node{};
  }
  \node (aa)[label=below:$\mathcal{N}_1(\alpha_1)$,circle,minimum size=1.2cm,draw] (0,0) {};
  \fill (0,0) circle (0.05) node[label=below:$\alpha_1$]{};

  \node [circle,minimum size=6cm,draw,fill=white]at (7,2) {};
  \node [label=below:$\mathcal{N}_R(\alpha_2)$,circle,minimum size=6cm,draw,pattern=crosshatch,pattern color=black!20]at (7,2) (bb) {};
  \foreach \p in {1,...,15}
  {
    \pgfmathsetmacro{\rho}{rand};
    \pgfmathsetmacro{\rad}{rand};
    \fill[color=black!40] ({7+2.9*sqrt(abs(\rad))*sin(360*\rho)},{2+2.9*sqrt(abs(\rad))*cos(360*\rho)}) circle (0.05) node{};
  }
  \node[label=below:$\mathcal{N}_{n_2+1}(\alpha_2)$,circle,minimum size=4.8cm,draw,fill=white] at (7,2) {};
  \node[label=below:$\mathcal{N}_{n_2}(\alpha_2)$,circle,minimum size=3.7cm,draw,pattern=crosshatch,pattern color=black!20] at (7,2) {};
      \foreach \p in {1,...,5}
  {
    \pgfmathsetmacro{\rho}{rand};
    \pgfmathsetmacro{\rad}{rand};
    \fill[color=black!40] ({7+1.4*sqrt(abs(\rad))*sin(360*\rho)},{2+1.4*sqrt(abs(\rad))*cos(360*\rho)}) circle (0.05) node{};
  }
  \node (aa)[label=below:$\mathcal{N}_1(\alpha_2)$,circle,minimum size=1.2cm,draw] at (7,2) {};
  \fill (7,2) circle (0.05) node[label=below:$\alpha_2$]{};

  \node at (2.1,4) {\large$\supp\mathcal{M}'$};

\end{tikzpicture}
\caption{Illustration for the bound on $\cE$ \eqref{eq cE def}. Gray dots ${\color{gray}\bullet}$ denote the $\underline\beta_1,\underline\beta_2$ labels. Since there are $\lvert\underline\beta_i\rvert<R$ labels and $R$ rings, there is always one empty ring by the pigeon-hole principle.}\label{neighbourhood sketch}
\end{figure}
According to \eqref{two groups pre cum} we can then write ($L_3'$ collects those indices where we took the middle term of \eqref{two groups pre cum} in the $l$ factor)
\begin{align*}
\cE(\alpha_{L_3},\bm\beta_{L_3}) =  \sum_{L_3=L_3'\sqcup L_3''}(-1)^{\abs{L_3''}}  \prod_{l\in L_3''} \Bigg[\sum_{\gamma_l^{(i)}\subset \beta_l^{(i)}}\sum_{\gamma_l^{(o)}\subsetneq \beta_l^{(o)}} \kappa(\alpha_l,\underline\beta_l^{(i)}\setminus \underline\gamma_l^{(i)},\underline\beta_l^{(o)}\setminus \underline\gamma_l^{(o)})\Bigg] \E f\prod_{l\in L_3'}\big[K(\alpha;\underline\beta_l^{(i)})-\kappa(\alpha,\underline\beta_l^{(i)})\big],
\end{align*}
where 
\begin{align*}
f\defeq \cM' \prod_{l\in L_3'} \big(\Pi \underline\beta_l^{(o)}\big) \prod_{l\in L_3''} \Big[ \big(\Pi \underline\gamma_l^{(i)}\big) \big(\Pi \underline\gamma_l^{(o)}\big)\Big]
\end{align*}
is a random variable supported in $\bigcap_{l\in L_3'} \NN_{n_l+1}(\alpha_l)^c$, i.e., well separated from the variables $K(\alpha_l;\underline\beta_l^{(i)})$ for $l\in L_3'$. It remains to estimate a quantity of the type $\E f g_1 \dots g_k$, where $f,g_1,\dots,g_k$ are random variables whose supports are pairwise separated by ``security layers'' and where each $g_i$ is of the form
$K-\kappa$ with $\E g_i=0$. Here $k=\abs{L_3'}$ and from Lemma \ref{tree decay lemma mutally separated} and Assumption \ref{assumption neighbourhood decay} it follows that $\E f g_1\dots g_k\le_k \norm{f}_{k+1} N^{-3 \lceil k/2\rceil}$. According to Lemma \ref{mixing cumulant decay lemma} the $\kappa(\alpha_l,\underline\beta_l^{(i)}\setminus\underline\gamma_l^{(i)},\underline\beta_l^{(o)}\setminus\underline\gamma_l^{(o)})$ factors are also at least $N^{-3}$ small and we can conclude that
 \begin{equation}\label{layers}
 \abs{\cE(\alpha_{L_3},\bm\beta_{L_3})}\le_{p,R} N^{-3\lceil\abs{L_3}/2\rceil}\norm{\cM'}_p.
 \end{equation}

Next, we use the triangle inequality to pull the $L_4$ summation out of $\norm{\cM'}_p$ to achieve a bound in terms of $\norm{\cM}_p$. We have 
\begin{align*}&\abs[3]{\Bigg(-\sum_{\alpha_l,\beta_l^l\in I} + \sum_{\alpha_l\in I\setminus \NN_{L_3}^{<l}}\sum_{\beta_l^l\in \NN(\alpha_l)\setminus \NN_{L_3}^{<l}}\Bigg) \kappa(\alpha_l,\beta_l^l)}\le \Bigg(\sum_{\alpha_l\in I\setminus\NN_{L_3}^{<l}}\sum_{\beta_l^l\in \NN_{L_3}^{<l}}+\sum_{\alpha_l\in I\setminus\NN_{L_3}^{<l}}\sum_{\beta_l^l\in I\setminus\NN(\alpha_l)} + \sum_{\alpha_l\in \NN_{L_3}^{<l}}\sum_{\beta_l^l\in I}\Bigg) \abs[0]{\kappa(\alpha_l,\beta_l^l)} \\
& \quad\le \Bigg( \sum_{\beta_{l}^l\in\NN_{L_3}^{<l}}\sum_{\alpha_l\in\NN(\beta_l^l)}+\sum_{\beta_{l}^l\in\NN_{L_3}^{<l}}\sum_{\alpha_l\in I\setminus\NN(\beta_l^l)}+\sum_{\alpha_l\in I}\sum_{\beta_l^l\in I\setminus\NN(\alpha_l)} +\sum_{\alpha_l\in \NN_{L_3}^{<l}}\sum_{\beta_l^l\in \NN(\alpha_l)}+\sum_{\alpha_l\in \NN_{L_3}^{<l}}\sum_{\beta_l^l\in I\setminus\NN(\alpha_l)} \Bigg)\abs[0]{\kappa(\alpha_l,\beta_l^l)}\le C N,
\end{align*}
where we estimated the first and the fourth term with two small summations purely by size $(CN^{1/2-\mu})^2\le CN$ and the other terms using the fact that $\abs{\kappa(\alpha,\beta)}\lesssim N^{-3}$ for $\beta\in I\setminus\NN(\alpha)$. Summarizing, we thus have that
\begin{align}\label{E prod Lambda cM bound}
\abs{\E \prod \Lambda^{(k)}(D)} &\le_{p,\mu} N^{-p}+\sum_{\bigsqcup L_i=[p]}\sum_{M}^{\sim(L)} \frac{N^{\abs{L_4}}}{N^{3\lceil \abs{L_3}/2\rceil}} \sum_{\sigma}^{\sim(M)}   \Bigg[\prod_{l\in L_3}\sum_{\alpha_l,\bm\beta_l\not\in\NN_{L_3}^{<l}}^{(M,l)}\Bigg]\Bigg[\prod_{l\in L_4}\max_{\alpha_l,\beta_l^l\in I}\Bigg] \norm{\cM}_p,
\end{align}
and it only remains to estimate the leading order term $\cM$, as defined in \eqref{E prod Lambda tuples full}. This has to be done separately for averaged and isotropic bound and should be considered as the main part of the proof. To simplify notations we will first prove the bound on $\cM$ for the case that $L_3=L_4=\emptyset$ and $\NN(\alpha)=I$. In particular $L_3=\emptyset$ implies that $\NN_{L_3}=\emptyset$ and therefore in the next two Sections~\ref{av bound sec} and \ref{iso bound sec} we now aim at deriving a bound on $\norm[1]{\cM((\Lambda^{(k)})_{k\in[p]}; L,M,\sigma)}_p$, where 
\begin{align}\label{E prod Lambda tuples}
\cM((\Lambda^{(k)})_{k\in[p]}; L,M,\sigma)\defeq \left[\prod_{l\in L}\sum_{\alpha_l,\bm\beta_l}^{(M,l)} \frac{\kappa(\alpha_l,\bm\beta_l)}{\abs{\bm\beta_l}!}\right]\left(\prod_{k\in L} \Lambda^{(k)}_{\alpha_k,\sigma_k(\bm\beta^k)}\right) \left(\prod_{k\not\in L} \Lambda^{(k)}_{\sigma_k(\bm\beta^k)}\right),\quad \sum_{\alpha_l,\bm\beta_l}^{(M,l)}\defeq \sum_{\alpha_l\in I}\prod_{k\in[p]}\sum_{\bm\beta_l^k\in I^{M_l^k}}.
\end{align}
The definition of $\cM$ in \eqref{E prod Lambda tuples} agrees with the one in \eqref{E prod Lambda tuples full} in the special case $L_3=L_4=\emptyset$, except for a tiny contribution from $\bm\beta_l\not\subset \NN(\alpha_l)$. The reason for extending the sum here to the whole index set is twofold: First, we do not have to keep track of the summation ranges of individual indices, and, second, we demonstrate that for the main terms separating the contribution outside of the neighbourhoods $\NN$ is not necessary, all estimates on $\cM$
would also hold for the unrestricted sum. In particular, the neighbourhood decay condition is not necessary for the main terms, they are used only for bounding $\cM'$ in terms of $\cM$ in Section \ref{sec neighbourhood errors}. This fact was already advertised in Example \ref{example gauss} where we claimed that in the Gaussian case we can considerably relax our decay conditions. Later, in Section \ref{sec mods} we will explain how to elevate the proof for the special case $L_3=L_4=\emptyset$ with extended index sets to the general case.

\subsection{Averaged bound on \texorpdfstring{$D$}{D}}\label{av bound sec}
To treat \eqref{E prod Lambda tuples} systematically, we introduce a graphical representation for any $M$, $L$ and permutations $\sigma$ in \eqref{E prod Lambda tuples}. For the averaged local law
we need averaged estimates on $D$, so we set 
 \[ 
 \Lambda^{(k)}(D)\defeq \braket{BD}\qquad \mbox{or} \qquad \Lambda^{(k)}(D)\defeq \overline{\braket{BD}},
 \]
  where $B$ is a generic norm-bounded matrix, $\norm{B}\lesssim 1$ and we recall that $\braket{\cdot}=N^{-1}\Tr$ denotes the normalized trace. A factor $\Lambda_{\alpha_1,\dots,\alpha_n}$ can be represented as a directed cyclic graph on the vertex set $\{\alpha_1,\dots,\alpha_n\}$. Up to sign we have 
\begin{align}\label{Lambda BD}\abs{\Lambda_{\alpha_1,\dots,\alpha_n}}= N^{-n/2}\braket{ B\Delta^{\alpha_1}G\Delta^{\alpha_2} G\dots \Delta^{\alpha_n} G }= N^{-1-n/2} G_{b_1a_2}G_{b_2a_3}\dots G_{b_{n-1}a_n}(GB)_{b_na_1},\end{align}
which we represent as a cyclic graph in such a way that the vertices represent labels $\alpha_i=(a_i,b_i)$ and a directed edge from $\alpha_i=(a_i,b_i)$ to $\alpha_j=(a_j,b_j)$ represents $G_{b_ia_j}$. Since we will always draw the graphs in a clockwise orientation we will not indicate the direction of the edges specifically. The specific $GB$ factor will be denoted by a wiggly line instead of a straight line used for the $G$ factors. As an example, we have the correspondences
\[\Lambda_{\alpha_1,\alpha_2,\alpha_3,\alpha_4} \leftrightarrow\plotLambdas{"$\alpha_1$"--"$\alpha_2$"--"$\alpha_3$"--"$\alpha_4$"--[decorate,decoration={snake,amplitude=.3mm,segment length=.6mm}]"$\alpha_1$"}, \quad\Lambda_{\alpha_1,\alpha_2}\leftrightarrow \plotLambdas{"$\alpha_1$" --[bend left] "$\alpha_2$"; "$\alpha_1$" --[bend right,decorate,decoration={snake,amplitude=.3mm,segment length=.6mm}] "$\alpha_2$"  } \quad\text{and}\quad \Lambda_{\alpha_1} \leftrightarrow \plotLambdas{"$\alpha_1$" --[loop below, my loop, decorate,decoration={snake,amplitude=.3mm,segment length=.6mm}] "$\alpha_1$" }.\]

In \eqref{E prod Lambda tuples} the labels of type $l$ are connected through the $\kappa(\alpha_l,\bm\beta_l)$ factor which strongly links those labels due to the decay properties of the cumulants. We represent this fact graphically as a vertex colouring of the graph  in which label types correspond to colours. The set of colours representing the label types $L$ will be denoted by $C$. The $M_l+1$ vertices of a given type $l$ will be denoted by $V_c$, where $c$ is the colour corresponding to $l$. 

We define $\Val(\Gamma)$, the \emph{value} of a graph $\Gamma$, as summation over all labels consistent with the colouring, such that equally coloured labels are linked through a cumulant, of the product of the corresponding $\Lambda$'s, just as in \eqref{E prod Lambda tuples}. For example, we have
\begin{align}\label{Lambda1221 ex av} \sum_{\alpha_1,\beta_1^2(1)} \kappa(\alpha_1,\beta_1^2(1))\sum_{\alpha_2,\beta_2^1(1)} \kappa(\alpha_2,\beta_2^1(1)) \Lambda_{\alpha_1,\beta_2^1(1)} \Lambda_{\alpha_2,\beta_1^2(1)} = \Val\left(\plotLambda{{1,2},{2,1}}\right) \end{align} 
or
\begin{align*} \sum_{\alpha_1,\beta_1^2(1)} \kappa(\alpha_1,\beta_1^2(1))\sum_{\alpha_2,\beta_2^1(1),\beta_2^1(2)} \frac{\kappa(\alpha_2,\beta_2^1(1),\beta_2^1(2))}{2!} \sum_{\alpha_3,\beta_3^2(1)} \kappa(\alpha_3,\beta_3^2(1)) &\Lambda_{\alpha_1,\beta_2^1(2),\beta_2^1(1)} \Lambda_{\alpha_2,\beta_1^2(1)} \Lambda_{\alpha_3,\beta_1^3(1)}  \\
&= \Val \left(\plotLambda{{1,2,2},{2,3},{3,1}}\right),\end{align*}
where we choose the variable names for the labels in accordance with \eqref{E prod Lambda tuples} following the convention that the elements of the tuple $\bm\beta_l^k$ are denoted by $(\beta_l^k(1),\beta_l^k(2),\dots)$. We warn the reader that $\Val(\Gamma)$, the value of a diagram itself is a random variable unlike in customary Feynman diagrammatic expansion theory. In the following we will derive bounds on the value of diagrams. To separate the conceptual from the technical difficulties we first derive those bounds in a vague $\lesssim$ sense which ignores a technical subtlety: The entries $G_{ab}$ of the resolvent are bounded with overwhelming probability, but usually not almost surely. In the first conceptual step we will tacitly assume such an almost sure bound and write $\abs{G_{ab}}\lesssim 1$. Later  in Section \ref{sec detailed bound} we will make the bounds rigorous in a high-moment sense. We note that if $\Lambda(D)=\overline{\braket{BD}}$, then the edges would represent $G^\ast$ and $(GB)^\ast$ instead of $G$ and $GB$ and the order would be reversed (recall that the double indices are transposed in \eqref{Lambda_def}) but the counting argument is not sensitive to these nuances, so we omit these distinctions in our graphs.

We now rephrase the rules on $M$ in this graphical representation. They dictate that we need to consider the set of all vertex coloured graphs $\Gamma$ with cyclic components such that
\begin{enumerate}[(a)]
\item there exist $p$ connected components, all of which are cycles,
\item each connected component contains at least one vertex, 
\item each colour colours at least two vertices,
\item if a colour colours exactly two vertices, then these vertices are in different components.
\item for each colour there exists a component in which the vertex after the wiggled edge (in clockwise orientation) is of that colour.
\end{enumerate}
We note that these rules, compared to \eqref{E prod Lambda tuples}, disregarded the restrictions on the permutations $\sigma_k$ for $k\not \in L$ as these are not relevant for the averaged bound. The set of graphs satisfying (a)--(e) will be denoted by $\cG^{\text{av}}(p,R)$ and for each $L,M,\sigma$ the main term $\cM$ from \eqref{E prod Lambda tuples} is  given by the value of some graph $\Gamma\in\cG^{\mathrm{av}(p,R)}$.
\begin{align}\label{av graphical expansion}
\cM\Big( \big(\braket{B\cdot}^{[p/2]},\overline{\braket{B\cdot}}^{[p/2]}\big) ;L,M,\sigma \Big) = \Val(\Gamma),\qquad \Gamma=\Gamma(L,M,\sigma)\in\cG^{\mathrm{av}(p,R)}
\end{align}
where $\braket{B\cdot}^{[p/2]}$ denotes the tuple of $p/2$ functionals mapping $D\mapsto \braket{BD}$ and similarly for $\overline{\braket{B\cdot}}$. As the number of such graphs is finite for given $p,R$ it follows that it is sufficient to prove the required bound for every single graph.

As for any fixed colour $\sum_{\pB{1}}\le N^2 \tnorm{\kappa}^{\text{av}}$, the naive size of the value $\Val(\Gamma)$ is bounded by
\begin{align}\Val(\Gamma)\lesssim N^{-p} \prod_{c\in C} N^{2 -\abs{V_c}/2}\le 1\label{naive size averaged}\end{align}
since according to \eqref{Lambda BD} every component contributes a factor $N^{-1}$ and every label contributes a factor $N^{-1/2}$, and where the ultimate inequality followed from $\abs{V_c}\ge 2$ and $\abs{C}\le p$. We now demonstrate that using Ward identities of the form \[\sum_{a}\abs{G_{ab}}^2=\frac{(\Im G)_{bb}}{\eta}\] we can improve upon this naive size by a factor of $\psi^{2p}$, where $\psi \approx 1/\sqrt{N\eta}$ and $\eta\defeq \Im z$. We will often use the Ward
identity in the form
\begin{subequations}
\begin{align} \label{ward}
  \sum_b \abs{G_{ab}}\le \sqrt{N}\sqrt{\sum_b \abs{G_{ab}}^2}  = N\sqrt{\frac{(\Im G)_{aa}}{N\eta}} \lesssim N\psi, \qquad \sum_b \abs{(GB)_{ab}} \lesssim \norm{B} N\psi
\end{align}
which explicitly exhibits a gain of a factor $\psi$ over the trivial bound of order N. Together with the previous bound
\begin{align}\label{ward2}
   \sum_b \abs{G_{ab}}^2 \le N\psi^2, \qquad \sum_b \abs{(GB)_{ab}}^2 \lesssim \norm{B}^2 N\psi^2
\end{align}
\end{subequations}
we will call \eqref{ward}--\eqref{ward2} \emph{Ward estimates}. Here we used the trivial bound $\abs{G}\lesssim 1$ and we set $\psi\defeq \sqrt{\Im G/N\eta}$
(where $\Im G$ is meant in an isotropic sense which we will define rigorously later).

We consider the subset of colours $C'\defeq \Set{c\in C| \abs{V_c}\le 3}\subset C$ which colour either two or three vertices and we intend to use Ward identities only when summing up vertices with those colours. However, one may not use Ward estimates for every such summation, e.g.
even if both $a$ and $b$ were indices of eligible labels, one cannot gain from both of them
in the sum $\sum_{a,b} \abs{G_{ab}}$. We thus need a systematic procedure to identify sufficiently many labels so that  each summation over them can be performed  by using  Ward estimates. In the following, we first describe a procedure how to \emph{mark} those edges we can
potentially use for Ward estimates. Secondly, we will show that for sufficiently many marked edges
the Ward estimates can be used in parallel.

\subsubsection*{Procedure for colours appearing twice in $\Gamma$}
If a colour $\pB{2}$ appears twice, then it appears in two different components of $\Gamma$, i.e., in one of the following forms
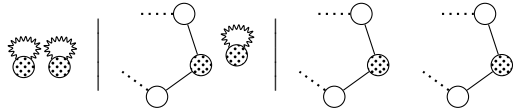
\[\plotLambda{{2},{2}}\Bigg| \plotLambda{{.,0,2,0},{2}}\Bigg|\plotLambda{{.,0,2,0},{.,0,2,0}}\]
where the white vertices can be of any colour other than $\pB{2}$ (and may even coincide), the dotted edges indicate an arbitrary continuation of the component and some additional edges may be wiggled. The picture only shows those two components with colour $\pB{2}$, the other components of $\Gamma$ are not drawn. Vertical lines separate different cases. When summing up the $\pB{2}$-coloured labels, we can use the Ward estimates on all edges adjacent to $\pB{2}$ using the operator norm $\tnorm{\kappa}_2^{\text{av}}= \norm[1]{\abs{\kappa(\ast,\ast)}}$ on $\kappa$. To see this we note that 
\begin{align}\label{ex est} \sum_{\alpha_1,\alpha_2} \abs{\kappa(\alpha_1,\alpha_2) A_{\alpha_1} B_{\alpha_2}} \le \tnorm{\kappa}_{2}^{\text{av}} \sqrt{\sum_{\alpha_1} \abs{A_{\alpha_1}}^2}\sqrt{\sum_{\alpha_2}\abs{B_{\alpha_2}}^2},\end{align}
after which \eqref{ward2} with $A_{\alpha_1},B_{\alpha_1}\in\{G_{b_1a_1},(GB)_{b_1a_1}, G_{c a_1}G_{b_1 d}, (GB)_{c a_1}G_{b_1 d}, G_{c a_1}(GB)_{b_1 d} \}$ and arbitrary fixed indices $c,d$ is applicable. 
\begin{remark}
In the sequel we will not write up separate estimates for edges representing $GB$ instead of $G$ as the same Ward estimates \eqref{ward}--\eqref{ward2} hold true and the bound is automatic in the sense that there are in total $p$ wiggly edges in $\Gamma$, each of which will contribute a factor of $\norm{B}$ to the final estimate, regardless of whether the corresponding edge has been bounded trivially $\abs{(GB)_\alpha}\lesssim \norm{B}$ or by \eqref{ward}--\eqref{ward2}.
\end{remark}
We find that for every edge connected to $\pB{2}$ we can gain a factor $\psi$ compared to the naive size of the $\pB{2}$-sum, using only the trivial bound $\abs{G}\lesssim 1$. We will indicate visually that
an edge has potential for a gain of $\psi$ through some colour by putting a mark (a small arrow) pointing from the vertex towards the edge. Thus in the case where $\pB{2}$ appears twice we  mark all edges adjacent to $\pB{2}$ to obtain the following marked graphs
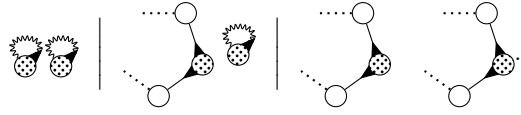
\[\plotLambda{{2m},{2m}}\Bigg| \plotLambda{{.,0,m2m,0},{2m}} \Bigg| \plotLambda{{.,0,m2m,0},{.,0,m2m,0}}.\]
We note that these marks indicate that we can use a Ward estimate for every marked edge, when performing the $\pB{2}$-summation, while keeping all other labels fixed. When simultaneously summing over labels from different colours it is not guaranteed any more that we can perform a Ward estimate for every marked edge. We will later resolve this possible issue by introducing the concept of \emph{effective} and \emph{ineffective} marks.  

\subsubsection*{Procedure for colours appearing three times in $\Gamma$} If a colour $\pB{2}$ appears three times, then the following ten setups are possible 
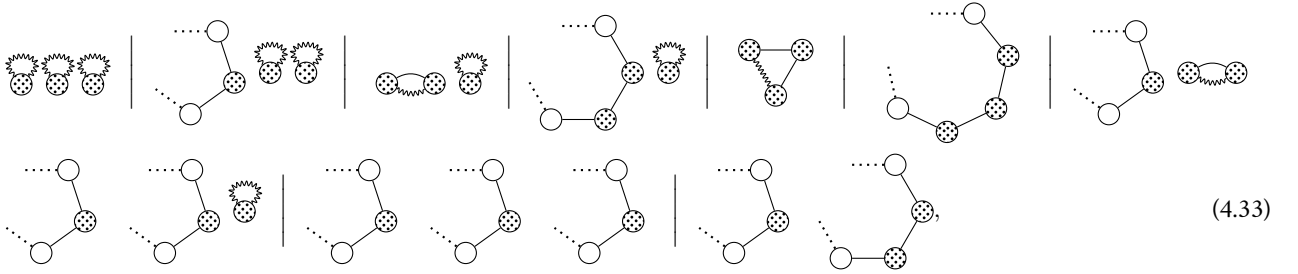
\begin{align}\nonumber
&\plotLambda{{2},{2},{2}}\Bigg|\plotLambda{{.,0,2,0},{2},{2}}\Bigg|\quad\plotLambda{{2,2},{2}}\Bigg|\plotLambda{{.,0,2,2,0},{2}}\Bigg| \quad\plotLambda{{2,2,2}}\quad\Bigg|\quad\plotLambda{{.,0,2,2,2,0}}\quad\Bigg| \plotLambda{{.,0,2,0},{2,2}}\\
&\plotLambda{{.,0,2,0},{.,0,2,0},{2}}\Bigg| \plotLambda{{.,0,2,0},{.,0,2,0},{.,0,2,0}}\quad\Bigg| \plotLambda{{.,0,2,0},{.,0,2,2,0}},\label{graphs 3 av}
\end{align}
where we explicitly allow components with open continuations to be connected (unlike in the previous case, where rule \eqref{rule cancellation} applied). We now mark the edges adjacent to $\pB{2}$ as follows and observe that at most two remain unmarked. Explicitly we choose the markings
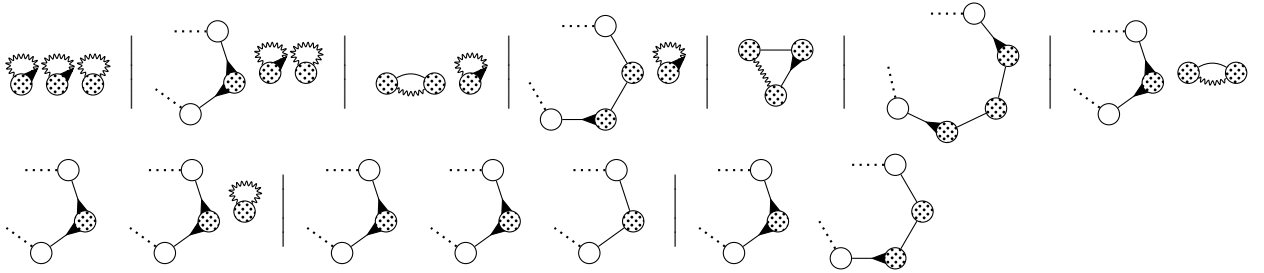
\begin{align*}
&\plotLambda{{2m},{2m},{2}}\Bigg|\plotLambda{{.,0,m2m,0},{2m},{2}}\Bigg|\quad\plotLambda{{2,2},{2m}}\Bigg|\plotLambda{{.,0,2,2m,0},{2m}}\Bigg| \quad\plotLambda{{2,2m,2}}\quad\Bigg|\quad\plotLambda{{.,0,m2,2,2m,0}}\quad\Bigg| \plotLambda{{.,0,m2m,0},{2,2}}\\
&\plotLambda{{.,0,m2m,0},{.,0,m2m,0},{2}}\Bigg| \plotLambda{{.,0,m2m,0},{.,0,m2m,0},{.,0,2,0}}\quad\Bigg| \plotLambda{{.,0,m2m,0},{.,0,2,2m,0}}
\end{align*}
and observe that in all but the fifth graph we can gain a factor of $\psi$ for every marked edge using the first term in the norm $\tnorm{\kappa}_3^{\text{av}}$. For example, in the second graph this follows from 
\[\sum_{\alpha_1,\alpha_2,\alpha_3} \abs{\kappa(\alpha_1,\alpha_2,\alpha_3)  G_{c a_3} G_{b_3 d} G_{b_1 a_1}G_{b_2a_2}} \lesssim \tnorm{\kappa}_3^{\text{av}} \sqrt{\sum_{\alpha_2} \abs{G_{b_2a_2}}^2 }\sqrt{\sum_{\alpha_3} \abs{G_{c a_3} G_{b_3 d}}^2 } \lesssim \tnorm{\kappa}_3^{\text{av}} N^2\psi^3\]
and in third graph from 
\[\sum_{\alpha_1,\alpha_2,\alpha_3} \abs{\kappa(\alpha_1,\alpha_2,\alpha_3) G_{b_1 a_2}G_{b_2a_1} G_{b_3a_3}} \lesssim \sum_{\alpha_2,\alpha_3} \abs{G_{b_3a_3}} \sum_{\alpha_1}\abs{\kappa(\alpha_1,\alpha_2,\alpha_3)}  \lesssim \tnorm{\kappa}_3^{\text{av}}N^2\psi,\]
where $c$ and $d$ are the connected indices from the white vertices in the graph. The computations for the other graphs are identical. We note that the markings we chose above are not the only ones possible. For example we could have replaced
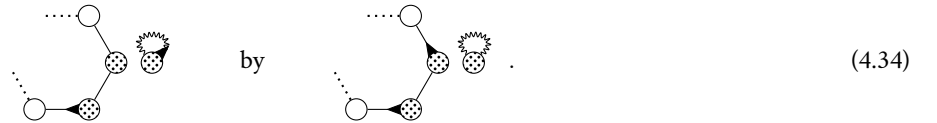
\begin{align}
\plotLambda{{.,0,2,2m,0},{2m}} \qquad \text{by}\qquad \plotLambda{{.,0,m2,2m,0},{2}}.\label{changed marking ex}
\end{align}

For the fifth graph in \eqref{graphs 3 av} the second term in the $\tnorm{\kappa}_3^{\text{av}}$ is necessary. The norms in \eqref{kappa ef norms} ensure that we can perform at least one Ward estimate and we have 
\[ \sum_{\alpha_1,\alpha_2,\alpha_3} \abs{\kappa(\alpha_1,\alpha_2,\alpha_3) G_{b_1a_2}G_{b_2a_3}G_{b_3a_1}} \lesssim \tnorm{\kappa}_3^{\text{av}} N^2 \psi.\]
Indeed, for example 
\[\sum_{\alpha_1,\alpha_2,\alpha_3} \abs{\kappa_{cd}(\alpha_1,\alpha_2,\alpha_3) G_{b_1a_2}G_{b_2a_3}G_{b_3a_1}} \lesssim \sum_{\alpha_1,\alpha_2,\alpha_3} \abs{\kappa_{cd}(\alpha_1,\alpha_2,\alpha_3) G_{b_3a_1}} \le \tnorm[0]{\kappa_{cd}}_{cd} N \sqrt{\sum_{b_3,a_1} \abs{G_{b_3a_1}}^2}\lesssim \tnorm[0]{\kappa_{cd}}_{cd} N^2\psi\]
and the other three cases are similar.

\subsubsection*{Procedure for all other colours in $\Gamma$}For colours in $C\setminus C'$, i.e., those which appear four times or more, we do not intend to use any Ward estimates and therefore we do not place any additional markings. Thus we only have to control the size of the summation over any fixed colour, as is guaranteed by the finiteness of $\tnorm{\kappa}_k^{\text{av}}$.

\subsubsection*{Counting of markings} After we have chosen all markings, we select the ``useful'' ones.  We call an edge \emph{ineffectively marked} if it only carries one mark and joins two distinctly $C'$-coloured vertices. All other marked edges we call \emph{effectively marked} because the parallel gain through a Ward estimate is guaranteed for all those edges. In total, there are at least $\sum_{c\in C'}\abs{V_{c}}$ edges adjacent to $C'$ (i.e., adjacent to a $C'$-coloured vertex). After the above marking procedure there are at most $2\sum_{c\in C'}(\abs{V_c}-2)$ unmarked or ineffectively marked edges adjacent to $C'$. To see this we note that edges between two $C'$-colours with only one marking are counted as unmarked from the perspective of exactly one of the two colours. Thus we find that there are at least 
\begin{align} \sum_{c\in C'}\abs{V_c}-2\sum_{c\in C'}(\abs{V_c}-2)=\sum_{c\in C'}(4-\abs{V_c}) \label{marking count average}\end{align}
effectively marked edges adjacent to $C'$ after the marking procedure. We illustrate this counting in an example. In the graph
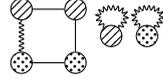
\[\plotLambda{{1,1,2,2},{1},{2}}\]
we have $V_{\pB{1}}=V_{\pB{2}}=3$ and there are six edges adjacent to $C'=\{\pB{1},\pB{2}\}$. After the marking procedure we could for example obtain the graphs
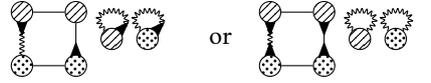
\[ \plotLambda{{m1,1,m2,2},{1m},{2m}}\quad\text{or}\quad\plotLambda{{m1,1m,m2,2m},{1},{2}},\]
where the second graph would result from the replaced marking in \eqref{changed marking ex}. In both cases there are two effectively marked edges, in accordance with \eqref{marking count average}; in the first example there are also two ineffectively marked edges. 

\subsubsection*{Power counting estimate} 
The strategy now is that we iteratively perform the Ward estimates colour by colour in $C'$ in no particular order. In each step we thus remove all the edges adjacent to some given colour, either through Ward estimates (if the edge was marked in that colour), or through the trivial bound $\abs{G_{\alpha}}\lesssim 1$. If some edge is missing because it already was removed in a previous step, then the corresponding $G$ is replaced by $1$ in that estimate (e.g.~in \eqref{ex est}). This might reduce the number of available Ward estimates in some steps, but the concept of effective markings ensures that whenever an effectively marked edge is removed, then a gain through a Ward estimate is guaranteed.  After the summation over all colours from $C'$ we have thus performed Ward estimates in all the effectively marked edges, which amounts to at least
\[\sum_{c\in C'} (4-\abs{V_c})\]
gains of the factor $\psi$. We note that ineffectively marked edges may not be estimated by a Ward estimates, as it might be necessary to bound the corresponding $G$ trivially while performing the sum over another colour. Using only the gains from the effective marks, we can improve on the naive power counting \eqref{naive size averaged} to conclude that the value of $\Gamma$ is bounded by
\begin{equation}\label{av val bound sim}\Val(\Gamma)\lesssim N^{-p}\prod_{c\in C\setminus C'}N^{2-\abs{V_c}/2}\prod_{c\in C'}(N\psi^2)^{2-\abs[0]{V_c}/2} \le \psi^{2p} N^{2\abs{C\setminus C'}-\abs[0]{V_{C\setminus C'}}/2} ,\end{equation}
where we used that $\abs{C'}\le\abs{C}\le p$, $\abs{V_c}=2,3$ for $c\in C'$ and $\abs{V_c}\ge 4$ otherwise, and that $N\psi^2\ge 1$. 
 Note that this last bound used $\braket{z}\le C$ and $\| H\|\le C$, without these conditions we have
$N\psi^2\ge (\|H\|^2+ \braket{z}^2)^{-1}\ge \braket{\| H\|}^{-2} \braket{z}^{-2}$. 

\subsubsection{Detailed bound}\label{sec detailed bound} The argument above tacitly assumed bounds of the form $\abs{G_{\alpha}}\lesssim 1$ and $\sum_{\alpha}\abs{G_\alpha}^2\lesssim N^2\psi^2$. Apart from unspecified and irrelevant constants, these bounds are not available almost surely, they hold only in the sense of high moments, e.g. $\E\abs{G_\alpha}^q\le_q 1$. Secondly, the definition of $\psi$ intentionally left the role of $\Im G$ in it vague. The precise definition of $\psi$ will involve high $L^q$ norms of $\Im G$. Moreover, different $G$-factors in the monomials $\Lambda$ are not independent. All these difficulties can be handled by the following general H\"older inequality. Suppose, we aim at estimating 
\[\E \sum_A X_A \sum_B Y_{A,B}\]
for random variables $X_A$, $Y_{A,B}$, then we use the H\"older inequality to estimate 
\begin{align} \norm[2]{\sum_A X_A \sum_B Y_{A,B}}_q\le \bigg(\sum_A\bigg)^\epsilon \norm[2]{\sum_A X_A}_{2q} \max_A\norm[2]{\sum_B Y_{A,B}}_{1/\epsilon} \label{eq triv bound}\end{align}
for $0<\epsilon\le 1/2q$. In our procedure \eqref{eq triv bound} enables us to iteratively bound the graphs colour by colour at the expense of an additional factor $N^{2pR\epsilon}$ in every colour step of the bound, as the total sum is at most of size $N^{2pR}$. To estimate a $G$ or an $\Im G$ directly we use the H\"older inequality and note that there are  at most $\abs{V}=\sum_c \abs{V_c}\le pR$ factors of the form $G$ or $GB$, so that we can estimate those terms isotropically by $\norm{G}_{pR/\epsilon}$, $\norm{B}\norm{G}_{pR/\epsilon}$ and $\norm{\Im G}_{pR/\epsilon}$. We use \eqref{eq triv bound} at most with $q\in\{1,2,4,\dots,2^{p-1}\}$ and thus have a restriction of $0<\epsilon\le 2^{-p}$. Thus, combining the power counting above with the iterated application of the H\"older inequality, we have shown that 
\begin{align}\label{av bound proof} 
\norm{\Val(\Gamma)}_p \le_{p,R,\epsilon}  N^{2 p^2 R \epsilon} \Big(1+\tnorm{\kappa}^{\text{av}}\Big)^p\norm{B}^{p}(1+\norm{G}_{pR/\epsilon}^{\abs{V}})\psi_{pR/\epsilon}^{2p}N^{2\abs{C\setminus C'}-\abs[0]{V_{C\setminus C'}}/2}, \quad \psi_{q}\defeq \sqrt{\frac{\norm{\Im G}_q}{N\eta}} \end{align}
for all $\Gamma\in\cG^{\text{av}}(p,R)$ and $0<\epsilon\le 1/2^p$. Therefore, together with \eqref{av graphical expansion} we conclude the bound 
\begin{align}\label{av bound proof final} 
\norm{\cM\Big( \big(\braket{B\cdot}^{[p/2]},\overline{\braket{B\cdot}}^{[p/2]}\big) ;L,M,\sigma \Big)}_p \le_{p,R,\epsilon}  N^{2 p^2 R \epsilon} \Big(1+\tnorm{\kappa}^{\text{av}}\Big)^p\norm{B}^{p}(1+\norm{G}_{pR/\epsilon}^{\abs{V}})\psi_{pR/\epsilon}^{2p}N^{2\abs{C\setminus C'}-\abs[0]{V_{C\setminus C'}}/2}\end{align}
on \eqref{E prod Lambda tuples}.

\subsection{Isotropic bound on \texorpdfstring{$D$}{D}}\label{iso bound sec}
We turn to the isotropic bound on $D$, i.e.~we give bounds on \eqref{E prod Lambda tuples} with functionals $\Lambda$ of the
following type. We consider fixed vectors $\vx,\vy$ and set $\Lambda(D)=D_{\vx\vy}$ or $\Lambda(D)=\overline{D_{\vx\vy}}$. Up to sign we then have 
\begin{align}
\abs{\Lambda_{\alpha_1,\dots,\alpha_n}} = N^{-n/2} (\Delta^{\alpha_1}G\dots\Delta^{\alpha_n}G)_{\vx\vy} = N^{-n/2} \vx_{a_1} G_{b_1a_2}\dots G_{b_{n-1}a_n} G_{b_n\vy}.
\end{align}
The graph component representing $\Lambda_{\alpha_1,\dots,\alpha_n}$ is a chain in contrast to the cycles in the averaged case. We also have additional \emph{edges} representing the first $\vx_{a_1}$ and last $G_{b_n\vy}$ factor which we will picture as $\plotlLambdas{1[xx,as=]--2[Gend,draw=none,as=]}$and $\plotlLambdas{1[Gend,draw=none,as=]--2[Gend,as=]}$, respectively. These are special edges that are adjacent to one vertex only (the dots $\bullet$ and $\circ$ are not considered as vertices). We will call them \emph{initial} and \emph{final} edge. Due to these special edges we should, strictly speaking, talk about a special class of hypergraphs consisting of a union of chains each of them starting and ending with such a special edge, but for simplicity we continue to use the term \emph{graph}. For example we have the correspondence
\begin{align}\Lambda_{\alpha_1,\alpha_2}\leftrightarrow\plotlLambdas{1[xx,as=]--"$\alpha_1$" --"$\alpha_2$" --  2[Gend,as=]}.\end{align}
For $\Lambda(D)=\overline{D_{\vx\vy}}$ the edges represent $\overline{x_{a_1}}$, $G^\ast_{b_k\vy}$ and $G^\ast_{b_ka_{k+1}}$ but we do not indicate complex and Hermitian conjugate visually as they have no consequences on the argument. We follow the same convention regarding the colouring, as we did in the averaged case and for example have the representation
\begin{align*} \sum_{\alpha_1,\beta_1^2(1)} \kappa(\alpha_1,\beta_1^2(1))\sum_{\alpha_2,\beta_2^1(1),\beta_2^1(2)} \frac{\kappa(\alpha_2,\beta_2^1(1),\beta_2^1(2))}{2!} \sum_{\alpha_3,\beta_3^2(1)} \kappa(\alpha_3,\beta_3^2(1)) &\Lambda_{\alpha_1,\beta_2^1(2),\beta_2^1(1)} \Lambda_{\alpha_2,\beta_3^2(1)} \Lambda_{\alpha_3,\beta_1^3(1)}  \\
&= \Val \left(\plotlLambda{{1,2,2},{2,3},{3,1}}\right).\end{align*}

We again rephrase the rules on $M$ as rules on the graph $\Gamma$. We consider all vertex coloured graphs $\Gamma$ such that the connected components are chains with an initial edge of type $\plotlLambdas{1[xx,as=]--2[Gend,as=,draw=none]}$and a final edge of type $\plotlLambdas{1[Gend,draw=none,as=]--2[Gend,as=]}$ such that
\begin{enumerate}[(a)]
\item there exist $p$ connected components, all of which are chains, 
\item every component contains at least one vertex,
\item every colour occurs at least once on a vertex adjacent to $\plotlLambdas{1[xx,as=]--"$,$"[draw=none]}$
\item every colour occurs at least twice,
\item if a colour occurs exactly twice, then it occurs in two different chains.
\end{enumerate}
The set of graphs satisfying (a)--(e) will be denoted by $\cG^{\text{iso}}(p,R)$ and for each $L,M,\sigma$ in \eqref{E prod Lambda tuples} we can write the main term $\cM$ as
\begin{align}\label{iso graphical expansion}
\cM\Big( \big(\braket{\vx,\cdot\vy}^{[p/2]},\overline{\braket{\vx,\cdot\vy}}^{[p/2]}\big) ;L,M,\sigma) = \Val(\Gamma),\qquad \Gamma=\Gamma(L,M,\sigma)\in\cG^{\mathrm{iso}(p,R)}
\end{align}
where $\braket{\vx,\cdot\vy}^{[p/2]}$ denotes the tuple of $p/2$ functionals mapping $D\mapsto \braket{\vx,D\vy}$ and similarly for $\overline{\braket{\vx,\cdot\vy}}$. As the number of such graphs is finite for given $p,R$ it follows that it sufficient to prove the required bound for every single graph.

In contrast to the averaged case, where each $\Lambda$ carried a factor $1/N$ from the definition of $\Lambda(D)=N^{-1}\Tr BD$, now the naive size of the sum over $\Gamma$ is not of order $1$, but of order 
\begin{align}\Val(\Gamma)\lesssim\prod_{c\in C}N^{2-\abs{V_c}/2} = N^{2\abs{C}-\abs{V}/2},\label{eq very naive isotropic}\end{align}
which can be large. Consequently we have to be more careful in our bound and first make use of a cancellation. 
\subsubsection*{Step 1: Improved naive size} We first observe that we can reduce the naive size \eqref{eq very naive isotropic} to order $1$, without using any Ward estimates, yet. 
The improvement  comes from the fact that sums of the type
\[
  \sum_a v_a G_{ab} = G_{\vv b}
\]
can be directly bounded via the right hand side  by $\abs{G_{\vv b} } \lesssim \norm{\vv}$ using the isotropic bound.
Note that the naive estimate on the left hand side would be
\[
   \abs{\sum_a v_a G_{ab} }\lesssim \sum_a \abs{v_a} \le \sqrt{N} \norm{\vv}
\]
and even with a Ward estimate it can only be improved to
\[
  \abs{ \sum_a v_a G_{ab}} \le \norm{ \vv} \sqrt{\sum_a \abs{G_{ab}}^2}  \le  \sqrt{N} \psi \norm{\vv}
\]
So the procedure ``summing up a vector $\vv$ into the argument of $G$'' is much
more efficient than a Ward estimate. The limitation of this idea is that only deterministic
vectors $\vv$ can be summed up, since isotropic bounds on $G_{\vu\vv}$ hold only for fixed
vectors $\vu, \vv$.
\subsubsection*{Improvement for colours occurring twice in $\Gamma$}
For colours which occur exactly twice we can sum up the $\vx$ into a $G$ factor without paying the price of an $N$ factor from this summation. To do so, we consider an arbitrary partition of $\kappa=\kappa_c+\kappa_d$, where one should think of that $\kappa_d(\alpha_1,\alpha_2)$ forces $\alpha_1=(a_1,b_1)$ to be close to $\alpha_2=(a_2,b_2)$, whereas $\kappa_d(\alpha_1,\alpha_2)$ forces $(a_1,b_1)$ to be close to $(b_2,a_2)$. In both cases we can, according to rule (b), perform two single index summations as follows. First, we sum up the index $a_1$ of $\vx$ as \[\sum_{a_1}\kappa(a_1b_1,a_2b_2)x_{a_1}=\kappa(\vx b_1,a_2b_2).\]
Then we sum up its companion $b_2$ or $a_2$, depending on whether we consider the cross or direct term:
\[\sum_{b_2}\kappa_c(\vx b_1,a_2b_2) G_{b_2\vv} =  G_{\kappa_c(\vx b_1,a_2\cdot)\vv}\quad\text{or}\quad \sum_{a_2}\kappa_d(\vx b_1,a_2b_2) G_{\vv a_2} =  G_{\vv\kappa_d(\vx b_1,\cdot b_2)}, \]
where $\vv$ can be any vector or index. Thus we effectively performed a single label (two index) summation into a single $G$ factor that will be estimated by a constant in the isotropic norm. We indicate this summation graphically by introducing half-vertices $\plotlLambdas{2[Gf,as=\nodepart{lower}$a$,circle split part fill={},rotate=180]}$and$\plotlLambdas{1[Gf,as=\nodepart{lower}$b$,circle split part fill={}]}$ representing the single leftover indices $a$ and $b$ corresponding to a label $\alpha=(a,b)$ and  new (half)edges $\plotlLambdas{1[xx,as=,fill=white]--2[draw=none,as=]}$and $\plotlLambdas{1[xx,draw=none,as=,fill=white]--2[xx,as=,fill=white]}$ representing the $G_{\kappa_c(\vx b_1,a_2\cdot)\vv}$ and $G_{\vv\kappa_d(\vx b_1,\cdot b_2)}$ factors. To indicate that the half-edges representing $\vx$ have been summed, we grey them out. This partial summation can thus be graphically represented as
\[\Val\left(\plotlLambda{{1,0,.},{.0,1,0,.}}\right)= \Val\left(\plotlLambda{{1x,0,.},{.0,1c,0,.}}\right) + \Val\left(\plotlLambda{{1x,0,.},{.0,1d,0,.}}\right),\]
since
\begin{align*}
&\sum_{a_1,b_1,a_2,b_2} \kappa(a_1b_1,a_2b_2) \left(\vx_{a_1} G_{b_1\vu}\right) \left(G_{\vv a_2}G_{b_2\vw}\right) = \sum_{b_1,a_2}  \left(G_{b_1\vu}\right) \left(G_{\vv a_2}G_{\kappa_c(\vx b_1,a_2\cdot)\vw}\right) + \sum_{b_1,b_2}  \left(G_{b_1\vu}\right) \left(G_{\vv \kappa_d(\vx b_1,\cdot b_2)}G_{b_2\vw}\right) 
\end{align*}
where $\vu,\vv,\vw$ are the connecting indices from the white vertices.

\subsubsection*{Improvement for colours occurring three times in $\Gamma$}
For colours which appear exactly three times we cannot perform the summation of $\vx$ directly. We can, however use a Cauchy-Schwarz in the vertex adjacent to the $\vx$--edge to improve the naive size of the ${\pB{1}}$--sum to $N^{3/2}$ from $N^2$. Explicitly, for any index or vector $\vv$ we use that 
\[\sum_{a_i,b_i} \abs{\vx_{a_i} G_{b_i\vv}}\lesssim \sum_{a_i,b_i} \abs{\vx_{a_i}}\le N^{3/2}\left(\sum_{a_i}\abs{\vx_{a_i}}^2\right)^{1/2}= N^{3/2} \norm{\vx}. \]
 To indicate the intend to use the Cauchy-Schwarz improvement on a specific $\vx$ edge, we mark the corresponding edge with a marking originating in the adjacent vertex, very much similar to the marking procedure in the averaged case. To differentiate this marking from those indicating the potential for a Ward estimate we use a grey marking $\plotlLambdas{ 10[xx,as=,grow right,draw=none] --[-{<[sep=-3pt,length=8pt,fill=lightgray,color=black]}] 11[as=,fill=white] ;}$. As an example we would indicate 
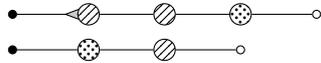
\[\plotlLambda{{m1,1,2},{2,1}}.\]

After these two improvements over \eqref{eq very naive isotropic} the naive size (naive in the sense without any Ward estimates, yet) of the summed graph is 
\begin{align}\Val(\Gamma)\lesssim\left(\prod_{c\in C, \abs{V_c}=2} N^{1-\abs{V_c}/2}\right)\left(\prod_{c\in C, \abs{V_c}=3} N^{3/2-\abs{V_c}/2}\right)\left(\prod_{c\in C, \abs{V_c}\ge4} N^{2-\abs{V_c}/2}\right)\le 1.\label{eq naive size iso}\end{align}
Notice that the first two factors give $1$, so the improved power counting for colours with two or three occurrences is neutral. We thus restored the order $1$ bound and can now focus on the counting of Ward estimates, with which we can further improve the bound. 

\subsubsection*{Step 2: Further improvements through Ward estimates} The counting procedure is very similar to what we used in the averaged law in the sense that we mark potential edges for Ward estimates colour by colour. To be consistent with the improved naive bound we count the grey initial edges (those from the summation of colours occurring twice) and the initial edges with a grey arrow (those from the summation of colours appearing three times) as unmarked, since they will not be available for Ward estimates. 

\subsubsection*{Marking procedure for colours occurring twice}
Colours occurring twice can, after Step 1, only occur in the reduced forms
\[\plotlLambda{{1x,0,.},{.0,1c,0,.}} \quad\text{and}\quad \plotlLambda{{1x,0,.},{.0,1d,0,.}}, \]
where we allow $\plotlLambdas{1[xx,draw=none,as=,fill=white] --[dotted,thick] 2[G,as=,fill=white] -- 3[xx,draw=none,as=,fill=white]}$ and $\plotlLambdas{1[xx,draw=none,as=,fill=white] -- 2[G,as=,fill=white] --[dotted,thick] 3[xx,draw=none,as=,fill=white]}$ to stand for an arbitrary continuation of the graph, as well as the initial $\plotlLambdas{1[xx,as=]--2[Gend,draw=none,as=]}$ and final edge $\plotlLambdas{1[Gend,draw=none,as=]--2[Gend,as=]}$. In both cases we mark the edges adjacent to the remaining two half-vertices to obtain:
\[\plotlLambda{{1xm,0,.},{.0,m1c,0,.}} \quad\text{and}\quad \plotlLambda{{1xm,0,.},{.0,1dm,0,.}} \]
Thus for colours appearing twice we always leave two edges unmarked (which includes the greyed out initial edge). Using the $\tnorm{\kappa}_2^{\text{iso}}$ norm we indeed find that the solid edges in the two graphs above can be bounded by
\begin{align}\sum_{b_1,a_2} \abs{G_{b_1\vu} G_{\vv a_2}G_{\kappa_c(\vx b_1,a_2\cdot)\vw}}\lesssim \norm{\vw}\sum_{b_1,a_2}\abs{G_{b_1\vu} G_{\vv a_2} \norm{\kappa_c(\vx b_1,a_2\cdot)}} \lesssim N^2\psi^2 \tnorm{\kappa}_2^{\text{iso}}\norm{\vx} \norm{\vu}\norm{\vv}\norm{\vw} \label{2 marking iso example}\end{align} 
and
\[\sum_{b_1,b_2} \abs{G_{b_1\vu}G_{\vv \kappa_d(\vx b_1,\cdot b_2)}G_{b_2\vw} }\lesssim \norm{\vv}\sum_{b_1,b_2}\abs{G_{b_1\vu} G_{b_2\vw} \norm{\kappa_d(\vx b_1,\cdot b_2)}} \lesssim N^2\psi^2 \tnorm{\kappa}_2^{\text{iso}}\norm{\vx} \norm{\vu}\norm{\vv}\norm{\vw} \]
where the vectors (which are allowed to be indices, as well) $\vu,\vv,\vw$ are the endpoints of the edges in the three white vertices. 
\begin{remark}
In the case that $\plotlLambdas{1[xx,draw=none,as=,fill=white] --[dotted,thick] 2[G,as=,fill=white] -- 3[xx,draw=none,as=,fill=white]}$ stands for the initial edge $\plotlLambdas{1[xx,as=]--2[Gend,draw=none,as=]}$, we cannot use the Ward estimate, but use instead that $\vx$ is a vector of finite norm, providing a gain of $N^{-1/2}\ll \psi$. For example, we could bound the graph
\[\plotlLambdas{ 10[xx,as=,grow right,fill=lightgray,draw=none] --[color=lightgray] 11[Gf,as=,circle split part fill={\csname pat1\endcsname}] --[{>[sep=-3pt,length=8pt]}-] 12[G,as=] --[dotted,thick] 13[Gend,draw=none,as=]; 20[xx,as=,grow right,] --[-{<[sep=-3pt,length=8pt]}] 21[Gc=\csname pat1\endcsname,rotate=180,circle split part fill={\csname pat1\endcsname},as=] -- 22[as=] --[dotted,thick] 23[Gend,draw=none,as=]; } \quad\text{by}\quad \sum_{b_1,a_2} \abs{ G_{b_1\vu} \vx_{a_2} G_{\kappa_c(\vx b_1,a_2\cdot)\vw}}\lesssim N^2\frac{\psi}{\sqrt{N}} \tnorm{\kappa}_2^{\text{iso}}\norm{\vx}^2 \norm{\vu}\norm{\vw},\]
which is better than \eqref{2 marking iso example} as $\sqrt N \psi\ge 1$. In the sequel we will not specifically distinguish this case when instead of a Ward estimate we have to use the finite norm of $\vx$, as the procedure is identical and the resulting bound is always smaller in the latter case. 

We also will not separately consider the case when $\plotlLambdas{1[xx,draw=none,as=,fill=white] -- 2[G,as=,fill=white] --[dotted,thick] 3[xx,draw=none,as=,fill=white]}$ stands for the final edge $\plotlLambdas{1[Gend,draw=none,as=]--2[Gend,as=]}$, as we can use the same Ward estimate as before, with the difference that $\vu$ and/or $\vw$ are replaced by $\vy$.

We will not manually keep track of the number of $\norm{\vx},\norm{\vy}$ in the  bound as it is automatic in the sense that there are $p$ initial and $p$ final edges in $\Gamma$, each contributing a factor of $\norm{\vx},\norm{\vy}$ to the final estimate.
\end{remark}

\subsubsection*{Marking procedure for colours occurring three times}
Colours appearing three times occur in one of the following four forms
\begin{align*}&\plotlLambda{{m1,1,1,0,.}}\quad\bigg|\quad\plotlLambda{{m1,1,0,.},{.0,1,0,.}}\\
&\rule{10cm}{0.4pt}\\
&\plotlLambda{{m1,0,.},{.0,1,1,0,.}}\quad\bigg|\quad\plotlLambda{{m1,0,.},{.0,1,0,.},{.0,1,0,.}} \end{align*}
and in all cases we mark the edges adjacent to $\pB{1}$ in such a way that at most three edges (including the initial edge with the grey mark) remain unmarked. Indeed, we mark the edges as follows.
\begin{align*}&\plotlLambda{{m1,1,1m,0,.}}\quad\bigg|\quad\plotlLambda{{m1,1,0,.},{.0,m1m,0,.}}\\
&\rule{10cm}{0.4pt}\\
&\plotlLambda{{m1m,0,.},{.0,m1,1,0,.}}\quad\bigg|\quad\plotlLambda{{m1m,0,.},{.0,m1m,0,.},{.0,1,0,.}}.\end{align*}
Very similar to the bound using $\tnorm{\kappa}_3^{\text{av}}$, we find that using the norm $\tnorm{\kappa}_3^{\text{iso}}$ we can perform Ward estimates on all marked edges.

\subsubsection*{Marking procedure for colours occurring more than three times} For any colour $c$ occurring more than three times we claim that we can always mark edges in such a way that at most $2\abs{V_c}-4$ edges adjacent to $V_c$ remain unmarked. Indeed, if we call an edge connected to two $c$--coloured vertices $c$--\emph{internal} and denote their set $E_{c}^{\text{int}}$, then there are $2\abs{V_c}-\abs[0]{E_{c}^{\text{int}}}$ edges adjacent to $c$. Out of this set of all $c$-adjacent edges, we mark any two and thus the claim is trivially fulfilled if $\abs[0]{E_c^{\text{int}}}\ge 2$. If $\abs[0]{E_c^{\text{int}}}=0$, then the graph contains two single vertices of colour $c$, for which we mark all four adjacent edges, i.e. \[\plotlLambda{{.0,m1m,0,.},{.0,m1m,0,.}},\] also confirming the claim in this case. Finally, if $\abs[0]{E_c^{\text{int}}}=1$, then the graph has to contain \[\plotlLambda{{.0,m1,1,0,.},{.0,m1m,0,.}},\] for which we mark the three indicated edges, confirming the claim also in this final case. We note (again similarly to $\tnorm{\kappa}_3^{\text{av}}$) that the norm $\tnorm{\kappa}_k^{\text{iso}}$ allows to perform Ward estimates on all marked edges. 

\subsubsection*{Counting of markings} In contrast to the averaged case, we now call an edge \emph{ineffectively marked} if it only carries one mark and connects any two distinctly coloured vertices (in the averaged case the analogous definition was restricted to $C'$-coloured vertices). All other marked edges we call \emph{effectively marked}. In particular the initial and final edge are always effectively marked, once they are marked. By construction, all effectively marked edges can be summed up by Ward estimates. In total, there are exactly $p+\sum_{c\in C}\abs{V_c}$ edges in $\Gamma$. After the marking procedure there are at most 
\[\sum_{c\in C,\abs{V_c}=2} 2 + \sum_{c\in C,\abs{V_c}=3} 3+\sum_{c\in C,\abs{V_c}\ge 4}(2\abs{V_c}-4) \] 
unmarked or ineffectively marked edges in $\Gamma$. Thus there are at least
\begin{align} \bigg(p+\sum_{c\in C}\abs{V_c}\bigg) -\bigg(\sum_{c\in C,\abs{V_c}=2} 2 + \sum_{c\in C,\abs{V_c}=3} 3+\sum_{c\in C,\abs{V_c}\ge 4}(2\abs{V_c}-4) \bigg)=p+\sum_{c\in C,\abs{V_c}\ge 4}(4-\abs{V_c})  \label{markings iso}\end{align}
effectively marked edges in $\Gamma$,  which can be negative, but it turns out that in this case the (improved) naive size already is sufficiently small. 

\subsubsection*{Power counting estimate}
The strategy for performing the Ward estimates is identical to that in the averaged case; we perform them colour by colour in an arbitrary order. According to the improved naive bound from Step 1, and recalling that the power counting for $\abs{V_c}=2$ and $\abs{V_c}=3$ gives $1$, i.e. is neutral, and the counting of additional effective markings we find that the summed value of $\Gamma$ is bounded by
\[\Val(\Gamma)\lesssim N^{2\abs{C\setminus C'}-\abs[0]{V_{C\setminus C'}}/2} \psi^{(p+4\abs{C\setminus C'}-\abs[0]{V_{C\setminus C'}})_+},\]
where $C'$ are those colours $c$ with $\abs{V_c}=2,3$.

\subsubsection*{Detailed estimate}
Finally, this power counting is performed with the procedure of iterated H\"older inequalities, exactly as in the averaged case to obtain 
\begin{align} \label{isotropic bound proof}
\norm{\Val(\Gamma)}_p \le_{\epsilon,R,p}  N^{2p^2 R\epsilon}\Big(1+\tnorm{\kappa}^{\text{iso}}\Big)^p \norm{\vx}^p\norm{\vy}^{p} (1+\norm{G}_{pR/\epsilon}^{\abs{V}}) \psi^{(p+4\abs{C\setminus C'}-\abs[0]{V_{C\setminus C'}})_+}_{pR/\epsilon} N^{2\abs{C\setminus C'}-\abs[0]{V_{C\setminus C'}}/2}\end{align}
for all $\Gamma\in\cG^{\text{iso}}(p,R)$ and $0<\epsilon\le 1/2^p$. Therefore we conclude together with \eqref{iso graphical expansion} that
\begin{equation} \label{isotropic bound proof final}
\begin{split}
\norm{\cM\Big( \big(\braket{\vx,\cdot\vy}^{[p/2]},\overline{\braket{\vx,\cdot\vy}}^{[p/2]}\big) ;L,M,\sigma \Big)}_p \le_{\epsilon,R,p}&  N^{2p^2 R\epsilon}\Big(1+\tnorm{\kappa}^{\text{iso}}\Big)^p \norm{\vx}^p\norm{\vy}^{p} \\ &\times (1+\norm{G}_{pR/\epsilon}^{\abs{V}}) \psi^{(p+4\abs{C\setminus C'}-\abs[0]{V_{C\setminus C'}})_+}_{pR/\epsilon} N^{2\abs{C\setminus C'}-\abs[0]{V_{C\setminus C'}}/2}.
\end{split}
\end{equation}

\subsection{Modifications for general case}\label{sec mods}
In the previous Sections \ref{av bound sec} and \ref{iso bound sec} we estimated $\cM$ defined in \eqref{E prod Lambda tuples} under the simplifying assumptions $L_3=L_4=\emptyset$ and $\NN(\alpha_l)=I$. For the final bound in \eqref{E prod Lambda cM bound} we need to treat all other cases. In this section we 
 now demonstrate that these simplifying assumptions \nc
  are not substantial and that the results from \eqref{marking count average} and \eqref{markings iso} on the number of available Ward estimates remain valid in the more general setting. By definition, $\cM$ depends on the labels of types $L_3$ and $L_4$, which are considered fixed in the subsequent discussion. The graphs we introduced to systematically bound $\cM$ do not change in their form for the general case, but only have additional \emph{fixed} vertices $\alpha_l$, $\bm\beta_l$ for $l\in L_3\cup L_4$, which we consider as uncoloured. Thus we enlarge the set graphs $\cG^{\text{av}}$ and $\cG^{\text{iso}}$ to $\widetilde\cG^{\text{av}}$ and $\widetilde\cG^{\text{iso}}$, which are defined by the previously stated rules (a)-(e) with the addition of
\begin{enumerate}[(i)]
\item[(f)] certain vertices may be uncoloured.
\end{enumerate}
These uncoloured vertices represent exactly those labels of types $L_3$ and $L_4$, which are parameters of $\cM$, as defined in 
\eqref{E prod Lambda tuples full}. For example, the previously studied graphs
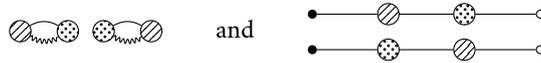
\[\plotLambda{{1,2},{2,1}}\qquad\text{and}\qquad \plotlLambda{{1,2},{2,1}}
\]
can be extended to
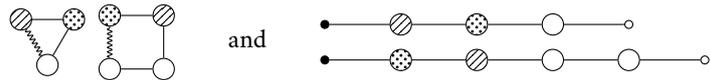
\[\plotLambda{{1,2,0},{2,1,0,0}}\qquad\text{and}\qquad \plotlLambda{{1,2,0},{2,1,0,0}}.
\]
The definition of the value naturally extends to these larger classes of graphs, but without a summation over the uncoloured vertices. In the above example \eqref{Lambda1221 ex av} is then replaced by 
\[\sum_{\alpha_1,\beta_1^2(1)} \kappa(\alpha_1,\beta_1^2(1))\sum_{\alpha_2,\beta_2^1(1)} \kappa(\alpha_2,\beta_2^1(1)) \Lambda_{\alpha_1,\beta_2^1(1),\gamma(1)} \Lambda_{\alpha_2,\beta_1^2(1),\gamma(2),\gamma(3)} = \Val\left(\plotLambda{{1,2,0},{2,1,0,0}}\right),\]
where $\gamma(1),\gamma(2),\gamma(3)$ are the fixed labels and the value of the graph depends on them. The isotropic case is analogous.

The argument in Sections \ref{av bound sec} and \ref{iso bound sec}, however, only concern those labels which are actually summed over, i.e., those of type $l$ for $l\in L_2$. In other words, we only aim at improving the $L_2$-summation by Ward estimates. The presence of additional fixed labels do neither change the naive bounds, the improvement through Ward estimates, nor the counting of those Ward estimates. 

Next, the restricted summations due to the neighbourhood sets $\NN(\alpha)\subset I$ do also not change the argument. In fact, Ward estimates stay true for restricted summations since
\[\sum_{a\in \mathcal J}\abs{G_{a\vx}}^2\le \sum_{a\in J}\abs{G_{a\vx}}^2 = \frac{\Im G_{\vx\vx}}{\eta}\] for arbitrary $\mathcal J\subset J$. Also the procedure for improving the naive size in Section \ref{iso bound sec} holds true if only summed over subsets, i.e.,
\[\sum_{a_1\in \mathcal J}\kappa(a_1b_1,a_2b_2)x_{a_1} = \kappa(\vx_{\mathcal J} b_1,a_2b_2),\]
where the sub-vector $\vx_{\mathcal J}$ has bounded norm $\norm{\vx_{\mathcal J}}\le \norm{\vx}$. 

Finally, the modification of $\cM$ by setting $W_{\NN_{L_3}}=0$ also does not change the substance of the argument as the bound verbatim also covers this modified $W$, and the final bounds can be rephrased in terms of $\norm{G}$ as $\norm[0]{\widehat G}_q\le_q 1+\norm{G}_{Cq/\mu}^{C/\mu}$, as demonstrated in Lemma \ref{G D triv bound lemma}.

\subsection{Proof of Theorem \ref{theorem step}}
We now have all the ingredients to complete the proof of Theorem \ref{theorem step} starting from \eqref{E prod Lambda cM bound}, where we recall that $\cM$ was defined in \eqref{E prod Lambda tuples full}. 

\subsubsection*{Proof of the averaged bound}
We recall from \eqref{naive size averaged} that for the averaged bound the naive size of $\cM$ is given by 
\[ \cM\lesssim N^{-p} N^{-\abs{L}/2-M_{L}/2}N^{2\abs{L_2}},\]
where the first factor comes from the normalized trace, the second from the derivatives and the third from the $L_2$ summations. We demonstrated in Section~\ref{av bound sec} (see \eqref{av val bound sim} and the counting estimate \eqref{marking count average}) that through Ward estimates we can improve the naive size $N^{2\abs{L_2}}$ of the $L_2$ summation to 
\begin{align*} \cM &\lesssim N^{-p} N^{-\abs{L_3\sqcup L_4}/2-M_{L_3\sqcup L_4}/2} \prod_{\substack{l\in L_2\\ M_l\ge 3}} N^{3/2-M_l/2} \prod_{\substack{l\in L_2\\ M_l\le 2}} (N\psi^2)^{3/2-M_l/2}\\
&\le N^{-\abs{L_1}} N^{-3\abs{L_3}/2-M_{L_3}/2} N^{-2\abs{L_4}} \psi^{2\abs{L_2}} \prod_{\substack{l\in L_2\\ M_l\ge 3}} N^{(3-M_l)/2},\end{align*}
where we used that $N\psi^2\ge 1$ and that $M_{L_4}=\abs{L_4}$ and we recall that $p=\abs{L_1}+\abs{L_2}+\abs{L_3}+\abs{L_4}$. Consequently we have from \eqref{E prod Lambda cM bound} that 
\begin{align}\label{EBDp Li bound} \E\abs{\braket{BD}}^p\lesssim_{p,\mu} & N^{-p} + \sum_{\bigsqcup L_i=[p]} N^{-\abs{L_1}} \psi^{2\abs{L_2}}\Bigg[\prod_{\substack{l\in L_2\\ M_l\ge 3}} N^{(3-M_l)/2}\Bigg] N^{-\abs{L_3}-\mu M_{L_3}} N^{-\abs{L_4}} , 
\\\nonumber
 \lesssim_{p,\mu} &  N^{-p} + \psi^{2p} \sum_{\bigsqcup L_i=[p]} 
 \Bigg[\prod_{\substack{l\in L_2\\ M_l\ge 3}} N^{(3-M_l)/2}\Bigg] N^{-\mu M_{L_3}} \lesssim_{p,\mu}  \psi^{2p} \sum_{\bigsqcup L_i=[p]}  N^{-\frac{1}{2}(M_{L_2}-3p)_+-\mu M_{L_3}}, 
 \end{align}
where we bounded the $L_3$-summation in \eqref{E prod Lambda cM bound} by $N^{2\abs{L_3}}(N^{1/2-\mu})^{M_{L_3}}=N^{2\abs{L_3}+M_{L_3}/2} N^{-\mu M_{L_3}}$ in the first line, and used $N^{-1}\le\psi^2$ in the second. To conclude the moment bound \eqref{av bound D eq} from \eqref{EBDp Li bound} we have to count the number of $\norm{G}_q$'s just as in the proof of \eqref{av bound proof final}. The key point 
is to collect enough $N^{-\mu}$ factors so that all but maybe $O(p)$ factors $\norm{G}_q$ could  be compensated by an $N^{-\mu}$. 
Since all $\abs{L_i}$ and  $M_{L_4} = \abs{L_4}$ are of order $p$,
the only way of collecting  more
than $Cp$  factors of $\norm{G}_q$ is having $M_{L_2}$ or $M_{L_3}$ bigger than a constant times $p$. But in this case we collect the same order of factors of the type $N^{-1/2}$ or $N^{-\mu}$ from \eqref{EBDp Li bound} and the claim follows since $N^{-1/2}\le N^{-\mu}$.

\subsubsection*{Proof of the isotropic bound}
We recall from \eqref{eq naive size iso} that for the isotropic law the improved naive size of $\cM$ is given by
\[ \cM\lesssim N^{-\abs{L_3\sqcup L_4}/2 - M_{L_3\sqcup L_4}/2} \prod_{\substack{l\in L_2\\ M_l\ge 3}} N^{3/2-M_l/2} \] and from \eqref{markings iso} that we can always perform at least $\big(p+\sum_{l\in L_2,\,M_l\ge 3} (3-M_l)\big)_+$ Ward estimates. 
Consequently, with Proposition \ref{prop explicit formula Lambda(D) power} and \eqref{E prod Lambda cM bound} we obtain
\begin{align}\label{EDxyp Li bound} \E\abs{D_{\vx\vy}}^p \lesssim_{p,\mu} N^{-p}+ \sum_{\bigsqcup L_i=[p]} N^{-\mu M_{L_3}} \psi^{\big(p+\sum_{l\in L_2,\,M_l\ge 3} (3-M_l)\big)_+} \prod_{\substack{l\in L_2\\ M_l\ge 3}} N^{3/2-M_l/2},\end{align}
where we again bounded the $L_3$ summation in \eqref{E prod Lambda cM bound} by $N^{2\abs{L_3}+M_{L_3}/2} N^{-\mu M_{L_3}}$. The rhs.~of \eqref{EDxyp Li bound} is bounded by $\psi^p$ since every missing $\psi$ power is compensated by an $N^{-1/2}\ll\psi$. To conclude the moment bound \eqref{bootstrapping step} from \eqref{EDxyp Li bound} we again have to count the number of $\norm{G}_q$-factors as in the proof of \eqref{isotropic bound proof final}.  This is 
 very similar to the averaged case above and completes the proof of Theorem \ref{theorem step}.

\subsubsection*{Modifications for large \(|z|\)}
Our proof so far assumed $\braket{z}\lesssim 1$  and $\| H \|\lesssim 1$. 
 The general case follows exactly along the same lines but carrying 
additional $\braket{z}$-factors.  The condition  $\| H \|\lesssim 1$ is relaxed to $\| H \|\lesssim N^\epsilon$ for any fixed $\epsilon>0$
with very high probability. This upper bound directly follows by applying  Chebysev's inequality to the  moment bound $\E \braket{H^k} \le C_k$,
for any $k\in \N$, 
which is obtained by a cumulant expansion using the assumed decays of the cumulants.
We therefore redefine $\psi : =N^\epsilon \braket{z}\sqrt{\Im G/N \eta}$  and with this definition $N\psi^2\gtrsim 1$, used in~\eqref{av bound proof} and~\eqref{isotropic bound proof} still holds. The other two key bounds we used, $|G_\alpha| \le 1$ 
and $\sum_\alpha |G_\alpha|^2 \lesssim N^2\psi^2$,
naturally change to  $|G_\alpha|\lesssim N^\epsilon/\braket{z}$ and $\sum_\alpha |G_\alpha|^2 \lesssim N^2\psi^2/\braket{z}^2$.  Ignoring the irrelevant $N^\epsilon$ factors, these obvious bounds
express the correct scaling of $G$ in the large $z$ regime, thus every $G$ factor naturally comes with a $1/\braket{z}$-factor
on top of any previous bounds we used so far. Since each $D$ carries one $G$-factor, and for the main term of the cumulant expansion of $\E |\Lambda(D)|^p$ the number of $G$-factors does not decrease, we obtain at least an additional $\braket{z}^{-p}$-factor. For the error term of the cumulant expansion the \(\braket{z}^{-p}\) decay in~\eqref{cE error} follows directly from the corresponding decays in Lemma~\ref{G D triv bound lemma}. Combining this with the redefined $\psi$ and adjusting the $\epsilon$-exponent, we obtain Theorem~\ref{theorem step} in the general case.

\section{Proof of the stability of the MDE and proof of the local law}\label{section stability}
Before going into the proof of Theorem \ref{isotropic local law}, we collect some facts from \cite{MR2376207,1604.08188,1706.08343} about the deterministic MDE \eqref{matrix dyson eq} and its solution.
\begin{proposition}[Stability of MDE and properties of the solution]\label{prop stability reference} The following hold true under Assumption \ref{assumption A}.
\begin{enumerate}[(i)]
\item The MDE \eqref{matrix dyson eq} has a unique solution $M=M(z)$ for all $z\in\HC$ and moreover the map $z\mapsto M(z)$ is holomorphic.
\item \label{mu def} The holomorphic function $\braket{M}\colon \HC\to\HC$ is the Stieltjes transform of a probability measure $\mu$ on $\R$.
\item \label{norm M bound without flatness}There exists a constant $c>0$ such that we have the bounds
\begin{align*}
\frac{c}{\braket{z}+\tnorm{\SS}\dist(z,\supp\mu)^{-1}}\le \norm{M(z)}\le \frac{1}{\dist(z,\supp\mu)} \quad\text{and}\quad \norm{\Im M}\le \frac{\eta}{\dist(z,\supp\mu)^2},
\end{align*}
where we recall the definition of $\tnorm{\SS}$ in \eqref{tnorm SS}.
\item \label{1-CMS bound without flatness}There exist constants $c,C>0$ such that 
\[ \norm{(1-\mathcal C_{M(z)} \SS)^{-1}}_{\text{hs}\to\text{hs}}\le c \Big[\frac{\braket{z}}{\dist(z,\supp \mu)}+\frac{\tnorm{\SS}}{\dist(z,\supp \mu)^{2}}\Big]^C, \]
where $\mathcal C$ is the \emph{sandwiching} operator $\mathcal C_R[T]\defeq RTR$. The norm on the lhs.~is the  operator norm
where  $1-\mathcal C_M \SS$ is viewed as a linear map on the space of matrices equipped with the Hilbert-Schmidt norm.
\end{enumerate}
If, in addition, Assumption \ref{assumption flatness} is also satisfied, then the following statements hold true, as well.
\begin{enumerate}[(i)] \setcounter{enumi}{4}
\item The measure $\mu$  from \eqref{mu def} is absolutely continuous with respect to the Lebesgue measure and has a continuous density $\rho\colon\R\to[0,\infty)$, called the self-consistent density of states, which is also real analytic on the open set $\set{\rho>0}$.
\item\label{norm M bound with flatness} There exist constants $c,C>0$ such that we have the bounds
\[ \frac{c}{\braket{z}} \le \norm{M(z)} \le  \frac{C}{\rho(z)+\dist(z,\supp\rho)} \quad\text{and}\quad c\,\rho(z)\le \Im M(z) \le C \braket{z}^2\norm{M(z)}^2 \rho(z)\]
in terms of the harmonic extension $\rho(z)\defeq \pi^{-1}\Im\braket{M(z)}$ of the self-consistent density of states to the upper half plane $\HC$.
\item\label{1-CMS bound with flatness} There exist constants $c,C>0$ such that 
\[ \norm{(1-\mathcal C_{M(z)} \SS)^{-1}}_{\text{hs}\to\text{hs}}\le c\left(1+\big[\rho(z)+\dist(z,\supp \rho)\big]^{-C}\right).\]
\end{enumerate}
\end{proposition}
\begin{proof}
Parts (i)--(ii) follow from \cite[Thm.~2.1]{MR2376207}. Parts (iii)--(iv) follow from \cite[Section 3]{1706.08343} and $\norm{M}\ge \norm{M^{-1}}^{-1}$. Finally, parts (v)--(vii) follow from \cite[Prop.~2.2, 4.2, 4.4]{1604.08188}.
\end{proof}

Due to Assumption \ref{assumption correlations}, \eqref{tnorm SS} and \eqref{S norm to norm bound} below we have $\tnorm{\SS}\le C$. Therefore parts \eqref{norm M bound without flatness},\eqref{1-CMS bound without flatness},\eqref{norm M bound with flatness} and \eqref{1-CMS bound with flatness} show that we have 
\begin{align}\label{eq reference key info}\braket{z} \norm {M(z)}\le_\epsilon N^\epsilon  \quad\text{and}\quad  \norm {(1-\cC_M\SS)^{-1} }_{\text{hs} \to \text{hs}} \le_\epsilon N^\epsilon\qquad\text{in}\quad\Dout^\delta\quad\text{and also in}\quad\DD_0^\delta\quad\text{under Assump.~\ref{assumption flatness}}\end{align}
for some $\delta=\delta(\epsilon)>0$. Similarly to \eqref{eq reference key info}, we will  often state estimates that hold
both in the spectral domain $\Dout^\delta$ without Assumption \ref{assumption flatness} as well as in the spectral domain $\DD_\gamma^\delta$
but under Assumption \ref{assumption flatness}. We recall that according to our convention about $\le_\epsilon$, \eqref{eq reference key info} implies the existence of a constant $C(\epsilon)$ such that the inequalities hold true with that constant for all $z$ in the given $\epsilon$-dependent domains.

\subsection{Definition of an isotropic norm suitable for the stability analysis}
For a fixed $z\in \HC$  define  the map
\[\cJ_z[G,D]\defeq 1+(z-A+\SS[G])G - D
\] 
on arbitrary matrices $G$ and $D$.
From the definition of $D=D(z)$ \eqref{eq D def} and the solution $M=M(z)$ of the MDE \eqref{matrix dyson eq} it follows that
 $\cJ_z[M(z), 0]=0$ and $\cJ_z[G(z),D(z)]=0$. 
Throughout this discussion we will fix $z$ and we omit it from the notation, i.e. $\cJ=\cJ_z$.
We will consider $G$ as a function  $G=G(D)$ of an arbitrary error matrix  $D$ satisfying $\cJ[G(D),D]=0$.
Via the implicit function theorem, this relation defines a unique function $G(D)$ for sufficiently small $D$
and $G(D)$ will be analytic
as long as $\cJ$ is stable.   The stability will be formulated in a specific norm 
that takes into account that the smallness of $D$ can only be established in isotropic sense, i.e. 
in the sense of high moment bound on $D_{\vx\vy}$ for any fixed  deterministic vectors $\vx, \vy$. 
To define this special norm, we fix vectors $\vx,\vy$ and define sets of vectors containing the standard basis vectors $e_a, a\in J$, recursively by
\[I_0\defeq\set{\vx,\vy}\cup\Set{e_a|a\in J},\qquad I_{k+1}\defeq I_k\cup \Set{M\vu|\vu\in I_k}\cup \Set{\kappa_c((M\vu)a,b\cdot),\kappa_d((M\vu)a,\cdot b)|\vu\in I_k,\,a,b\in J},\]
which give rise to the norm
\[\norm{G}_\ast = \norm{G}_\ast^{K,\vx,\vy} \defeq \sum_{0\le k< K}N^{-k/2K} \norm{G}_{I_k} + N^{-1/2} \max_{\vu\in I_K}\frac{\norm{G_{\cdot\vu}}}{\norm{\vu}},\qquad \norm{G}_I \defeq \max_{\vu,\vv\in I} \frac{\abs{G_{\vu\vv}}}{\norm{\vu}\norm{\vv}}, \]
where we will choose $K$ later.

\begin{theorem}
\label{thr:Stability}
Let $K \in \N$, $\vx,\vy \in \C^N$, and denote the open ball of radius $\delta$ around $M$ in $(\C^{N\times N},\norm{\cdot}_\ast^{K,\vx,\vy})$ by $B_\delta(M)$. Then for 
\begin{align}
\label{definition of epsilons}
\epsilon_1&\defeq \frac{\Big[ 1 + \tnorm{\SS}\norm{M}^2 + \tnorm{\SS}^2\norm{M}^4 \norm{(1-\mathcal{C}_M\SS)^{-1}}_{\text{hs}\to\text{hs}} \Big]^{-2}}{10 N^{1/K}\norm{M}^2\tnorm{\SS}},\qquad \epsilon_2 \defeq \sqrt{\frac{\epsilon_1}{10\tnorm{\SS}}}
\end{align}
there exists a unique function $G\colon B_{\epsilon_1}(0)\to B_{\epsilon_2}(M)$ with $G(0)=M$ that satisfies $\cJ[G(D),D]=0$. Moreover, the function $G$ is analytic and satisfies
\begin{align}
\label{G Lipschitz in D}
\norm{G(D_1)-G(D_2)}_\ast \le 10\,N^{1/2K}\norm{(1-\mathcal{C}_M\SS)^{-1}}_{\ast\to\ast}\norm{M}\norm{D_1-D_2}_\ast.
\end{align}
for any $D_1,D_2\in B_{\epsilon_1}(0)$.
\end{theorem}

\begin{proof} First, we rewrite the equation $\cJ[G,D]=0$ in the form $\widetilde\cJ[V,D]=0$, where \[\widetilde\cJ[V,D]\defeq (1-\cC_M\SS)V-M\SS[V]V+MD,\qquad V\defeq G-M\]
and for arbitrary $V$ and $D$ we claim the bounds 
\begin{subequations}
 \begin{align}
 \label{bound on MSRR}\norm{M\mathcal{S}[V]V}_\ast &\le N^{1/2K}\tnorm{\SS} \norm{M} \norm{V}_{\ast}^2,\\
 \label{bound on MD}\norm{MD}_\ast &\le N^{1/2K}\norm{M} \norm{D}_\ast,\\
 \label{stability} \norm{(1-\cC_M\SS)^{-1}}_{\ast \to\ast}&\le 1 +\tnorm{\SS} \norm{M}^2 + \tnorm{\SS}^2\norm{M}^4 \norm{(1-\cC_M\SS)^{-1}}_{\text{hs}\to\text{hs}}.
 \end{align}
 \end{subequations} 
We start with the proof of \eqref{bound on MSRR}. Let $\kappa=\kappa_c+\kappa_d$ be an arbitrary partition which induces a partition of $\SS=\SS_c+\SS_d$ (as in Remark \ref{S decomp remark}). Then for $\vu,\vv \in I_k$ we compute
\begin{subequations}
\begin{align}
\label{starting bound kappac}
\frac{\abs{(M\SS_c[V]V)_{\vu\vv}}}{\norm{\vu}\norm{\vv}}&\le\frac{1}{N}\sum_{a,b} \frac{\abs{V_{ab} V_{\kappa_c((M\vu)a,b\cdot) \vv}}}{\norm{\vu}\norm{\vv}} \le \tnorm{\kappa_c}_c\norm{V}_{\text{max}} \norm{M} \min\Set{\norm{V}_{I_{k+1}},\frac{\norm{V_{\cdot\vv}}}{\norm{\vv}}},\\
\frac{\abs{(M\SS_d[V]V)_{\vu\vv}}}{\norm{\vu}\norm{\vv}}&\le\frac{1}{N}\sum_{a,b} \frac{\abs{V_{a\kappa_d((M\vu)a,\cdot b)} V_{b \vv}}}{\norm{\vu}\norm{\vv}} \le \tnorm{\kappa_d}_d\norm{M} \min\Set{\norm{V}_{I_{k+1}} \frac{\norm{V_{\cdot\vv}}}{\sqrt N \norm{\vv}},\norm{V}_{\text{max}}\frac{\norm{V_{\cdot\vv}}}{\norm{\vv}}},\label{eq MSdRR bound}
\end{align}
where we used $\abs{V_{a\vw}}\le \sqrt N \norm{V}_{\text{max}}\norm{\vw}$ in the second bound of \eqref{eq MSdRR bound}, so that 
\begin{align*}
\norm{M\SS_e[V]V}_\ast = \sum_{0\le k< K} \frac{\norm{M\SS_e[V]V}_{I_k}}{N^{k/2K}} + \max_{\vu\in I_K}\frac{\norm{(M\SS_e[V]V)_{\cdot\vu}}}{\sqrt N\norm{\vu}} \le N^{1/2K}\tnorm{\kappa_e}_e \norm{M} \norm{V}_\ast^2
\end{align*}
for $e\in\{c,d\}$ and \eqref{bound on MSRR} follows immediately, recalling \eqref{tnorm SS}.
\end{subequations} We continue with the proof of \eqref{bound on MD}, which follows from the fact that for $\vu,\vv\in I_k$ we have 
\[\frac{\abs{(MD)_{\vu\vv}}}{\norm{\vu}\norm{\vv}}\le \norm{M} \min\Set{\norm{D}_{I_{k+1}},\frac{\norm{D_{\cdot\vv}}}{\norm{\vv}}}.\] Finally, we show \eqref{stability}. We use a three term geometric expansion to obtain
\begin{align}
\label{geometric expansion of 1-CMS}
\norm{(1-\mathcal{C}_M\SS)^{-1}}_{\ast\to\ast} &\le  1 + \norm{\mathcal{C}_M\SS}_{\ast\to\ast}+ 
\norm{\mathcal{C}_M\SS}_{\ast \to \text{hs}}\norm{(1-\mathcal{C}_M\SS)^{-1}}_{\text{hs} \to \text{hs} }
\norm{\mathcal{C}_M\SS}_{\text{hs} \to \ast}
\\
&\le 1 + \norm{M}^2\norm{\SS}_{\text{max}\to \norm{\cdot}}+\norm{M}^4\norm{\SS}_{\text{max} \to \norm{\cdot}}\norm{(1-\mathcal{C}_M\SS)^{-1}}_{\text{hs} \to \text{hs} }
\norm{\SS}_{\text{hs} \to \norm{\cdot}}\nonumber
\end{align}
and it only remains to derive bounds on $\norm{\SS}_{\text{max}\to \norm{\cdot}}$ and $\norm{\SS}_{\text{hs}\to \norm{\cdot}}$. We begin to compute for the cross part $\kappa_c$ and arbitrary normalized vectors $\vv,\vu\in\C^N$ that
\begin{align*} 
\abs{\SS_c[V]_{\vv\vu}} = \abs[2]{\frac{1}{N}\sum_{b_1,a_2} \braket{\kappa_c(\vv b_1,a_2 \cdot),\vu} V_{b_1a_2}}\le \frac{\norm{V}_{\max}}{N}\sum_{b_1,a_2} \norm{\kappa_c(\vv b_1,a_2 \cdot)}\le \tnorm{\kappa_c}_c\norm{V}_{\max}, 
\end{align*}
and
\begin{align*}
\abs{\SS_c[V]_{\vv\vu}} &= \abs[2]{\frac{1}{N}\sum_{a_1,a_2,b_2} v_{a_1}\braket{\kappa_c(a_1 \cdot,a_2 b_2),V_{\cdot a_2}} u_{b_2}}\le \frac{1}{N}\sum_{a_1,a_2,b_2} \abs{v_{a_1}}\norm{\kappa_c(a_1 \cdot,a_2 b_2)}\abs{u_{b_2}}\norm{V_{\cdot a_2}}\le \frac{\tnorm{\kappa_c}_c}{N} \sum_{a_2} \norm{V_{\cdot a_2}} \\
&\le \tnorm{\kappa_c}_c\sqrt{\frac{1}{N}\sum_{b_1,a_2} \abs{V_{b_1 a_2}}^2} = \tnorm{\kappa_c}_c \norm{V}_{\text{hs}}.
\end{align*}
Next, we estimate for the direct part $\kappa_d$ that
\begin{align*}
\abs{\SS_d[V]_{\vv\vu}} &= \abs[2]{\frac{1}{N}\sum_{b_1,b_2} \braket{\kappa_d(\vv b_1,\cdot b_2),V_{b_1\cdot}} u_{b_2} } \le \frac{1}{N} \sum_{b_1,b_2} \norm{V_{b_1\cdot}}\norm{\kappa_d(\vv b_1,\cdot b_2)}\abs{u_{b_2}} \le \frac{\tnorm{\kappa_d}_d}{N} \sqrt{\sum_{b_1}\norm{V_{b_1\cdot}}^2} \\
& \le \frac{\tnorm{\kappa_d}_d}{N} \sqrt{\sum_{b_1,a_2}\abs{V_{b_1a_2}}^2}\le\tnorm{\kappa_d}_d \min\left\{ \frac{\norm{V}_{\text{hs}}}{\sqrt N}, \norm{V}_{\max} \right\},
\end{align*}
so that it follows that,
using \eqref{tnorm SS},\begin{align}\label{S norm to norm bound}
\norm{\SS[V]}=\sup_{\norm{\vv},\norm{\vu}\le 1} \abs{\SS[V]_{\vv\vu}} \le \tnorm{\SS} \min\left\{ \norm{V}_{\text{hs}},\norm{V}_{\max}\right\},\quad \max\left\{\norm{S}_{\text{max}\to\norm{\cdot}},\norm{S}_{\text{hs}\to\norm{\cdot}}\right\}\le \tnorm{\SS}
\end{align}
and therefore \eqref{stability} follows from \eqref{geometric expansion of 1-CMS} with \eqref{S norm to norm bound}. Now the statement \eqref{G Lipschitz in D} follows from the implicit function theorem as formulated in Lemma \ref{impl fct lemma} applied to the equation $\widetilde \cJ[G-M, D] =0$ written in the form
\[(1-\cC_M \SS)V - M\SS[V]V = -MD\]
with $A = 1-\cC_M \SS, B=M$ and $d=D$ in the notation of Lemma \ref{impl fct lemma}.
\end{proof}
This general stability result will be used in the following form
\begin{align}\norm{G- M}_\ast \le_\epsilon N^{\epsilon+1/2K} \frac {\norm{D}_\ast}{\braket{z}}\qquad\text{in}\quad\Dout^\delta\quad\text{and in}\quad\DD_0^\delta\quad\text{under Assump.~\ref{assumption flatness}}   \label{G-M ast bound}\end{align}
for some $\delta=\delta(\epsilon)>0$, as long as $\norm{D}_*\le N^{-1/2K}\braket{z}^2$ by applying it to $D_1=0, D_2= D(z)$ and using  \eqref{eq reference key info} and \eqref{stability}.

\subsection{Stochastic domination and relation to high moment bounds}
In order to keep the notation compact, we now introduce a commonly used (see, e.g., \cite{MR3068390}) notion of high-probability bound.
\begin{definition}[Stochastic Domination]\label{def:stochDom}
If \[X=\left( X^{(N)}(u) \,\lvert\, N\in\N, u\in U^{(N)} \right)\quad\text{and}\quad Y=\left( Y^{(N)}(u) \,\lvert\, N\in\N, u\in U^{(N)} \right)\] are families of random variables indexed by $N$, and possibly some parameter $u$, then we say that $X$ is stochastically dominated by $Y$, if for all $\epsilon, D>0$ we have \[\sup_{u\in U^{(N)}} \P\left[X^{(N)}(u)>N^\epsilon  Y^{(N)}(u)\right]\leq N^{-D}\] for large enough $N\geq N_0(\epsilon,D)$. In this case we use the notation $X\prec Y$. 
\end{definition}
It can be checked (see \cite[Lemma 4.4]{MR3068390}) that $\prec$ satisfies the usual arithmetic properties, e.g. 
if $X_1\prec Y_1$ and $X_2\prec Y_2$, then also  $X_1+X_2\prec Y_1 +Y_2$ and  $X_1X_2\prec Y_1 Y_2$.
We will say that a (sequence of) events  $A=A^{(N)}$  holds with \emph{overwhelming probability} if $\P (A^{(N)}) \ge 1- N^{-D}$ for any $D>0$ and $N\ge N_0(D)$.  In particular, under Assumption \ref{assumption high moments}, we have $w_{ij}\prec 1$.

In the following lemma we establish that a control of the $\norm{\cdot}_\ast^{K,\vx,\vy}$-norm for all $\vx,\vy$ in a high probability sense is essentially equivalent to a control of the $\norm{\cdot}_p$-norm for all $p$.
\begin{lemma}\label{prec p conversion}
Let $R$ be a random matrix and $\Phi$ a deterministic control parameter. Then the following implications hold:
\begin{enumerate}[(i)]
\item If $\Phi \ge N^{-C}$, $\norm{R} \le N^C$ and $\abs[0]{R_{\vx\vy}}\prec \Phi$ for all normalized $\vx,\vy$ and some $C$, then also $\norm{R}_p\le_{p,\epsilon} N^\epsilon \Phi$ for all $\epsilon>0,p\ge1$. \label{prec to p norm}
\item Conversely, if $\norm{R}_p \le_{p,\epsilon} N^\epsilon \Phi $ for all $\epsilon>0,p\ge 1$, then $\norm{R}_\ast^{K,\vx,\vy} \prec \Phi$ for any fixed $K\in\N$, $\vx,\vy\in\C^N$. \label{p to star norm}
\end{enumerate}
\end{lemma}
\begin{proof}
We begin with the proof of \eqref{p to star norm} and infer from Markov's inequality and H\"older's inequality (as in \eqref{eq triv bound}) that
\begin{equation}
\label{Markov inequality}
\P\left(\norm{R}_\ast>N^{\sigma}\Phi\right) \le \P\left( 2\norm{R}_{I_K}>N^{\sigma}\Phi\right) \le_p \frac{\E \norm{R}_{I_K}^p}{N^{\sigma p}\Phi^p} \le_p \abs{I_K}^{2/r} \frac{\E \norm{R}_{pr}^p}{N^{\sigma p}\Phi^p} \le_{p,r,\epsilon} \abs{I_K}^{2/r} N^{\epsilon p-\sigma p},
\end{equation} 
and since $\abs{I_K}\le 4^K N^{K+2}$ we conclude that $\norm{R}_\ast \prec \Phi$ by choosing $\epsilon$ sufficiently small and $p,r$ sufficiently large. On the other hand, \eqref{prec to p norm} directly follows from 
\[
\norm{R}_p \le  N^{\epsilon}\Phi + \sup_{\norm{\vx},\norm{\vy}\le 1}\big(\abs{R_{\vx\vy}}\P[\abs{R_{\vx\vy}}\ge N^\epsilon \Phi]^{1/p}\big).\qedhere
\]
\end{proof}

\subsection{Bootstrapping step} The proof of the local law follows a \emph{bootstrapping procedure}: First, we prove the local law for $\eta\ge N$, and afterwards we iteratively show that if the local law holds for $\eta\ge N^{\gamma_0}$, then it also holds for $\eta\ge N^{\gamma_1}$ for some $\gamma_1<\gamma_0$. We now formulate the iteration step. 

\begin{proposition}\label{proposition step}
The following holds true under the assumptions of Theorem \ref{isotropic local law}: Let $\delta,\gamma>0$ and $\gamma_0>\gamma_1\ge \gamma$ with $4(2C_\ast/\mu+1)(\gamma_0-\gamma_1)<\gamma<1/2$ and suppose that
\begin{align}\label{G gamma 0 bound}
\norm{G-M}_p \le_{\gamma,p} \frac{N^{-\gamma/6}}{\braket{z}} \inD{\gamma_0},
\end{align}
holds for all $p\ge 1$, where $C_\ast$ is the constant from Theorem \ref{theorem step}. Then the same inequality \eqref{G gamma 0 bound} (with a possibly different implicit constant depending on $\gamma,\delta,p$) holds also true in $\DD_{\gamma_1}^{\delta'}$ for some $\delta'=\delta'(\gamma,\delta)>0$. Furthermore, the same statement holds true under the assumptions of Theorem \ref{isotropic local law away} if we replace $\DD_{\gamma_{0}}^\delta$ and $\DD_{\gamma_{1}}^\delta$ by $\DD_{\gamma_{0}}^\delta\cap\Dout^\delta$ and $\DD_{\gamma_{1}}^\delta\cap\Dout^\delta$, respectively, in the above
sentence.
\end{proposition}
\begin{proof}
We first prove the assertion under the assumptions of Theorem \ref{isotropic local law}. In the proof we will abbreviate the step size from $\gamma_0$ to $\gamma_1$ by $\gamma_s\defeq \gamma_0-\gamma_1$. We will suppress the dependence of the constants on $\delta,\gamma$ in our notation. In particular, \eqref{G gamma 0 bound} and \eqref{eq reference key info} imply $\norm{G}_{p} \le_{p,\gamma} N^{\gamma_s}\braket{z}^{-1}$ in $\DD_{\gamma_0}^{\delta'}$ with $\delta'=\delta'(\gamma)$. For fixed $E$ the function $\eta\mapsto f(\eta)\defeq \eta \norm{G(E+\ii\eta)}_p$ satisfies 
\begin{align}
\liminf_{\epsilon\to 0} \frac{f(\eta+\epsilon)-f(\eta)}{\epsilon} &\ge \norm{G(E+\ii\eta)}_p-\eta\norm{\lim_{\epsilon\to 0} \frac{G(E+\ii(\eta+\epsilon))-G(E+\ii\eta)}{\epsilon}}_p\\
&=\norm{G(E+\ii\eta)}_p-\eta\norm{G(E+\ii\eta)^2}_p\ge 0,\nonumber
\end{align}
where we used \[\eta\abs{\braket{\vx,G^2 \vy}}\le \frac{\eta}{2} \left(\braket{\vx,\abs{G}^2\vx}+ \braket{\vy,\abs{G}^2\vy}\right)\le \frac{1}{2}\left(\braket{\vx, \Im G\vx}+\braket{\vy,\Im G\vy}\right) \]
in the last step. We thus know that $\eta\mapsto\eta\norm{G(E+\ii\eta)}_p$ 
is monotone and we can conclude that $\braket{z}\norm{G}_{p}\le_{p,\gamma} N^{2\gamma_s}$ in $\DD_{\gamma_1}^{\delta'}$. From \eqref{bootstrapping step} and $\gamma_s<\mu$ it thus follows that 
\begin{equation}\label{Dbound}
\norm{D}_p \le_{p,\gamma,\epsilon} N^{\epsilon + 2(C_\ast/\mu+1/2)\gamma_s-\gamma/2}\le N^{\epsilon-\gamma/4} \quad\text{in}\quad\DD_{\gamma_1}^{\delta'}.
\end{equation}
Note that the exponent in the right hand side is independent of $p$; this was possible
because the power of $\norm{G}_q$  in \eqref{bootstrapping step} was linear in $p$. 
 
We now relate these high moment bounds to high probability bounds in the $\norm{\cdot}_\ast$ norm, as defined before Theorem~\ref{thr:Stability} and find for any fixed $\vx,\vy$ and $K$ that $\norm{D}_\ast\prec N^{-\gamma/4}$ from Lemma \ref{prec p conversion}\eqref{p to star norm} (we recall that the $\norm{\cdot}_\ast$ implicitly depends on $\vx,\vy$ and $K$). Next, we apply \eqref{G-M ast bound} to obtain
\begin{align}\label{G-M char bound}
\norm{G-M}_\ast\chi(\norm{G-M}_\ast \le N^{-\gamma/9})\prec \frac{N^{-\gamma/5}}{\braket{z}}  \quad\text{in}\quad\DD_{\gamma_1}^{\delta'},
\end{align}
provided $K\ge 10/\gamma$. The bound \eqref{G-M char bound} shows that there is a gap in the set of possible values for $\norm{G-M}_\ast$. The extension of \eqref{G gamma 0 bound} to $\DD_{\gamma_1}^{\delta'}$ then follows from a standard continuity argument using a fine grid of intermediate values of $\eta$: Suppose that \eqref{G-M char bound} were true as a deterministic inequality. Since $\eta\mapsto \norm[0]{(G-M)(E+\ii\eta)}_\ast$ is continuous, and for $\eta=N^{-1+\gamma_0}$ we know that $\norm{(G-M)(E+\ii\eta)}_\ast\le N^{-\gamma/6}$ by \eqref{G gamma 0 bound} and Lemma \ref{prec p conversion}\eqref{p to star norm}, we would conclude the same bound for $\eta=N^{-1+\gamma_1}$. Going back to the $\norm{\cdot}_p$-norm by Lemma \ref{prec p conversion}\eqref{prec to p norm} we could conclude \eqref{G gamma 0 bound} in $\DD_{\gamma_1}^\delta$. Since \eqref{G-M char bound} may not control $\norm{G-M}_\ast$ on a set of very small probability, and we cannot exclude a ``bad'' set for every $\eta\in [N^{-1+\gamma_1},N^{-1+\gamma_0}]$, we use a fine $N^{-3}$-grid. The relation \eqref{G-M char bound} is only used for a discrete set of $\eta$'s and intermediate values are controlled by the $\eta^{-1}$-Lipschitz continuity of $\norm{G-M}_\ast$ in the continuity argument above. This completes the proof of Proposition~\ref{proposition step} in the setup of Theorem~\ref{isotropic local law}. The proof in the setup of Theorem~\ref{isotropic local law away} is identical except for the fact that the inequalities \eqref{eq reference key info} and \eqref{G-M ast bound} only hold true in the restricted set $\Dout^\delta$ without Assumption \ref{assumption flatness}.
\end{proof}

\subsection{Proof of the local law and the absence of eigenvalues outside of the support}\label{sec local law proof}
We now have all the ingredients to complete the proof of Theorems \ref{isotropic local law away} and \ref{isotropic local law}.
\begin{proof}[Proof of Theorems \ref{isotropic local law away}, \ref{isotropic local law} and Corollary \ref{cor no eigenvalues outside}] 
We will first prove Theorem \ref{isotropic local law} and then remark in the end how to adapt it to prove Theorem \ref{isotropic local law away}. The proof involves five steps. In the first step we derive a weak initial isotropic bound, which we improve in the second step to obtain the isotropic local law. In the third step we use the isotropic local law to obtain the averaged local law in the bulk, which we use in the fourth step to establish that with very high probability there are no eigenvalues outside the support of $\rho$, also proving Corollary \ref{cor no eigenvalues outside}. Finally, in the fifth step we use the fact that there are no eigenvalues outside the support of $\rho$ to improve the isotropic and averaged law outside the support.

\subsubsection*{Step 1: Initial isotropic bound.} We claim the initial bound
\begin{align}
\label{bound for large eta}
\norm{G-M}_p \le_{p,\gamma} \frac{N^{-\gamma/6}}{\braket{z}}\inD{\gamma}
\end{align}
for some $\delta=\delta(\gamma)$. First, we aim at proving \eqref{bound for large eta} for large $\eta\ge N$, i.e., in $\DD_{\gamma=2}^\delta=\DD_{2}^\delta$ for arbitrary $\delta$. We use that \[\norm{H}=\max_k \abs{\lambda_k}\le \sqrt{\Tr \abs{H}^2} \le \sqrt{\Tr \abs{A}^2}+\sqrt{N^{-1}\Tr \abs{W}^2}\prec \sqrt N,\] as follows from Assumptions \ref{assumption A} and \ref{assumption high moments}.
 In fact, by computing traces of higher moments  $W^k$ and using the summable decay of the cumulants, this bound can easily
be improved to $\| H\|\prec 1$. 
 Since $\abs{z}\ge N$ and $\norm{H}\prec\sqrt N$, we have $\norm{G}_p\le_p \braket{z}^{-1}$ and $\norm{\Im G}_p\le_p \braket{z}^{-2}\eta$ and thus from Theorem \ref{theorem step} it follows that that 
\[
\norm{D}_p \le_{p, \epsilon} \frac{N^{\epsilon}}{\braket{z}\sqrt{N}}\inD{2}.
\]
We now fix normalized vectors $\vx,\vy$ and any $K\ge 10/\gamma$ in the norm $\norm{\cdot}_\ast=\norm{\cdot}_\ast^{K,\vx,\vy}$ and translate these $p$ norm bounds into high-probability bounds using Lemma \ref{prec p conversion} to infer $\norm{D}_\ast \prec \braket{z}^{-1} /\sqrt N$ and $\norm{G}_\ast\prec \braket{z}^{-1}$. Using the stability in the form of \eqref{G-M ast bound} and absorbing $N^\epsilon$ factors into $\prec$ we conclude 
\[
\norm{G-M}_\ast \prec \frac{N^{1/2K}}{\braket{z}^2\sqrt{N}}\inD{2}.
\]
Now \eqref{bound for large eta} in $\mathbb D_{2}^\delta$ follows from \ref{prec p conversion}\eqref{prec to p norm} since $\vx,\vy$ and $K$ were arbitrary. By applying Proposition \ref{proposition step} iteratively starting from $\gamma_0=2$ and (possibly) reducing $\delta$ in every step we can then conclude that \eqref{bound for large eta} holds in all of $\mathbb D_{\gamma}^\delta$ for some $\delta=\delta(\gamma)>0$. 

\subsubsection*{Step 2: Iterative improvement of the isotropic bound.} We now iteratively improve the initial bound \eqref{bound for large eta} until we reach the intermediate bound 
\begin{align}
\label{First bound on G-M}
\norm{G-M}_p \le_{p,\epsilon} \frac{N^{\epsilon}}{\braket{z}}\bigg(\sqrt{\frac{\norm{\Im M}}{N \eta}}+\frac{1}{\braket{z}}\frac{1}{N \eta}\bigg)\inD{\gamma}
\end{align}
for $\delta=\delta(\epsilon)>0$. From \eqref{bound for large eta} and the bound on $\braket{z}\norm{M}$ from \eqref{eq reference key info} we conclude that $\braket{z}\norm{G}_p$ is $N^\epsilon$-bounded in $\DD_\gamma^\delta$ for some $\delta=\delta(\epsilon)>0$. Then from Theorem \ref{theorem step} and \eqref{bound for large eta}, again, it follows that 
\begin{align}
\label{initial control on G and D}
\norm{D}_p\le_{p,\epsilon} N^{\epsilon}\sqrt{\frac{\norm{\Im G}_q}{N \eta}} \quad\text{and}\quad \norm{G-M}_\ast+\norm{D}_\ast \prec N^{-\gamma/6} \inD{\gamma}.  
\end{align}
From now on all claimed bounds hold true uniformly in all of $\DD_{\gamma}^\delta$; we will therefore suppress this qualifier in the following steps. In order to prove \eqref{First bound on G-M}, we show inductively
\begin{align}
\label{inductive improvement}
\norm{G-M}_p \le_{p,\epsilon} N^\epsilon\Psi_l,
\end{align}
where we define successively improving control parameters $(\Psi_l)_{l=0}^L$ through $\Psi_0\defeq  1$ and $\Psi_{l+1}\defeq N^{-\sigma}\Psi_l=N^{-(l+1)\sigma}$, where $\sigma \in (0,1)$ is arbitrary.
The final iteration step $L$ is chosen to be the largest integer such that  
\begin{align}
\label{lower bound on Psis}
\Psi_L \ge \frac{N^\sigma}{\braket{z}}\bigg(\sqrt{\frac{\norm{\Im M}}{N \eta}}+ \frac{1}{\braket{z}}\frac{N^{\sigma}}{N \eta}\bigg).
\end{align}
 For the induction step from $l$ to $l+1$, we write  $\Im G = \Im M + \Im (G-M)$ and we continue from \eqref{initial control on G and D} and \eqref{inductive improvement} and estimates that
\[
\norm{D}_p \le_{p,\epsilon} N^\epsilon \left(\sqrt{\frac{\norm{\Im M}}{N\eta}}+\sqrt{\frac{\Psi_l}{N\eta}}\right)\le_{p,\epsilon} N^{\epsilon}\bigg(\sqrt{\frac{\norm{\Im M}}{N \eta}}+\frac{1}{\braket{z}}\frac{N^{\sigma}}{N \eta}+ \braket{z}N^{-\sigma}\Psi_l\bigg).
\]
Thus we also have, for any normalized $\vx,\vy$,
\[
\norm{D}_\ast=\norm{D}_\ast^{K,\vx,\vy} \prec \sqrt{\frac{\norm{\Im M}}{N \eta}}+\frac{1}{\braket{z}}\frac{N^{\sigma}}{N \eta}+ \braket{z}N^{-\sigma}\Psi_l
\]
and from \eqref{G-M ast bound} we conclude 
\[
\norm{G-M}_\ast\,\prec\, \frac{N^{1/2K}}{\braket{z}}\bigg(\sqrt{\frac{\norm{\Im M}}{N \eta}}+ \frac{1}{\braket{z}}\frac{N^{\sigma}}{N \eta}\bigg)+N^{1/2K-\sigma}\Psi_l
\]
provided $K \ge 7/\gamma$ (c.f.~the bound on $\norm{D}_\ast$ from \eqref{initial control on G and D} and the definition of $\epsilon$-neighbourhoods in \eqref{definition of epsilons}). In particular, since $K$ can be chosen arbitrarily large, we find, for any normalized $\vx,\vy$ that
\[
\abs{(G-M)_{\vx\vy}} \prec \frac{1}{\braket{z}}\bigg(\sqrt{\frac{\norm{\Im M}}{N \eta}}+ \frac{1}{\braket{z}}\frac{N^{\sigma}}{N \eta}\bigg)+N^{-\sigma}\Psi_l\le 2N^{-\sigma}\Psi_l,
\]
where we used $l<L$ and \eqref{lower bound on Psis} in the last step. By the definition of $\Psi_{l+1}$ we infer
\[
\norm{G-M}_p\le_{p,\epsilon}N^{\epsilon}\Psi_{l+1},
\]
completing the induction step, and thereby the proof of \eqref{First bound on G-M}. 

Finally, in order to obtain \eqref{iso bulk} from \eqref{First bound on G-M}, we recall \begin{align}
 \norm{\Im M}\le \norm{M}\le_\epsilon N^\epsilon  \label{norm Im M bound}
 \end{align}  from Proposition \ref{prop stability reference}\eqref{norm M bound with flatness} and \eqref{iso bulk} follows.

\subsubsection*{Step 3: Averaged bound.} First, it follows from \eqref{matrix dyson eq} and \eqref{eq D def} or equivalently from $\widetilde \cJ[G-M,D]=0$ that $G-M$ satisfies the following quadratic relation  
\[ G -M = (1-\cC_M \SS)^{-1} \big[-MD + M \SS[G-M](G-M)\big] \]
and therefore 
\[ \norm{\braket{B(G-M)}}_p \le \norm{\braket{B (1-\cC_M\SS)^{-1} [MD]}}_p + \norm{\braket{B(1-\cC_M\SS)^{-1}\big[ M\SS[G-M](G-M)\big]}}_p.\]
By geometric expansion, as in \eqref{geometric expansion of 1-CMS}, it follows that 
\begin{align*}
\norm{(1-\cC_{M}\SS)^{-1}}_{\norm{\cdot}\to\norm{\cdot}} \le 1 + \norm{M}^2 \tnorm{\SS}+ \norm{M}^4 \tnorm{\SS}^2\norm{(1-\cC_M\SS)^{-1}}_{\text{hs}\to\text{hs}}
\end{align*}
and thus that $\norm{\big((1-\cC_M S)^{-1}\big)^\ast[B^\ast]} \le_\epsilon N^\epsilon \norm{B}$ by \eqref{eq reference key info}. Using \eqref{av bound D eq}, where $\big((1-\cC_M S)^{-1}\big)^\ast[B^\ast]$ plays the role of $B$, and writing $\norm{\Im G}_q \le \norm{\Im M}+\norm{G-M}_q$ and using \eqref{First bound on G-M} we can conclude that
\begin{align}\label{First av bound G-M}
\norm{\braket{B(G-M)}}_p \le_{p,\epsilon,\gamma}\frac{\norm{B}N^\epsilon}{\braket{z}}\left[ \braket{z}\frac{\norm{\Im M}}{N\eta }+ \sqrt{\frac{\norm{\Im M}}{N\eta}}\frac{1}{N\eta}+\frac{1}{(N\eta)^2}\right] 
\end{align}
from Lemma \ref{S[R]T bound}. Now \eqref{av bulk} follows directly from \eqref{First av bound G-M} and \eqref{norm Im M bound}.

The proof of Theorem \ref{isotropic local law} is now complete. For the proof of Theorem \ref{isotropic local law away} the first three steps are identical except that we only work in the resticted domains $\DD_\gamma^\delta\cap\DD_{\text{out}}^\delta$. Due to \eqref{eq reference key info} and \eqref{G-M ast bound}, it then follows that in $\Dout^\delta$ the only place where the above proof used Assumption \ref{assumption flatness} is \eqref{norm Im M bound}. In the absence of Assumption \ref{assumption flatness} we replace \eqref{norm Im M bound} by the bound $\norm{\Im M}\le \eta \dist(z,\supp\mu)^{-2}$ from Proposition \ref{prop stability reference} in \eqref{First bound on G-M} and \eqref{First av bound G-M}, which only adds another negligible $N^\epsilon$ factor. This proves 
\begin{align}\label{outside first bound}
\norm{G-M}_p \le_{p,\epsilon} \frac{N^{\epsilon}}{\braket{z}}\bigg(\sqrt{\frac{1}{N}}+\frac{1}{\braket{z}}\frac{1}{N \eta}\bigg), \quad\norm{\braket{B(G-M)}}_p \le_{p,\epsilon,\gamma}\frac{\norm{B}N^\epsilon}{\braket{z}}\left[ \frac{1}{N }+ \sqrt{\frac{1}{N}}\frac{1}{N\eta}+\frac{1}{(N\eta)^2}\right]
\end{align} in the restricted domain $\DD_\gamma^\delta\cap\Dout^\delta$. We now need two additional steps to prove Theorem \ref{isotropic local law away} in all of $\Dout^\delta$.

\subsubsection*{Step 4: Absence of eigenvalues outside of the support.} For $B=1$ it follows from \eqref{outside first bound} and a spectral decomposition of $H$ that with very high probability in the sense of Corollary \ref{cor no eigenvalues outside} there are no eigenvalues outside the support of $\mu$. Indeed, if there is an eigenvalue $\lambda$ with $\dist(\lambda,\supp\mu)\ge N^{-\delta}$, then $\abs{\braket{G(\lambda+i\eta)}}\ge \abs{\braket{\Im G(\lambda+i\eta)}}\ge 1/N\eta$. From \eqref{outside first bound} with $\epsilon=1/4$ and $\gamma=1/2$ we have 
\begin{align*}
\P\left( \exists \lambda\text{ with }\dist(\lambda,\supp\mu)\ge N^{-\delta} \right) \le \P\left(\abs{\braket{G-M}}\ge c/N\eta \text{ in }\Dout^\delta\cap\DD_{1/2}^\delta \right) \lesssim \inf_{\eta\ge N^{-1/2}}\left(N^\epsilon\left[\eta+\frac{1}{\sqrt N}+\frac{1}{N\eta}\right]\right)^p \lesssim N^{-p/4}.
\end{align*}
Now Corollary \ref{cor no eigenvalues outside} follows from the remark about the dependence of $\delta$ on $\epsilon$ in Theorem \ref{isotropic local law away}.

\subsubsection*{Step 5: Improved bounds outside of the support.} Now we fix $z$ such that $\dist(z, \supp \rho)\ge N^{-\delta}$ and $\eta \ge N^{-1+\gamma}$. Then we have $\norm{\Im G}\prec \eta\braket{z}^{-2}$ and  $\norm{G}\prec \braket{z}^{-1}$
and also $\norm{\Im G}_p \le_{p,\epsilon} N^{\epsilon} \eta \braket{z}^{-2}$ and $\norm{G}_p\le_{p,\epsilon} N^{\epsilon}\braket{z}^{-1}$ and we infer from Theorem \ref{theorem step} that
\[
\norm{D}_p\le_{p,\epsilon}\frac{N^{\epsilon}}{\braket{z}\sqrt{N}} \quad\text{and therefore}\quad \norm{D}_\ast\prec\frac{1}{\braket{z}\sqrt{N}}.
\]
Again using stability in the form of \eqref{G-M ast bound} we find
\[
\norm{G-M}_\ast \prec \frac{N^{1/2K}}{\braket{z}^2\sqrt{N}}
\]
and since $K$ was arbitrary we also have 
\[
\norm{G-M}_p \le_{p,\epsilon} \frac{N^{\epsilon}}{\braket{z}^2\sqrt{N}}.
\]
By Lipschitz-continuity of $G$ and $M$ with Lipschitz constant of order one we can extend the regime of validity of this bound from $\eta \ge N^{-1+\gamma}$ to $\eta \ge 0$ to conclude \eqref{iso outside}. The improvement on the averaged law outside of the support of the $\rho$ then follows immediately from the improved isotropic law and the fact that with very high probability there are no eigenvalues outside of the support of $\rho$. 
\end{proof}

\section{Delocalization, rigidity and universality} \label{section deloc rig univ}
In this section we infer eigenvector delocalization, eigenvalue rigidity and universality in the bulk from the local law in Theorem \ref{isotropic local law}. These proofs are largely independent of the correlation structure of
the random matrix, so arguments that have been developed
for Wigner matrices over the last few years can be applied with
minimal modifications. Especially the \emph{three step strategy}
for proving bulk universality  (see \cite{MR2917064} for a short summary) has been streamlined recently \cite{MR3729630,landon2016fixed,MR3687212} so that the only model-dependent input is the local law.
The small modifications required for the correlated setup
have been presented in detail in \cite{1604.08188} and we will not repeat them.
Here we only explain why the proofs in \cite{1604.08188} work under
the more general conditions imposed in the current paper.
In fact, the proof of the eigenvector delocalization and eigenvector rigidity from \cite{1604.08188}
holds \emph{verbatim} in the current setup as well. The proof of the
bulk universality in \cite{1604.08188} used that the correlation length
was $N^{\epsilon}$ at a technical step that can be easily modified
for our weaker assumptions. In the following we will highlight which arguments of \cite{1604.08188} have to be modified in the current, more general, setup.
\begin{proof}[Proof of Corollary \ref{cor delocalization} on bulk eigenvector delocalization]
As usual, delocalization of eigenvectors corresponding to eigenvalues in the bulk is an immediate corollary of the local law since for the eigenvectors $\bm u_k=\big(u_k(i)\big)_{i\in J}$ and eigenvalues $\lambda_k$ of $H$ and $i\in J$ we find from the spectral decomposition
\[ C\gtrsim \Im G_{ii}= \eta\sum_k \frac{\abs{u_k(i)}^2}{(E-\lambda_k)^2+\eta^2}\ge \frac{\abs{u_k(i)}^2}{\eta}\quad\text{for}\quad z=E+\ii\eta,\]
where the first inequality is meant in a high-probability sense and follows from the boundedness of $M$ and Theorem \ref{isotropic local law}, and the last inequality followed assuming that $E$ is $\eta$-close to $\lambda_k$.
\end{proof}
\begin{proof}[Proof of Corollary \ref{cor rigidity} on bulk eigenvalue rigidity]
Rigidity of bulk eigenvalues follows, verbatim as in \cite[Corollary 2.9]{1604.08188}, from the improved local law away from the spectrum and \cite[Lemma 5.1]{MR3719056}.
\end{proof}
\begin{proof}[Proof of Corollary \ref{cor universality} on bulk universality]
Bulk universality follows from the \emph{three step strategy}, out of which only the third step requires a minor modification, compared to \cite{1604.08188}. Since in \cite{1604.08188} arbitrarily high polynomial decay outside of $N^\epsilon$ neighbourhoods was assumed, we have to replace to three term Taylor expansion in \cite[Lemma 7.5]{1604.08188} by an $2/\mu$-term cumulant expansion to accommodate for neighbourhoods of sizes $N^{1/2-\mu}$. 

The key input for the universality proof through Dyson Brownian motion is the Ornstein Uhlenbeck (OU) process, which creates a family $H(t)$ of interpolating matrices between the original matrix $H=H(0)$ and a matrix with sizeable Gaussian component, for which universality is known from the second step of the three step strategy. The OU process is defined via 
\begin{align}
\diff H(t) = -\frac{1}{2}(H(t)-A)\diff t + \Sigma^{1/2}[\diff B(t)],\qquad\text{where}\quad \Sigma[R]\defeq \E\braket{W^\ast R}W,  \label{OU}
\end{align}
where $B(t)$ is a matrix of independent (real, or complex according to the symmetry class of $H$) Brownian motions. It is designed in a way which preserves mean and covariances along the flow, i.e., $H(t)=A+N^{-1/2}W(t)$ and it is easy to check that $\E W(t)=0$ and $\Cov{w_\alpha(t),w_\beta(t)}=\Cov{w_\alpha(0),w_\beta(0)}$, where $W(t)=(w_\alpha(t))_{\alpha\in I}$. Furthermore, Assumptions \ref{assumption correlations}, \ref{assumption neighbourhood decay} hold also, uniformly in $t$, for $W(t)$. Indeed, adding an independent Gaussian vector $\bm g=(g_{\alpha_1},\dots,g_{\alpha_k})$ to $(w_{\alpha_1},\dots,w_{\alpha_k})$ leaves the cumulant invariant by additivity
\[ \kappa(w_{\alpha_1}+g_{\alpha_1},\dots,w_{\alpha_k}+g_{\alpha_k}) = \kappa(w_{\alpha_1},\dots,w_{\alpha_k})+ \kappa(g_{\alpha_1},\dots,g_{\alpha_k})\] 
and the fact that cumulants of Gaussian vectors vanish for $k\ge 3$ (for $k\ge 2$ we already noticed that, by design, the expectation and the covariance is invariant under $t$). We now estimate 
\[ \E f(N^{-1/2} W(t))-\E f(N^{-1/2} W(0)) \]
for smooth functions $f$. For notational purposes we set $v_\alpha(t)=N^{-1/2}w_{\alpha}(t)$ and $V(t)=N^{-1/2} W(t)$ and will often suppress the $t$-dependence. It follows from Ito's formula that 
\begin{align*}
2\frac{\diff}{\diff t}\E f(V) = - \E\sum_\alpha v_\alpha (\partial_\alpha f)(V) + \sum_{\alpha,\beta} \Cov{v_\alpha,v_\beta} \E (\partial_{\alpha}\partial_{\beta} f)(V).
\end{align*}
We now apply Proposition \ref{cum exp prop} to the first term and obtain
\begin{align*}
2\frac{\diff}{\diff t} \E f &= - \sum_{2\le m<R}\sum_{\alpha} \sum_{\bm \beta \in \NN^m} \frac{\kappa(v_\alpha,v_{\bm\beta})}{m!} (\E \partial_\alpha\partial_{\bm \beta} f)-\sum_{ m<R}\sum_{\alpha} \sum_{\bm \beta \in \NN^m} \E\frac{K(v_\alpha;v_{\bm\beta})-\kappa(v_\alpha,v_{\bm\beta})}{m!} \partial_\alpha\partial_{\bm \beta} f\big\rvert_{W_\NN=0} \\
&\qquad- \sum_{\alpha} \Omega(\partial_\alpha f,\alpha,\NN) +\sum_{\alpha}\sum_{\beta\in I\setminus\NN} \kappa(v_\alpha,v_\beta) \E \,\partial_{\alpha}\partial_{\beta} f,
\end{align*}
where we used a cancellation for the $m=1$ term in $\beta\in\NN$ and the fact that $\kappa(v_\alpha)=\E v_\alpha=0$ for the $m=0$ term. We now estimate the four terms separately. The sum in the last term is of size $N^4$, the derivative contributes an $N^{-1}$ and the covariance is assumed to be $N^{-3}$ small, i.e., the last term is of order $1$. The first term for fixed $m$ is of size $\tnorm{\kappa}^{\text{av}}N^{2-(m+1)/2}$ and therefore altogether of size $\tnorm{\kappa}^{\text{av}}\sqrt N$. Estimating the sums by their size, and the derivative by its prefactor $N^{-(R+1)/2}$, we find from \eqref{hoelder error bound} that the third term is of size 
\[ N^2 \abs{\NN}^R N^{-(R+1)/2}\le N^{3/2-\mu R},\]
which can be made smaller than $\sqrt N$ by choosing $R=2/\mu$. Finally, the second term is naively of size $N^{3/2}$, but using \eqref{two groups pre cum}, the  security layers and the pigeon-hole 
principle as in \eqref{K-kappa i j} or  in \eqref{layers}, this can be improved to $N^{-3/2}$. We can conclude that 
\begin{align}\abs{\E\frac{\diff}{\diff t}f(V(t))}\lesssim \sqrt N \quad\text{and therefore}\quad \abs{\E f(V(t))-\E f(V(0))}\lesssim t \sqrt N. \end{align}
The remaining argument of \cite[Section 7.2]{1604.08188} can be, assuming fullness as in Assumption \ref{assumption fullness}, followed verbatim to conclude bulk universality.
\end{proof}

\appendix
\section{Cumulants}\label{appendix cumulants}
In this section we provide some results on cumulants which we refer to in the main part of the proof. The section largely follows the approach of \cite{MR775392,MR725217}, but our application requires a more quantitative version of the independence property exhibited by cumulants, which we work out here.

Cumulants $\kappa_{\bm m}$ of a random vector $\bm w =(w_1,\dots, w_l)$ are traditionally defined as the coefficients of log-characteristic function
\[ \log \E e^{i \bm t \cdot \bm w } = \sum_{\bm m} \kappa_{\bm m} \frac{(i\bm t)^{\bm m}}{\bm m!},\]
while the (mixed) moments of $\bm w$ are the coefficients of the characteristic function
\[ \E e^{i \bm t\cdot \bm w} = \sum_{\bm m} (\E \bm w^{\bm m}) \frac{(i\bm t)^{\bm m}}{\bm m!}, \]
where $\sum_{\bm m}$ is the sum over all multi-indices $\bm m=(m_1,\dots,m_l)$. Thus
\begin{align}\label{cum moment gen fcts}
\exp\left( \sum_{\bm m} \kappa_{\bm m} \frac{(i\bm t)^{\bm m}}{\bm m!} \right) = \sum_{\bm m} (\E \bm w^{\bm m}) \frac{(i\bm t)^{\bm m}}{\bm m!}.
\end{align}
It is easy to check that for a set $A\subset [l]$ the coefficient of $\prod_{a\in A}t_a$ in \eqref{cum moment gen fcts} is given by 
\begin{align*}
\E \Pi \underline w_A=\Big(\prod_{a\in A}\partial_{t_a}\Big) \exp\left( \sum_{\bm m} \kappa_{\bm m} \frac{\bm t^{\bm m}}{\bm m!} \right) \bigg\rvert_{\bm t=0} = \sum_{\mathcal P\vdash A} \kappa^{\mathcal P},
\end{align*}
where $\cP\vdash A$ indicates the summation over all partitions of the (multi)set $A$, and where for partitions $\mathcal P=\{\mathcal P_1,\dots,\mathcal P_b\}$ of $A$ we defined 
$\kappa^{\mathcal P}=\prod_{k=1}^b \kappa_{\chi(\mathcal P_k)}$ with $\chi(\mathcal P_k)$ being the characteristic multi-index of the set $\mathcal P_k$. Thus for a partition $\mathcal Q$ of $[l]$ it follows that 
\begin{align} \label{moment to cumulant}
M^\cQ\defeq \prod_{\mathcal Q_i\in \mathcal Q} \E \Pi \underline w_{\mathcal Q_i} = \prod_{\mathcal Q_i\in \mathcal Q} \sum_{\mathcal P\vdash \mathcal Q_i} \kappa^{\mathcal P} = \sum_{\mathcal P\le \mathcal Q} \kappa^\cP,
\end{align}
where $\mathcal P\le \mathcal Q$ indicates that $\cP$ is a finer partition than $\cQ$.

Now we establish the inverse of the relation \eqref{moment to cumulant}, i.e., express cumulants in terms of products of moments. To do so, we notice that the set of partitions $\cP$ on $[l]$ (or, in fact, any finite set) is a partially ordered set with respect to the relation $\le$. It is, in fact, also a lattice, as any two partitions $\cP,\cQ$ have both a unique greatest lower bound $\cP \land \cQ$ and a unique least upper bound $\cP \lor \cQ$. One then defines the \emph{incidence algebra} as the algebra of scalar functions $f$ mapping intervals $[\cP,\cQ]=\Set{\cR|\cP\le\cR\le\cQ}$ to scalars $f(\cP,\cQ)$ equipped with point-wise addition and scalar multiplication and the product $\ast$
\[(f\ast g)(\cP,\cQ)=\sum_{\cP\le \cR\le \cQ} f(\cP,\cR)g(\cR,\cQ).\]
There are three special elements in the incidence algebra; the $\delta$ function mapping $[\cP,\cQ]$ to $\delta(\cP,\cQ)=1$ if $\cP=\cQ$ and $\delta(\cP,\cQ)=0$ otherwise, the $\zeta$ function mapping all intervals $[\cP,\cQ]$ to $\zeta(\cP,\cQ)=1$, and finally the M\"obius function defined inductively via
\[\mu(\cP,\cQ) = \begin{cases}
1, &\text{if }\cP=\cQ,\\
-\sum_{\cP\le \cR<\cQ} \mu(\cP,\cR),&\text{if }\cP < \cQ.
\end{cases} \]
The $\delta$ function is the unit element of the incidence algebra. It is well known (and easy to check) that the multiplicative inverse of the zeta function is the M\"obius function, and vice versa, i.e., that $\mu\ast\zeta=\zeta\ast\mu=\delta$. Thus it follows that for any functions $F$ and $G$ on the partitions, we have 
\[ F(\cP)= \sum_{\cQ\le \cP} G(\cQ) \qquad\text{if and only if}\qquad G(\cQ)=\sum_{\cP\le \cQ} \mu(\cP,\cQ)F(\cP).  \]
Applying this equivalence to \eqref{moment to cumulant} yields
\begin{align}\label{moments to cumulants mu} \kappa^{\cP} = \sum_{\cQ\le \cP} \mu(\cQ,\cP) M^\cQ\end{align}
and thus it only remains to identify $\mu$. One can check that for $\cP\le\cQ$, $\mu(\cP,\cQ)$ is given by
\[\mu(\cP,\cQ)=(-1)^{n-r}0!^{r_1}1!^{r_2}\dots (n-1)!^{r_n},\]
where $n$ is the number of blocks of $\cP$, $r$ is the number of blocks of $\cQ$ and $r_i$ is the number of blocks of $\cQ$ which contain exactly $i$ blocks of $\cP$. For the particular choice of the trivial partition $\{[l]\}$ of $[l]$ it follows that 
\begin{align}\label{cumulant to moment}
\kappa(w_1,\dots,w_l)\defeq\kappa_{(1,\dots,1)}= \kappa^{\{[l]\}} = \sum_{\cP} (-1)^{\abs{\cP}-1}(\abs{\cP}-1)! M^\cP =  \sum_{\cP} (-1)^{\abs{\cP}-1}(\abs{\cP}-1)! \prod_{\cP_i\in \mathcal \cP} \E \Pi \underline w_{\cP_i},
\end{align}
providing an alternative (purely combinatorial) definition of cumulants.

\begin{lemma}\label{mixing cumulant decay lemma}
If for a partition of the index set $[n]=A\sqcup B$ with $\abs{A},\abs{B}>0$ the random variables $\underline w_A$ and $\underline w_B$ are independent, then $\kappa(\underline w_{[n]})=\kappa(\underline w_A,\underline w_B)=0$. If, instead of independence, we merely assume that \begin{align} \Cov{ f(w_i\,\mid\,i\in A) , g(w_j\,\mid\,j\in B) }\le \epsilon \norm{ f}_2 \norm{g}_2 \label{cov f g bound}\end{align}
for all $f,g$, and that the random variables $w_i$ have finite $2n$-th moments $\max_i \E \abs{w_i}^{2n}\le \mu_{2n}$, then we still have
\begin{align}\label{eq kappa bound cov}
\abs[1]{\kappa(\underline w_{[n]})}\le \epsilon\, C(n,\mu_{2n}).
\end{align}
\end{lemma}
\begin{proof}
We first recall the well known proof, based on the relations \eqref{moment to cumulant}--\eqref{moments to cumulants mu}, that the cumulant of independent $\underline w_A$, $\underline w_B$ vanishes. Let $\cP$ be a partition on $[n]$, $\cQ$ a partition on $A$ and $\cR$ a partition on $B$. $\cP$ naturally induces partitions $\cP\cap A$ and $\cP \cap B$ on $A$ and $B$; conversely $\cQ$ and $\cR$ naturally induce a partition $\cQ\cup \cR$ on $[n]$. We observe that $\cQ \le \cP\cap A$ and $\cR \le \cP\cap B$ if and only if $\cQ\cup\cR \le \cP$. We then compute
\begin{align}\nonumber
\kappa(\underline w_{[n]})&=\sum_{\cP} \mu(\cP,\{[n]\}) M^\cP = \sum_{\cP} \mu(\cP,\{[n]\}) M^{\cP\cap A} M^{\cP\cap B} = \sum_{\cP} \mu(\cP,\{[n]\}) \bigg(\sum_{\cQ\le \cP\cap A} \kappa^\cQ \bigg)\bigg(\sum_{\cR\le \cP\cap B} \kappa^\cR \bigg)\\
&= \sum_{\cQ\vdash A}\sum_{\cR\vdash B}\sum_{\cP\vdash[n]} \zeta(\cQ\cup\cR,\cP)\mu(\cP,\{[n]\})\kappa^\cQ\kappa^\cR = \sum_{\cQ\vdash A}\sum_{\cR\vdash B} \delta(\cQ\cup\cR,\{[n]\})\kappa^\cQ \kappa^\cR=0,\label{cumulant decay indep}
\end{align}
where the first equality followed from \eqref{moments to cumulants mu}, the second equality from independence, the third equality from \eqref{moment to cumulant}, the fourth equality from the previous observation, the fifth equality from $\delta=\zeta\ast \mu$ and the ultimate equality from the fact that the trivial partition cannot be decomposed into two partitions on smaller sets, using that $\abs{A},\abs{B}>0$.

If $\underline w_A$ and $\underline w_B$ are not independent but merely \eqref{cov f g bound} holds, then there is an additional covariance term in the second step in the above equation. We write
\begin{align} M^\cP = \prod_{\cP_i\in \cP } \E \Pi \underline w_{\cP_i} = \prod_{\cP_i\in \cP } \Big[(\E \Pi \underline w_{\cP_i\cap A})(\E \Pi \underline w_{\cP_i\cap B})+\Cov{\Pi \underline w_{\cP_i\cap A},\Pi \underline w_{\cP_i\cap B}}\Big],\label{eq moment product cov}\end{align}
and thus the claim follows from \eqref{cov f g bound}.
\end{proof}

\section{Precumulants and Wick polynomials}\label{wick poly appendix}
The precumulants defined in Section \ref{sec cum exp} are structurally similar to the well known \emph{Wick polynomials} (which are also known as \emph{Appell polynomials}). We first recall some basic definitions and facts about Wick polynomials from \cite{MR899982chap}. For a random vector $\bm X$ of length $\abs{\bm X}$ we can define the Wick polynomial $\wick{\bm X}$ as the derivative \[\wick{\bm X}\defeq\partial_{t_1}\dots\partial_{t_{\abs{\bm X}}} \frac{e^{\bm t\cdot \bm X}}{\E e^{\bm t\cdot \bm X}}\Big\rvert_{\bm t=0}.\]
Alternatively, we can define $\wick{\bm X}$ combinatorially as 
\begin{subequations}
\begin{align}
\wick{\bm X} = \sum_{\bm X'\subset \bm X} (\Pi \bm X') \sum_{\cP\vdash \bm X\setminus\bm X'} (-1)^{\abs{\cP}} \prod_{\cP_i\in\cP} \kappa(\cP_i).
\end{align}
or indirectly via
\begin{align}\label{eq indir wick def}
\Pi \bm X = \sum_{\bm X'\subset \bm X} \wick{\bm X'} \big(\E \Pi (\bm X\setminus \bm X')\big).
\end{align}
\end{subequations}
One useful property of Wick polynomials is that for any random variable $Y$ we have
\begin{align}\label{wick decoupling}
\E Y \wick{\bm X_1\sqcup \bm X_2} = 0\qquad \text{whenever}\quad \bm X_1\quad\text{is independent of}\quad \{\bm X_2,Y\} 
\end{align}
and $\bm X_1$ is not empty. Eq.~\eqref{wick decoupling} follows, for example, immediately from the analytical definition since
\begin{align*}
\E Y \wick{\bm X_1\sqcup \bm X_2} = \partial_{\bm t} \frac{\E Y e^{\bm t_1\cdot \bm X_1+\bm t_2\cdot \bm X_2}}{\E e^{\bm t_1\cdot \bm X_1+\bm t_2\cdot \bm X_2}}\Big\rvert_{\bm t=0}= \partial_{\bm t} \frac{\E Y e^{\bm t_2\cdot \bm X_2}}{\E e^{\bm t_2\cdot \bm X_2}}\Big\rvert_{\bm t=0} 
\end{align*} 
by independence and the remaining derivative vanishes as the function is constant with respect to $\bm t_1$.

Our pre-cumulants $K(X;\bm Y)$ and their centered versions $K(X;\bm Y)-\kappa(X,\bm Y)$ are inherently non-symmetric functions due to the special role of $X$. After symmetrization, however, we can express them through Wick polynomials as
\begin{align}\label{pre cum Wick rel}
\sum_{X\in \bm X} \big[K(X; \bm X\setminus \{X\})-\kappa(\bm X)\big]= \abs{\bm X} \Pi \bm X - \sum_{\bm X'\subset \bm X} \abs{\bm X'}\big(\E \Pi \bm X'\big) \wick{\bm X\setminus \bm X'}.
\end{align} 
In order to prove \eqref{pre cum Wick rel} we start from \eqref{pre cum alt def} and compute 
\begin{align*}
\sum_{X\in \bm X} \big[K(X; \bm X\setminus \{X\})-\kappa(\bm X)\big] &= \abs{\bm X} \Pi \bm X - \sum_{\bm X'\subset \bm X}\abs{\bm X\setminus \bm X'} (\Pi \bm X') \kappa(\bm X\setminus\bm X')\\
& = \abs{\bm X} \Pi \bm X - \sum_{\bm X''\subset \bm X'\subset \bm X}\abs{\bm X\setminus \bm X'}\wick{\bm X''} \big(\E \Pi(\bm X'\setminus \bm X'')\big)  \kappa(\bm X\setminus\bm X'),
\end{align*}
where the second inequality followed from \eqref{eq indir wick def}. We now relabel the summation indices to obtain 
\begin{align*}
\sum_{X\in \bm X} \big[K(X; \bm X\setminus \{X\})-\kappa(\bm X)\big] &= \abs{\bm X} \Pi \bm X - \sum_{\bm X''\subset \bm X'\subset \bm X}\abs{\bm X''}\wick{\bm X\setminus\bm X'} \big(\E \Pi(\bm X'\setminus \bm X'')\big)  \kappa(\bm X''),
\end{align*}
from which \eqref{pre cum Wick rel} follows using the well known cumulant identity 
\begin{align}\label{cum id to prove}
\abs{\bm X'}\E \Pi \bm X'=\sum_{\bm X''\subset \bm X'}\abs{\bm X''} \big(\E \Pi(\bm X'\setminus \bm X'')\big)  \kappa(\bm X'').
\end{align}
In order to prove \eqref{cum id to prove}, we use \eqref{moment to cumulant} on the rhs.~to obtain
\begin{align*}
\sum_{\bm X''\subset \bm X'}\abs{\bm X''} \big(\E \Pi(\bm X'\setminus \bm X'')\big)  \kappa(\bm X'') = \sum_{\bm X''\subset \bm X'}\abs{\bm X''} \kappa(\bm X'')\sum_{\cP\vdash \bm X'\setminus \bm X''} \kappa^{\cP} = \sum_{\cP \vdash \bm X'} \kappa^{\cP} \sum_{\substack{\bm X''\subset \bm X'\\\bm X''\in\cP}} \abs{\bm X''}=\abs{\bm X'}\sum_{\cP \vdash \bm X'} \kappa^{\cP},
\end{align*}
from which \eqref{cum id to prove} follows by another application of \eqref{moment to cumulant}.

Finally we remark that a quantitative variant of \eqref{wick decoupling} for the pre-cumulants was centrally used in our proof in Section \ref{sec neighbourhood errors}. Qualitatively the analogue of \eqref{wick decoupling} for pre-cumulants reads
\begin{align}
\E Y \big[K(X;\bm X_1, \bm X_2)-\kappa(X,\bm X_1,\bm X_2)\big] =0 \qquad \text{whenever}\quad \{X,\bm X_1\}\quad\text{is independent of}\quad\{\bm X_2,Y\}
\end{align}
and $\bm X_2$ is non-empty. Indeed, from the pre-cumulant decoupling identity \eqref{two groups pre cum} we have that 
\begin{align*}
\E Y \big[K(X;\bm X_1, \bm X_2)-\kappa(X,\bm X_1,\bm X_2)\big] = \E Y (\Pi \bm X_2) \big[K(X;\bm X_1)-\kappa(X,\bm X_1)\big] - \sum_{\substack{\bm X_1'\subset\bm X_1\\\bm X_2'\subsetneq\bm X_2}} \E Y(\Pi \bm X_1')(\Pi \bm X_2') \kappa(X,\bm X_1\setminus\bm X_1',\bm X_2\setminus\bm X_2')
\end{align*}
and the first term vanishes due to independence and \eqref{two groups pre cum}, and the second term vanishes due to Lemma \ref{mixing cumulant decay lemma} because the argument of $\kappa$ splits into two independent groups.

\section{Modifications for complex Hermitian \texorpdfstring{$W$}{W}}\label{complex W}
Our main arguments were carried out for the real symmetric case. We now explain how to modify our proofs 
if $W$ is complex Hermitian. A quick inspection of the proofs shows that the only modification concerns
Proposition \ref{cum exp prop} where we have to replace the cumulant expansion by its complex variant. We reduce the problem to the real case by considering real and imaginary parts of each variable separately. Another option would have been to consider $w$ and $\overline w$ independent variables, but our choice seems to require the least modifications. In order to compute $\E w_{i_0} f(\bm w)$ for a random vector $\bm w\in\C^\cI$, $w_{i_0}\in\C$ and a function $f\colon\C^\cI\to\C$, we can define $\widetilde f\colon\R^{\cI\sqcup\cI}\to\C$ by mapping $(\bm w^{\Re},\bm w^{\Im})\mapsto f(\bm w^{\Re}+i\bm w^{\Im})$, where the new index set $\cI\sqcup\cI$ should be understood as two copies of $\cI$ in the sense that $\cI\sqcup\cI=\set{(i,\Re),(i,\Im)|i\in\cI}$. If we want to expand $w_{i_0} f(\bm w)$ in the variables of some fixed index set $\NN\subset\cI$, we separately apply Proposition \ref{cum exp prop} to $\E \widetilde w_{(i_0,\Re)} \widetilde f(\widetilde{\bm w})$ and $\E \widetilde w_{(i_0,\Im)} \widetilde f(\widetilde{\bm w})$ in $\NN\sqcup\NN$, where $\widetilde{\bm w}=(\Re \bm w,\Im \bm w)$ and $\widetilde w_{(i,\Re)}=\Re w_i$, $\widetilde w_{(i,\Im)} = \Im w_i$. It follows that
\begin{align}\label{complex cum exp 1st step}
\E w_{i_0} \widetilde f(\widetilde{\bm w}) = \sum_{l>0}\sum_{\widetilde{\bm i}\in (\NN\sqcup\NN)^l} \frac{\kappa( \widetilde w_{(i_0,\Re)}, \widetilde{\bm w_{\widetilde i}} )+ \kappa(\ii \widetilde w_{(i_0,\Im)}, \widetilde{\bm w_{\widetilde i}} )}{l!} \partial_{\widetilde{\bm i}} (\E \widetilde f) + \widetilde\Omega^1 + \widetilde\Omega^2, 
\end{align}
where the error terms are those from two applications of \eqref{cum expansion statement}. We note that we can make sense of $\kappa$ with complex arguments directly through Definition \eqref{cumulant to moment}. We now want to go back to a summation over our initial index set $\NN$ and therefore regroup the terms in \eqref{complex cum exp 1st step} according to the first indices of $\widetilde{\bm i}$. To formulate the result compactly we introduce the tensors
\begin{align*}
\widetilde\kappa(w_{i_0},\dots,w_{i_l}) \defeq \kappa\Big[ \begin{pmatrix}\Re w_{i_0}\\\ii\Im w_{i_0}\end{pmatrix}\otimes \dots\otimes \begin{pmatrix}\Re w_{i_l}\\\ii\Im w_{i_l}\end{pmatrix} \Big]\in (\R\times \ii\R)^{\otimes(l+1)}\quad\text{and}\quad \widetilde\partial_{\bm i}  \defeq  \begin{pmatrix}\partial_{\Re w_{i_1}}\\\partial_{\Im w_{i_1}}\end{pmatrix}\otimes \dots\otimes \begin{pmatrix}\partial_{\Re w_{i_l}}\\\partial_{\Im w_{i_l}}\end{pmatrix},
\end{align*}
where the application of $\kappa$ is understood in an entrywise sense and the derivative tensor has dimension $(\C^2)^{\otimes l}$. By saying that $\kappa$ is understood in an entrywise sense, we mean, by slight abuse of notation that, for example,
\[ \kappa \left(  \begin{pmatrix}v_1\\v_2\end{pmatrix} \otimes \begin{pmatrix}w_1\\w_2\end{pmatrix}  \right) = \kappa \bigg( \sum_{i,j=1}^2  v_i w_j\, e_i\otimes e_j \bigg) \defeq \sum_{i,j=1}^2 \kappa(v_i, w_j)\, e_i \otimes e_j,\]
where $e_1,e_2$ is the standard basis of $\R\times \ii\R$. Due to the special nature of the index $i_0$ we see from \eqref{complex cum exp 1st step} that $\Re w_{i_0}$ and $\ii\Im w_{i_0}$ always occur in a sum of two and the rhs.~of \eqref{complex cum exp 1st step} can be expressed in terms of the partial trace $\Tr_1 \widetilde \kappa (w_{i_0},\dots,w_{i_l})\in (\R\times i\R)^{\otimes l}$ along the first dimension, which corresponds to $i_0$. Thus we can compactly write \eqref{complex cum exp 1st step} as
\begin{align}\label{eq complex cum exp 2nd step}
\E w_{i_0} f(\bm w) = \sum_{0\le l<R} \sum_{\bm i\in \NN^l} \frac{\braket{\Tr_1\widetilde\kappa(w_{i_0},\bm w_{\bm i}), \E(\widetilde\partial_{\bm i} f)}}{l!} + \widetilde\Omega^1 + \widetilde\Omega^2,
\end{align}
where the scalar product is taken between two tensors of size $2^l$. For example, the $l=1$ term from \eqref{eq complex cum exp 2nd step} reads 
\[ \sum_{i_1\in\NN} \left(\frac{\kappa(\Re w_{i_0},\Re w_{i_1})+\kappa(\ii\Im w_{i_0},\Re w_{i_1})}{1!}(\E \partial_{\Re w_{i_1}} f) + \frac{\kappa(\Re w_{i_0},\ii\Im w_{i_1})+\kappa(\ii\Im w_{i_0},\ii\Im w_{i_1})}{1!}(\E \partial_{\Im w_{i_1}} f) \right).\]
The rest of the argument in Section \ref{section step} can be carried out verbatim for any specific choice of distribution of $\Re$, $\Im$ to the entries of $\kappa$. We only have to replace the norms $\tnorm{\kappa}^{\text{av}}$ and $\tnorm{\kappa}^{\text{iso}}$ in Assumption \ref{assumption correlations} by applying them entrywise to $\widetilde \kappa$, i.e.,
\begin{subequations}\label{widetildekappa norm}
\begin{align}
\tnorm{\widetilde\kappa (w_{\alpha_1},\dots,w_{\alpha_k})}^{\text{av}} &\defeq \sum_{\mathfrak{X}_1,\dots,\mathfrak{X}_k\in\{\Re,\Im\}} \tnorm{\kappa(\mathfrak X_1 w_{\alpha_1},\dots,\mathfrak X_k w_{\alpha_k})}^{\text{av}},\\ 
\tnorm{\widetilde\kappa (w_{\alpha_1},\dots,w_{\alpha_k})}^{\text{iso}} &\defeq \sum_{\mathfrak{X}_1,\dots,\mathfrak{X}_k\in\{\Re,\Im\}} \tnorm{\kappa(\mathfrak X_1 w_{\alpha_1},\dots,\mathfrak X_k w_{\alpha_k})}^{\text{iso}}. 
\end{align}
\end{subequations}
\begin{genericthm*}[Hermitian $\kappa$-correlation decay]
We assume that for all $R\in\N$ and $\epsilon>0$
\[\tnorm{\widetilde\kappa}^{\text{av}}\le_{\epsilon,R} N^{\epsilon } \quad\text{and}\quad \tnorm{\widetilde\kappa}^{\text{iso}} \le_{\epsilon,R} N^{\epsilon }.\]
\end{genericthm*}
Since there are at most $2^R$ such choices this change has no impact on any of the claimed bounds which always implicitly allow for an $R$--dependent constant. 

\section{Proofs of auxiliary results} 
\begin{lemma}[Quadratic Implicit Function Theorem]\label{impl fct lemma}
Let $\norm{\cdot}$ be a norm on $\C^d$, $A,B\in \C^d$  and $Q\colon\C^d\times \C^d\to\C^d$ a bounded $\C^d$-valued quadratic form, i.e.,
\[
\norm{Q} = \sup_{x,y}\frac{\norm{Q(x,y)}}{\norm{x}\norm{y}}< \infty.
\]
Suppose that $A$ is invertible. Then for $\epsilon_2\defeq \big[2\norm{A^{-1}}\norm{Q}\big]^{-1}$ and $\epsilon_1\defeq \epsilon_2\big[2\norm{A^{-1}}\norm{B}\big]^{-1}$ there is a unique function $X\colon B_{\epsilon_1} \to B_{\epsilon_2}$ such that
\[
AX(d)+Q(X(d),X(d))=Bd,
\]
where $B_{\epsilon}$ denotes the open $\epsilon$-ball around $0$. Moreover, the function $X$ is analytic and satisfies
\[
\norm{X(d_1)-X(d_2)} \le 2 \norm{A^{-1}}\norm{B}\norm{d_1-d_2}\quad\text{for all}\quad d_1, d_2 \in B_{\epsilon_1/2}.
\]
\end{lemma}
\begin{proof}
A simple application of the Banach fixed point theorem. 
\end{proof}

\begin{lemma}\label{S[R]T bound}
For random matrices $R,T$ and $p\ge 1$ it holds that $\norm{\SS[V]T}_p\le \tnorm{\SS} \norm{V}_{2p} \norm{T}_{2p}$.
\end{lemma}
\begin{proof}
Let $\kappa=\kappa_c+\kappa_d$ be an arbitrary partition, which induces a partition of $\SS$ since \[\SS[V]=\frac{1}{N}\sum_{\alpha_1,\alpha_2} \kappa(\alpha_1,\alpha_2) \Delta^{\alpha_1}V\Delta^{\alpha_2}.\] For vectors $\vx,\vy$ with $\norm{\vx},\norm{\vy}\le 1$ we compute 
\begin{align*} \norm{(\SS[V]T){\vx\vy}}_p &= \norm[2]{\frac{1}{N}\sum_{b_1,a_2,b_2} \kappa(\vx b_1,a_2b_2) V_{b_1a_2}T_{b_2\vy}}_p\le \norm[2]{\frac{1}{N}\sum_{b_1,b_2} V_{b_1\kappa_c(\vx b_1,\cdot b_2)} T_{b_2\vy}}_p + \norm[2]{\frac{1}{N}\sum_{b_1,a_2} R_{b_1a_2} T_{\kappa_d(\vx b_1,a_2\cdot)\vy}}_p\\
&\le \frac{\norm{V}_{2p}\norm{T}_{2p}}{N}\Big[ \sum_{b_1,b_2} \norm{\kappa_d(\vx b_1,\cdot b_2)} + \sum_{b_1,a_2} \norm{\kappa_c(\vx b_1,a_2\cdot)}\Big] \le \Big[\tnorm{\kappa_d}_d+\tnorm{\kappa_c}_c\Big] \norm{V}_{2p}\norm{T}_{2p}
\end{align*}
and the result follows from optimizing over the decompositions of $\kappa$ and recalling the definition \eqref{tnorm SS}.
\end{proof}
\begin{lemma}\label{G D triv bound lemma}
For any $t\in [0,1]$, $q\ge 1$,  $\epsilon>0$ 
 and multi-set $\underline\beta\subset I$ we have under Assumption \ref{assumption A} that
\begin{subequations}
\begin{align}\label{G triv bound}
\norm[1]{\partial_{\underline\beta} G \big\rvert_{\widehat W} }_q &\le_{\abs[0]{\underline\beta}} N^{-\abs[0]{\underline\beta}/2} 
 N^\epsilon \braket{z}^{-1}
\Big(1+\norm{G}_{2q(\abs[0]{\underline\beta}+1)}\Big)^{ (\abs[0]{\underline\beta}+1)/\mu} \\
\label{D triv bound}
\norm[1]{\partial_{\underline\beta} D \big\rvert_{\widehat W}}_q&\le_{\abs[0]{\underline\beta}} N^{-\abs[0]{\underline\beta}/2} (1+\tnorm{\SS}) N^\epsilon  \braket{z}^{-1} \Big(1+\norm{G}_{6q(\abs[0]{\underline\beta}+2)}\Big)^{ (\abs[0]{\underline\beta}+2)/\mu},
\end{align}
\end{subequations}
where $\widehat W_\alpha= t w_\alpha$ for $\alpha\in\NN$ and $\widehat W_\alpha=w_\alpha$ otherwise for a set $\NN\subset I$ of size $\abs{\NN}\le  N^{1/2-\mu}$.
\end{lemma}
\begin{proof}
We write $\underline\beta=\{\beta_1,\dots,\beta_n\}$ and its easy to see inductively that 
\begin{align}
\partial_{\underline\beta} G\big\rvert_{\widehat W} = \frac{(-1)^n}{N^{n/2}} \sum_{\sigma\in S_n} \widehat G\Delta^{\beta_{\sigma(1)}}\widehat G\Delta^{\beta_{\sigma(2)}} \widehat G\dots \widehat G \Delta^{\beta_{\sigma(n)}} \widehat G,
\end{align}
where $\widehat G=G(\widehat W)$. From the resolvent identity it follows that 
\[\widehat G-G =   \sum_{k=1}^{R-1} G\Big( \frac{W-\widehat W}{\sqrt{N}}G\Big)^k+
\widehat G\Big( \frac{W-\widehat W}{\sqrt{N}}G\Big)^R 
\]
and therefore by the trivial bound $\norm[0]{\widehat G}\le 1/\eta$ 
and Assumption \ref{assumption high moments} it follows that 
 for \(R:=\lceil 1/\mu\rceil\) and \(\eta\ge N^{-1}\)   we have
\begin{equation}\label{GG} \norm[1]{\widehat G-G}_q \le \sum_{k=1}^{R-1} \frac{\abs{\NN}^k \norm{G}^{k+1}_{(2k+1)q}\max_\alpha \norm{w_\alpha}_{(2k+1)q}^k}{N^{k/2}} + \frac{\abs{\NN}^R \norm{G}^{R}_{2Rq}\max_\alpha \norm{w_\alpha}_{(2R+1)q}^R}{N^{R/2}\eta} \le_q \norm{G}_{2Rq} (1+\norm{G}_{2Rq})^{R} .\end{equation}
Since $\norm{H}\le N^{\epsilon/2}$ with very high probability for sufficiently large $N$ it follows that $\norm{G}_{r}\le \norm{G}\lesssim N^{\epsilon/2}/\braket{z}$ for $\abs{z}\gg N^{\epsilon/2}$ which immediately implies \eqref{G triv bound}.

Similarly, \eqref{D triv bound} follows from the easily verifiable identities
\begin{align}
\partial_{\underline\beta} D\big\rvert_{\widehat W} = \frac{(-1)^n}{N^{n/2}} \sum_{\sigma\in S_n} \Big[ \widehat D \Delta^{\beta_{\sigma(1)}}\widehat G\dots\Delta^{\beta_{\sigma(n)}} \widehat G + \sum_{k=1}^n \SS[\widehat G\Delta^{\beta_{\sigma(1)}}\widehat G\dots \Delta^{\beta_{\sigma(k)}}\widehat G ]\widehat G\Delta^{\beta_{\sigma(k+1)}}\widehat G\dots \Delta^{\beta_{\sigma(n)}} \widehat G \Big]
\end{align} 
 and
\begin{equation}\label{D bound}
 \widehat D=D-\SS[G]G + (D-\SS[G]G)\frac{W-\widehat W}{\sqrt N} \widehat G + \frac{\widehat W-W}{\sqrt N} \widehat G + \SS[\widehat G] \widehat G
\end{equation} 
together with Lemma~\ref{S[R]T bound}, ~\eqref{G triv bound}, \eqref{GG} and 
\begin{align}\label{D bound2}\norm[0]{ D}_{r}\le_r N^\epsilon \braket{z}^{-1} (1+ \norm[0]{G}_{r}) +
 \norm[0]{\SS[G]G}_{r}. \end{align}
To see why \eqref{D bound2} holds we write 
 $D= HG-AG+\SS[G]G$ and use $\norm{AG}_r\lesssim \norm{G}_r $ while
 $\norm[0]{H G}_r = \norm{1+ zG}_r\lesssim 1+ N^{\epsilon/2} \norm{G}_r$
for $|z|\lesssim N^{\epsilon/2}$. For large $|z|\gg N^{\epsilon/2}$ we estimate 
 that $\norm[0]{H G}_r  \le \norm{ H  G} \le  N^{\epsilon} \braket{z}^{-1}$  since $\norm{ H}\le  N^{\epsilon/2} $.
\end{proof}

\printbibliography

\end{document}